\documentclass{amsart}

\usepackage[numbers]{natbib}
 
\usepackage{amsmath}
\usepackage{amsthm, thmtools}
\usepackage{cleveref}
\usepackage[geometry]{ifsym}
\usepackage{amsfonts}

\usepackage{algorithm}
\usepackage{algorithmic}
\usepackage{tabularx}
\usepackage{comment}
\usepackage{caption}

\usepackage{makecell}
\usepackage{url}

\usepackage{mathtools}
\usepackage[utf8]{inputenc}
\usepackage[T1]{fontenc}

\usepackage{graphicx}

\usepackage[caption=false,font=footnotesize]{subfig}
\usepackage{fullpage}
\usepackage{wrapfig}
\usepackage{epsfig}

\usepackage{array}

\newtheorem{thm}{Theorem}

\newtheorem{defn}{Definition}

\newtheorem{proposition}{Proposition}

\newtheorem{remark}{Remark}

\DeclareMathOperator*{\argmin}{\arg\!\min}

\DeclareMathOperator*{\vecmat}{\text{vec}}

\newcommand{\tensor}[1]{\mathbf{#1}}

\DeclareMathOperator*{\paramspaceclean}{\mathcal{M}}
\newcommand{\paramspacecleanT}{\mathcal{T}_x\mathcal{M}}

\DeclareMathOperator*{\htuckspaceclean}{\mathcal{H}}
\DeclareMathOperator*{\troot}{t_{\text{root}}}
\DeclareMathOperator*{\liegroup}{\mathcal{G}}
\DeclareMathOperator*{\fancyA}{\mathcal{A}}
\DeclareMathOperator*{\quotspace}{\mathcal{M}/\mathcal{G}}
\newcommand{\vertspace}{\mathcal{V}_x \paramspaceclean}
\newcommand{\horizspace}{\mathcal{H}_x \paramspaceclean}
\DeclareMathOperator*{\tr}{\text{tr}}
\DeclareMathOperator*{\qf}{\text{\rm{qf}}}

\newcommand{\fullspace}{\mathbb{R}^{n_1 \times n_2 \times ... \times n_d}}
\DeclareMathOperator*{\itrvar}{i}

\begin{document}

\title{ Optimization on the Hierarchical Tucker manifold - applications to tensor completion }
\author{Curt Da Silva$^{1}$ and Felix J. Herrmann$^{2}$\\$^1$ Department of Mathematics, University of British Columbia\\$^2$ Department of Earth and Ocean Sciences, University of British Columbia}
\begin{abstract}
  In this work, we develop an optimization framework for problems
  whose solutions are well-approximated by \emph{Hierarchical Tucker}
  (HT) tensors, an efficient structured tensor format based on
  recursive subspace factorizations. By exploiting the smooth manifold
  structure of these tensors, we construct standard optimization
  algorithms such as Steepest Descent and Conjugate Gradient for
  completing tensors from missing entries. Our algorithmic framework
  is fast and scalable to large problem sizes as we do not require
  SVDs on the ambient tensor space, as required by other
  methods. Moreover, we exploit the structure of the Gramian matrices
  associated with the HT format to regularize our problem, reducing overfitting for high
  subsampling ratios. We also find that the organization of the tensor
  can have a major impact on completion from realistic seismic
  acquisition geometries. These samplings are far from idealized
  randomized samplings that are usually considered in the literature but are realizable in practical scenarios. Using these
  algorithms, we successfully interpolate large-scale seismic data sets and demonstrate the competitive computational scaling of our algorithms as the problem sizes grow.
\end{abstract}
\maketitle
\section{Introduction}
The matrix completion problem is concerned with interpolating a $m
\times n$ matrix from a subset of its entries. The amount of recent
successes in developing solution techniques to this problem is a
result of assuming a \emph{low-rank} model on the 2-D signal of
interest and by considering subsampling schemes that \emph{increase}
the rank of the underlying matrix \cite{candes2009exact},
\cite{svtmc}, \cite{candes2010power}. The original signal is recovered
by promoting low-rank structures subject to data constraints.

Using a similar approach, we consider the problem of interpolating a
$d-$dimensional tensor from samples of its entries. That is, we aim to
solve,
\begin{align}
\label{eq:tensorcompletion}
\min_{\tensor{X} \in \htuckspaceclean} \frac{1}{2}\|P_{\Omega} \tensor{X} - b\|^2_2,
\end{align}
where $P_{\Omega}$ is a linear operator $P_{\Omega} : \fullspace \to
\mathbb{R}^{m}$, $b \in \mathbb{R}^{m}$ is our subsampled data satisfying $b = P_{\Omega} \tensor{X}^*$ for some ``solution'' tensor $\tensor{X}^*$ and $\htuckspaceclean$ is a \emph{specific} class of
low-rank tensors to be specified later. Under the assumption that $\tensor{X}^* $ is well approximated by an element in $\htuckspaceclean$, our goal is to recover $\tensor{X}^*$ by solving (\ref{eq:tensorcompletion}). For concreteness, we concern ourselves with the
case when $P_{\Omega}$ is a restriction operator, i.e.,
\begin{align*}
P_{\Omega} \tensor{X} = \tensor{X}_{i_1, i_2, \dots, i_d} \quad \text{ if } (i_1, i_2, \dots, i_d) \in \Omega,
\end{align*}
and $\Omega \subset [n_1] \times [n_2] \times \dots \times [ n_d]$ is
the so-called \emph{sampling set}, where $[n] = \{1, \dots, n\}$. In
the above equation, we suppose that $|\Omega| = m \ll n_1 n_2 \dots
n_d$, so that $P_{\Omega}$ is a subsampling operator.

Unlike the matrix case, there is no unique notion of rank for tensors,
as we shall see in \Cref{sec:prevwork}, and there are multiple tensor
formats that generalize a particular notion of \emph{separability}
from the matrix case---i.e, there is no unique extension of the SVD to 
tensors. Although each tensor format can lead to compressible
representations of their respective class of low-rank signals, the
truncation of a general signal to one of these formats requires access
to the \emph{fully} sampled tensor $\tensor{X}$ (or at the very least \emph{query}-based access to the tensor \cite{blackboxapprox-htuck}) in order to achieve reasonable
accuracy, owing to the use of truncated SVDs acting on various 
\emph{matricizations} of the tensor. As in matrix completion,
randomized missing entries change the behavior of the singular values
and vectors of these matricizations and hence of the final
approximation. Moreover, when the tensor of interest is actually a
discretized continuous signal, there can be a number of constraints,
physical or otherwise, that limit our ability to ideally sample it. For instance, in the seismic case, the tensor of interest is a
multi-dimensional wavefield in the earth's subsurface sampled at an
array of receivers located at the surface. In real-world seismic
experiments, budgetary constraints or environmental obstructions can
limit both the total amount of time available for data acquisition as
well as the number and placement of active sources and
receivers. Since seismic and other methods rely on having fully sampled data for drawing accurate inferences,
tensor completion is an important technique for a variety of scientific fields that acquire multidimensional data.

In this work, we consider the class of Hierarchical Tucker
(abbreviated HT) tensors, introduced in \cite{htuckinitial, htuckersvd}, as
our low-rank tensors of interest. The set of all such tensors is a
smooth, embedded \emph{submanifold} of $\fullspace$, first studied in
\cite{htuckgeom}, which we equip with a Riemannian metric. Using this
Riemannian structure, we can construct optimization algorithms in
order to solve (\ref{eq:tensorcompletion}) for $d$-dimensional tensors. We
will also study some of the effects of higher dimensional sampling and
extend ideas from compressive sensing and matrix completion to the HT
tensor case for our specific seismic examples.

%
\subsection{Previous Work}
\label{sec:prevwork}
To provide the reader with some context on tensor representations, let
us briefly detail some of the available structured tensor formats,
including tensor completion results, here (see
\cite{prevwork:structuredtensoroverview} and
\cite{prevwork:tensordecompositions} for a very comprehensive
overview). Here we let $N=\max_{i=1\cdots d}n_i$ be the maximum individual dimension size,  
$N^d := \prod_{i=1}^{d} n_i$ denote the dimension of the ambient space $\fullspace$, and, for each tensor format discussed, $K$ is the maximum of all of the rank parameters associated to that format.

The so-called Candecomp/Parafac (CP) decomposition is a very
straightforward application of the separation of variables
technique. Very much like the SVD of a matrix, one stipulates
that, for a function $f$ living on a tensor product space, one can
write
\begin{align*}
  f(x_1, x_2, \dots, x_d) \approx \sum_{i=1}^{K} f_i^{(1)}(x_1) f_i^{(2)}(x_2)\dots f_i^{(d)}(x_d).
\end{align*}
Thus its discretization can be written as
\begin{align*}
  \tensor{f} \approx \sum_{i=1}^{K} f_i^{(1)} \otimes f_i^{(2)} \otimes \dots \otimes f_i^{(d)}
\end{align*}
where $\otimes$ is the Kronecker product and $f_i^{(j)}\in
\mathbb{R}^{n_j}$ is the discretization of the one dimensional
function $f_i^{(j)}(x_j)$. In addition to its intuitive construction, the
CP decomposition of rank $K$ only requires $dNK$ parameters versus the
$N^d$ of the full tensor and tensor-tensor operations can be performed
efficiently on the underlying factors rather than the full tensors
themselves (see \cite{prevwork:tensorToolbox} for a comprehensive set
of MATLAB tools).

Unfortunately, despite the parsimoniousness of the CP construction,
the approximation of an arbitrary (full) tensor by CP tensors has both
theoretical and numerical difficulties. In particular, the set of all
CP tensors of rank at most $K$ is not closed, and thus the notion of a
best $K-$rank approximation is difficult to compute in many cases
\cite{prevwork:cp-approx-illposed}. Despite this shortcoming,
various authors have proposed iterative and non-iterative algorithms
in the CP format for approximating full tensors \cite{prevwork:tensordecompositions} as well as
interpolating tensors with missing data, such as the Alternating Least
Squares approach (a block Gauss-Seidel type method) proposed alongside
the CP format in \cite{prevwork:originalCP} and
\cite{prevwork:originalCP2}, with convergence analysis in
\cite{prevwork:cp-als-convergence}, and a nonlinear least-squares
optimization scheme in \cite{prevwork:TTB_CPOPT}. 

The CP format is a specific case of the more general \emph{Tucker
  format}, which aims to write a tensor $\mathbf{f}$ as a multilinear
product
\begin{align*}
\tensor{f} \approx U_1 \times_1 U_2 \times_2 \dots U_d \times_d \tensor{C}
\end{align*}
where $\tensor{C} \in \mathbb{R}^{k_1 \times k_2 \times
  ... \times k_d}$ is the so-called \emph{core tensor} and the
matrices $U_j \in \mathbb{R}^{n_j \times k_j}$, $j = 1, \dots, d$ are
the \emph{factors} of the decomposition. Here we use the notation of
the multilinear product, that is, $U_i \times_i \tensor{C}$ indicates that $\tensor{C}$
is multiplied by $U_i$ in dimension $i$, e.g., see \cite{prevwork:cp-approx-illposed,prevwork:tensordecompositions}. We will elaborate on this construction in \Cref{sec:multilinearproduct}. The
CP format follows from this formulation when the core tensor is
\emph{diagonal}, i.e., $\tensor{C}_{i_1,i_2,i_3,\dots, i_d} = \tensor{C}_{i_1,i_1,\dots,
  i_1} \delta_{i_1,i_2,\dots,i_d}$, where $\delta_{i_1,i_2,\dots,i_d} = 1$ when $i_1 = i_2 = \dots = i_d$ and $0$ otherwise. 

The Tucker format enjoys many benefits in terms of approximation
properties over its CP counterpart. Namely, the set of all Tucker
tensors of at most multilinear rank $\mathbf{k} = (k_1, k_2, \dots, k_d)$ is closed
and as a result every tensor $\tensor{f}$ has a best at most multilinear rank-$\mathbf{k}$ Tucker
approximation. A near-optimal approximation can be computed
efficiently by means of the Higher Order SVD
\cite{prevwork:multilinearsvd}. For the tensor completion problem, the
authors in \cite{prevwork:tuckercompletion} consider the problem of
recovering a Tucker tensor with missing entries using the
Douglas-Rachford splitting technique, which decouples interpolation
and regularization by nuclear norm penalization of different
matricizations of the tensor into subproblems that are then solved via
a particular proximal mapping. An application of this approach to
seismic data is detailed in \cite{prevwork:kreimer2012tensor} for the interpolation problem and \cite{prevwork:kreimer2012tensordnoise} for denoising. Depending on the size and ranks of the tensor to be recovered, there are
theoretical and numerical indications that this approach is no better
than penalizing the nuclear norm in a single matricization (see
\cite{oymak2012simultaneously} for a theoretical justification in the Gaussian measurement case, as well as \cite{tensorvsmatrix} for an experimental demonstration of this effect). Some preliminary results on theoretical guarantees for recovering low-rank Tucker tensors from subsampled measurements are given in \cite{huang2014provable} for pointwise measurements and a suitable, tensor-based incoherence condition and \cite{mu2013square}, which considers a nuclear norm penalty of the matricization of the first $d/2$ modes of $\tensor{X}$ as opposed to a sum of nuclear norms of each of its $d$ modes, as is typically considered. 

Aside from convex relaxations of the tensor rank minimization problem,
the authors in \cite{tuckercompletion} develop an alternative
manifold-based approach to Tucker Tensor optimization similar to our
considerations for the Hierarchical Tucker case and subsequently
complete such tensors with missing entries. Each evaluation of the objective and Riemannian gradient requires $O(d(N+|\Omega|)K^d + dK^{d+1})$ operations, whereas our method only requires $O(dNK^2 + d|\Omega|K^3 + dK^4)$ operations. As a result of using the Hierarchical Tucker format instead of the Tucker format, our method scales much better as $d$, $N$, and $K$ grow. 

Previous work in completing tensors in the Tensor Train format, which is the Hierarchical Tucker format with a specific, degenerate binary dimension tree, includes \cite{prevwork:ttals,prevwork:altttorig}, wherein the authors use an alternating least-squares approach for the tensor completion problem. The derivations of the smooth manifold structure of the set of TT tensors can be found in \cite{ttmanifolds}. This work is a precursor for the manifold structure of Hierarchical Tucker tensors studied in \cite{htuckgeom}, upon which we expand in this article. For a comprehensive review of various tensor formats, we refer the reader to \cite{prevwork:tensorsurvey,prevwork:tensorsurvey2}. 

Owing to its extremely efficient storage requirements (which are 
\emph{linear} in the dimension $d$ as opposed to exponential in $d$),
the Hierarchical Tucker format has enjoyed a recent surge in
popularity for parametrizing high-dimensional problems. The hTucker toolbox \cite{htucktoolbox} contains a suite
of MATLAB tools for working with tensors in the HT format, including
efficient vector space operations, matrix-tensor and tensor-tensor
products, and truncations of full arrays to HT format. This
truncation, the so-called Hierarchical SVD developed in
\cite{htuckersvd}, allows one to approximate a full tensor in HT
format with a near-optimal approximation error. Even though
the authors in \cite{blackboxapprox-htuck} develop a HT truncation
method that does not need access to every entry of the tensor in order
to form the HT approximation, their approach requires algorithm-driven access to the entries, which does not apply for the seismic examples we consider below. A HT approach for solving dynamical systems is outlined in \cite{Lubich:2013ku}, which considers similar manifold structure as in this article applied in a different context. 

\subsection{Contributions and Organization}
In this paper, we extend the primarily theoretical results of
\cite{htuckgeom} to practical algorithms for solving optimization
algorithms on the HT manifold. In \Cref{sec:htd}, we introduce the
Hierarchical Tucker format.  We restate some of the results of
\cite{htuckgeom} in \Cref{sec:htd} to provide context for the
Riemannian metric we introduce on the quotient manifold in \Cref{sec:riemgeom}. Equipped with
this metric, we can now develop optimization methods on the HT
manifold in \Cref{sec:optimization} that are fast and SVD-free. For
large-scale, high-dimensional problems, the computational costs of
SVDs are prohibitive and affect the scalability of tensor completion
methods such as \cite{prevwork:tuckercompletion}. Since we are using the HT manifold rather than the Tucker manifold, we avoid an exponential dependence on the internal rank parameters as in \cite{tuckercompletion}. In \Cref{subsec:gramregularization}, we exploit the structure of HT tensors to regularize different
matricizations of the tensor \emph{without} having to compute SVDs of
these matricizations, lessening the effects of overfitting when there
are very few samples available. To the best of our knowledge, our approach is the first
instance of exploiting the manifold structure of HT tensors for
solving the tensor completion problem. We conclude by demonstrating the
effectiveness of our techniques on interpolating various seismic data
volumes with missing data points in all dimensions as well as missing
receivers, which is more realistic. Our numerical results are similar to those presented previously in \cite{DaSi1307:Hierarchical}, but much more extensive and include our regularization and Gauss-Newton based methods. In this paper, we also compare our method to a reference implementation of \cite{tuckercompletion} and achieve very reasonable results for our seismic data volumes. 

We note that the algorithmic results here generalize readily to complex tensor completion $\mathbb{C}^{n_1 \times n_2 \times \dots \times n_d}$ and more general subsampling operators $P_{\Omega}$. 

\section{Notation} 
%
In this paper, we denote vectors by lower case letters $x, y, z, \dots$, matrices by upper case, plain letters $A, B, C, \dots, X, Y Z$, and tensors by upper case, bold letters $\tensor{X}, \tensor{Y}, \tensor{Z}$.

\subsection{Matricization}
We let the
\emph{matricization} of a $n_1 \times n_2 \times \dots \times n_d$
tensor $\tensor{X}$ along the modes $t = (t_1, t_2, \dots, t_k)
\subset \{1, \dots, d\}$ be the matrix $X^{(t)}$ such that the indices
in $t$ are vectorized along the rows and the indices in $t^c$ are
vectorized along the columns, i.e., if we set $s = t^C$, then
\begin{align*}
X^{(t)} \in \mathbb{R}^{(n_{t_1} n_{t_2} ... n_{t_k}) \times (n_{s_1} n_{s_2} ... n_{s_{d -k}})} \\
(X^{(t)})_{(i_{t_1},..., i_{t_k}),(i_{s_1}, ... , i_{s_{d-k}})} := \mathbf{X}_{i_1,...,i_d}.
\end{align*}

We also use the notation $(\cdot)_{(t)}$ for the
\emph{dematricization} operation, i.e., $(X^{(t)})_{(t)} = \tensor{X}$, which
reshapes the matricized version of $\tensor{X}$ along modes $t$ back to its
full tensor form.

\subsection{Multilinear product}
\label{sec:multilinearproduct}
A natural operation to consider on tensors is that of the
\emph{multilinear product} \cite{htuckgeom,htuckersvd,prevwork:tensordecompositions}. 
\begin{defn} Given a $d-$tensor $\tensor{X}$ of size $n_1 \times n_2 \times
  \dots n_d$ and matrices $A_i \in \mathbb{R}^{m_i \times n_i}$, the
  multilinear product of $\{A_i\}_{i=1}^{d}$ with $\tensor{X}$, is the $m_1
  \times m_2 \times ... \times m_d$ tensor $\tensor{Y} = A_1 \times_1 A_2 \times_2 \dots A_d \times_d  \tensor{X}$,
  is defined in terms of the matricizations of $\tensor{Y}$ as
\begin{align*}
  Y^{(i)} = A_i X^{(i)} A_{d}^T \otimes A_{d-1}^T \otimes \dots
  A_{i+1}^T \otimes A_{i-1}^T \dots \otimes A_1^T, \quad i = 1, 2, \dots, d.
\end{align*}
Conceptually, we are applying operator each operator $A_i$ to dimension $i$ of the tensor $\tensor{X}$, keeping all other coordinates fixed. For example, when $A, X, B$ are matrices of appropriate sizes, the quantity $AXB^T$ can be written as $AXB^T = A \times_1 B \times_2 X$.
\end{defn}

The standard Euclidean inner product between two $d-$dimensional
tensors $X$ and $Y$ can be defined in terms of the standard Euclidean
product for vectors, by letting
\begin{align*}
  \langle \tensor{X}, \tensor{Y} \rangle := \vecmat(\tensor{X})^T \vecmat(\tensor{Y})
\end{align*}
where $\vecmat(\tensor{X}):= X^{(1,2,\dots,d)}$ is the usual vectorization
operator. This inner product induces a norm $\|\tensor{X}\|_2$ on the set of
all $d-$dimensional tensors in the usual way, $\|\tensor{X}\|_2 = \sqrt{\langle
  \tensor{X}, \tensor{X} \rangle}$.

Here we state several properties of the multilinear product, which are
 straightforward to prove.
\begin{proposition} 
Let $\{A_i\}_{i=1}^{d}$, $\{B_i\}_{i=1}^{d}$ be
  collections of linear operators and $\tensor{X}, \tensor{Y}$ be
  tensors, all of appropriate sizes, so that the multilinear products
  below are well-defined. Then we have the following:
\begin{enumerate}
\item $(A_1 \times_1 \dots, A_d \times_d) \circ (B_1 \times_1 \dots, B_d \times_d \tensor{X}) = (A_1 B_1) \times_1 \dots (A_d B_d) \times_d \tensor{X} \hfill \text{ \cite{prevwork:cp-approx-illposed} }$ 
\item $\langle A_1 \times \dots A_d \times_d \tensor{X}, B_1 \times_1 \dots B_d \times_d
\tensor{Y} \rangle = \langle (B_1^T A_1) \times_1 \dots (B_d^T A_d)\times_d
  \tensor{X}, \tensor{Y} \rangle $
\end{enumerate}
\end{proposition}

\subsection{Tensor-tensor contraction}
Another natural operation to consider between two tensors is
\emph{tensor-tensor contraction}, a generalization of matrix-matrix
multiplication. We define tensor-tensor contraction in terms of
tensors of the same dimension for ease of presentation
\cite{elden2009newton}.
\begin{defn}
  Given a $d-$tensor $\tensor{X}$ of size $n_1 \times \dots \times n_d$ and a
  $d-$tensor $\tensor{Y}$ of size $m_1 \times \dots \times m_d$, select $s, t
  \subset \{ 1, \dots, d \}$ such that $|s| = |t|$ and $n_{s_i} =
  m_{t_i}$ for $i=1, \dots, |s|$. The \emph{tensor-tensor contraction}
  of $\tensor{X}$ and $\tensor{Y}$ along modes $s,t$, denoted $\langle \tensor{X}, \tensor{Y}
  \rangle_{(s,t)}$, is defined as $(2d-(|s| + |t|))-$tensor $Z$ of
  size $(n_{s^c}, m_{t^c})$, satisfying
\begin{align*}
  \tensor{Z} = \langle \tensor{X}, \tensor{Y} \rangle_{(s,t)} = ( X^{(s^c)} Y^{(t)} )_{(s^c),(t^c)}.
\end{align*}
Tensor tensor contraction over modes $s$ and $t$ merely sums over the dimensions specified by $s, t$ in $\tensor{X}$ and $\tensor{Y}$ respectively, leaving the dimensions $s^c$ and $t^c$ free. 

\end{defn}

The inner product $\langle \tensor{X}, \tensor{Y} \rangle$ is a special case of the tensor product when $s = t = \{1, \dots, d\}$. 

We also make use of the fact that when the index sets $s, t$ are $s, t
= [d]\setminus i$ with $\tensor{X}$, $\tensor{Y}$, and
$A_i$ are appropriately sized for $i=1, \dots, d$, then
\begin{align}
\label{eq:multilinearproperty}
  \langle A_1 \times_1 A_2 \times_2 \dots A_d \times_d \tensor{X}, \tensor{Y} \rangle_{[d]\setminus i,
    [d]\setminus i} = A_i \langle A_1 \times_1 A_2 \times_2 \dots A_{i-1} \times_{i-1} 
  A_{i+1} \times_{i+1} \dots A_d \times_d \tensor{X}, \tensor{Y} \rangle_{[d]\setminus i, [d]\setminus
    i}
\end{align}
i.e., applying $A_i$ to dimension $i$ commutes with contracting tensors over every dimension except the $i$th one. 

\section{Smooth Manifold Geometry of the Hierarchical Tucker Format}
\label{sec:manifoldgeomht}
In this section, we review the definition of the Hierarchical Tucker format (\Cref{sec:htd}) as well as previous results \cite{htuckgeom} in the smooth manifold geometry of this format (\Cref{sec:quotmanifold}). We extend these results in the next section by introducing a Riemannian metric on the space of HT parameters and subsequently derive the associated Riemannian gradient with respect to this metric. A reader familiar with the results in \cite{htuckgeom} can glance over this section quickly for a few instances of notation and move on to \Cref{sec:riemgeom}.

\subsection{Hierarchical Tucker Format}
\label{sec:htd}
The standard definition of the Hierarchical Tucker format relies on the
notion of a \emph{dimension tree}, chosen apriori, which specifies the
format \cite{htuckersvd}. Intuitively, the dimension tree specifies which groups of
dimensions are ``separated'' from other groups of dimensions, where
``separation'' is used in a similar sense to the SVD in two dimensions. 

\begin{defn}
\label{def:T}
  A \emph{dimension tree} $T$ is a non-trivial binary tree such that
  \begin{itemize}
  \item the root, $\troot$, has the label $\troot = \{1, 2, \dots, d\}$
  \item for every $t \not\in L$, where $L$ is the set of leaves of
    $T$, the labels of its left and right children, $t_l, t_r$, form a
    partition of the label for $t$, i.e., $ t_l \cup t_r = t$ and $ t_l\cap t_r = \emptyset$.
  \end{itemize}
\end{defn}

An example of a dimension tree when $d = 6$ is given in \Cref{fig:dimtree}. 
\begin{figure}
\centering
\includegraphics[scale=0.9]{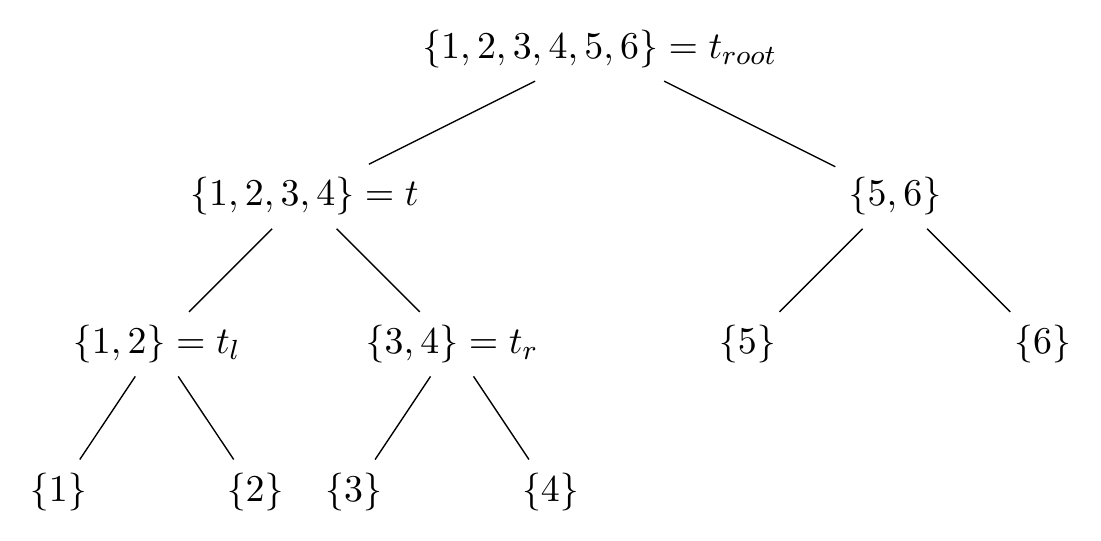}
\caption{ Complete dimension tree for $\{1, 2, 3, 4, 5, 6\}$. }
\label{fig:dimtree}
\end{figure}

\begin{remark}
For the following derivations, we take the point of view of each quantity with a subscript $(\cdot)_t$ is associated to the node $t \in T$.  By Definition \ref{def:T}, for each $t \in T$, there is a corresponding subset of $\{1, \dots, d\}$ associated to $t$. If our HT tensor has dimensions $n_1 \times n_2 \times ... \times n_d$, we let $n_t = \prod_{i \in t} n_i$ and, when $t \in T \setminus L$, $n_t$ satisfies $n_t = n_{t_l} n_{t_r}$. 
\end{remark}

\begin{defn}
\label{def:phix}
Given a dimension tree $T$ and a vector of \emph{hierarchical ranks}
$(k_t)_{t\in T}$ with $k_t \in \mathbb{Z}^+$, a tensor $\tensor{X} \in
\fullspace$ can be written in the
\emph{Hierarchical Tucker format} if there exist parameters $x = ((U_t)_{t \in L},
(\tensor{B}_t)_{t \in T \setminus L})$ such that $\phi(x) = \tensor{X}$, where
\begin{align}
\label{eq:phix}
\vecmat(\phi(x)) & = U_{t_l} \times_1 U_{t_r} \times_2 B_{\troot} & t = \troot\\
U_t & = (U_{t_l} \times_1 U_{t_r} \times_2 \tensor{B}_t)^{(1,2)} & t \not\in L \cup \troot \notag 
\end{align}
where $U_t \in \mathbb{R}_*^{n_t \times k_t}$, the set of full-rank
$n_t \times k_t$ matrices, for $t \in L$ and $\tensor{B}_t \in
\mathbb{R}_*^{k_{t_l} \times k_{t_r} \times k_t}$, the set of
$3$-tensors of full \emph{multilinear} rank, i.e., 
\begin{align*}
\text{rank}(B_t^{(1)}) = k_{t_l}, \quad \text{rank}(B_t^{(2)}) = k_{t_r}, \quad \text{rank}(B_t^{(3)}) = k_t.
\end{align*}

\end{defn}

We say the parameters $x = (U_t, \tensor{B}_t)$ are in \emph{Orthogonal
  Hierarchical Tucker} (OHT) format if, in addition to the above construction, we
also have
\begin{align}
  U_t^T U_t &= I_{k_t} \quad \text{for } t\in L \notag \\
  (B_t^{(1,2)})^T B_t^{(1,2)} & = I_{k_t} \quad \text{for } t\not\in L \cup \troot
\label{eq:oht}
\end{align}

We have made a slight modification of the definition of the HT format
compared to \cite{htuckgeom} for ease of presentation. When $d = 2$,
our construction is the same as the subspace decomposition introduced
in \cite{prevwork:subspacemc} for low-rank matrices, but our approach
is not limited to this case.

Owing to the recursive construction (\ref{def:phix}), the intermediate matrices $U_t$ for $t \in T \setminus L$ do not need to be stored. Instead, specifying $U_t$ for $t \in L$ and $\tensor{B}_t$ for $t \in T \setminus L$ determines $\tensor{X} = \phi(x)$ completely. Therefore, the overall number of parameters $x = ((U_t)_{t \in L}, (\tensor{B}_t)_{t \in T \setminus L})$ is bounded above by $dNK + (d - 2)K^3 + K^2$, 
where $N = \max_{i=1, \dots, d} n_i$ and $K = \max_{t \in T} k_t$. When $d \ge 4$ and $K \ll N$, this quantity is much less than the $N^d$ parameters typically needed to represent $\tensor{X}$. 
\begin{defn}
  The \emph{hierarchical rank} of a tensor $\tensor{X} \in
  \fullspace$ corresponding to a
  dimension tree $T$ is the vector $\mathbf{k} = (k_t)_{t \in T}$
  where
\begin{align*}
k_t = \text{rank}(X^{(t)}).
\end{align*}
\end{defn}
We consider the set of \emph{Hierarchical Tucker tensors of fixed rank $\mathbf{k} = (k_t)_{t \in T}$}, that is,
\begin{align*}
  \htuckspaceclean = \{ \tensor{X} \in \fullspace |\, \text{rank}({X}^{(t)}) = {k}_t\quad \text{for} \, t\in
  T\setminus \troot \} . 
\end{align*}

We restrict ourselves to parameters $x$ that are strictly orthogonalized, as in (\ref{eq:oht}). In addition to significantly simplifying the resulting notation, this restriction allows us to avoid cumbersome and unnecessary matrix inversions, in particular for the resulting subspace projections in future sections. Moreover, using only orthogonalized parameters avoids the problem of algorithms converging to points with possibly lower than prescribed HT rank, see \cite[Remark 4.1]{htuckgeom}. This restriction does not reduce the expressibility of the HT format, however, since for any non orthogonalized parameters $x$ such that $\mathbf{X} = \phi(x)$, there exists orthogonalized parameters $x'$ with $\mathbf{X} = \phi(x')$ \cite[Alg. 3]{htuckersvd}. 

We use the grouping $x = (U_t, \tensor{B}_t)$ to denote $((U_t)_{t \in L}, (\tensor{B}_t)_{t \in T \setminus L})$, as these are our ``independent variables'' of interest in this case. In order to avoid cumbersome notation, we also suppress the dependence on $(T, \mathbf{k})$ in the following, and presume a fixed dimension tree $T$ and hierarchical ranks $\mathbf{k}$. 

\begin{figure}
\centering
\includegraphics[scale=0.9]{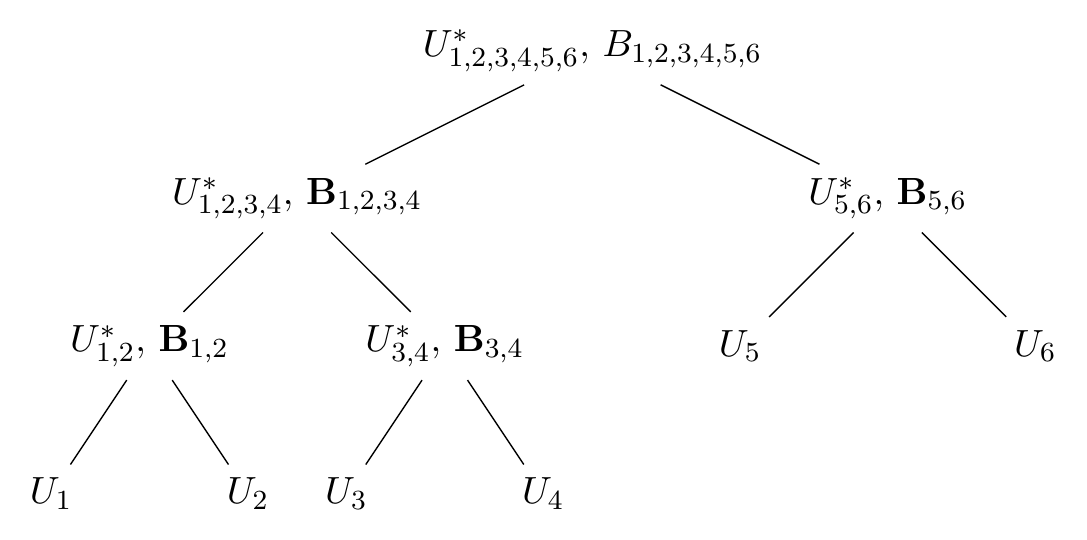}
\caption{ Dimension tree representation of (\ref{eq:phix}) with dimensions $\{1, 2, 3, 4, 5, 6\}$. Starred quantities are computed recursively. }
\label{fig:dimtree2}
\end{figure}

\subsection{Quotient Manifold Geometry}
\label{sec:quotmanifold}
The results in this section are adapted from those in \cite{htuckgeom} to the orthogonalized parameter case and we include them here in the interest of keeping this article self contained. \\
Below, let $S_{k_{t_l},k_{t_r},k_t}$ be the closed submanifold of $\mathbb{R}_*^{k_{t_l} \times k_{t_r} \times k_t}$, the set of $3-$tensors with full multilinear rank, such that $\mathbf{S} \in S_{k_{t_l},k_{t_r},k_t}$ is orthonormal along modes $1$ and $2$, i.e., $(S^{(1,2)})^T S^{(1,2)} = I_{k_t}$ and let $\text{St}(n_t,k_t)$ be the $n_t \times k_t$ Stiefel manifold of $n\times k_t$ matrices with orthonormal columns. \\
Our orthogonal parameter space $\paramspaceclean$ is then 
\begin{align*}
\paramspaceclean = \bigtimes_{t \in L} \text{St}(n_t, k_t) \times \bigtimes_{t \not\in L \cup \troot} S_{k_{t_l}, k_{t_r}, k_t} \times  \mathbb{R}_*^{k_{(\troot)_l} \times k_{(\troot)_r}}
\end{align*}
with corresponding tangent space at $x = (U_t, \tensor{B}_t) \in \paramspaceclean$
\begin{align*}
\mathcal{T}_x \paramspaceclean = \bigtimes_{t \in L} \mathcal{T}_{U_t} \text{St}(n_t, k_t) \times \bigtimes_{t \not\in L \cup\troot}  \mathcal{T}_{\tensor{B}_t} S_{k_{t_l}, k_{t_r}, k_{t} } \times \mathbb{R}^{k_{(\troot)_l} \times k_{(\troot)_r}}.
\end{align*}
Note that $\mathcal{T}_{Y} \text{St}(n,p) = \{ Y \Omega + Y^{\perp} K : \Omega^T = -\Omega \in \mathbb{R}^{p \times p}, K \in \mathbb{R}^{(n-p) \times p} \}$. We omit an explicit description of $\mathcal{T}_{B_t} S_{k_{t_l}, k_{t_r}, k_{t} }$ for brevity. \\

Let $\phi : \paramspaceclean \to \htuckspaceclean$ be the parameter to tensor map in (\ref{eq:phix}). Then for each $\tensor{X} \in \htuckspaceclean$, then there is an inherent ambiguity in its representation by parameters $x$, i.e., $\tensor{X} = \phi(x) = \phi(y)$ for distinct parameters $x$ and $y$ with the following relationship between them. 

Let $\liegroup$ be the Lie group
\begin{align*}
\liegroup = \{ (A_t)_{t \in T} : A_t \in O(k_t) \}.
\end{align*}
where $O(p)$ is the orthogonal group of $p\times p$ matrices and the group action of component-wise multiplication. Let $\theta$ be the group action 
\begin{align}
\label{eq:groupaction}
& \theta : \paramspaceclean \times \liegroup \to \paramspaceclean \\
& (x,\mathcal{A}) := ((U_t, B_t), (A_t)) \mapsto \theta_x(\mathcal{A}) := (U_t A_t, A_{t_l}^{T} \times_1 A_{t_r}^{T} \times_2 A_t^T \times_3 \tensor{B}_t). \notag
\end{align}
Then $\phi(x) = \phi(y)$ if and only if there exists a unique $ \mathcal{A} = (A_t)_{t \in T} \in \liegroup$ such that $x = \theta_{A}(y)$ \cite[Prop. 3.9]{htuckgeom}. Therefore these are the only types of ambiguities we must consider in this format. 

It follows that the orbit of $x$, 
\begin{align*}
\liegroup x = \{ \theta_{\mathcal{A}}(x) : \mathcal{A} \in \liegroup \},
\end{align*}
is the set of all parameters that map to the same tensor $\tensor{X} = \phi(x)$ under $\phi$. This induces an equivalence relation on the set of parameters $\paramspaceclean$,
\begin{align*}
x \sim y \text{ if and only if } y \in \liegroup x.
\end{align*}
If we let $\quotspace$ be the corresponding quotient space of equivalence classes and $\pi : \paramspaceclean \to \quotspace$ denote the quotient map, then pushing $\phi$ down through $\pi$ results in an injective function
\begin{align*}
\hat{\phi} : \quotspace \to \htuckspaceclean 
\end{align*}
whose image is all of $\htuckspaceclean$, and hence is an isomorphism (in fact, a diffeomorphism). 

The vertical space, $\vertspace$, is the subspace of
$\paramspacecleanT$ that is tangent to $\pi^{-1}(x)$. That is, $dx^v = (\delta U_t^v, \delta \tensor{B}_t^v) \in \vertspace $ when it is of the form \cite[Eq. 4.7]{htuckgeom}
\begin{align}
\label{eq:vertspace}
\delta U_t^v &= U_t D_t & \quad &\text{for } t \in L \notag \\
\delta \tensor{B}_t^v &= D_t \times_3 \tensor{B}_t - D_{t_l} \times_1 \tensor{B}_t - D_{t_r} \times_2 \tensor{B}_t & \quad &\text{for } t \in T \setminus L \cup \troot \notag \\
\delta B_{\troot}^v  &= - D_{t_l} B_{\troot} - B_{\troot} D_{t_r}^T & \quad &\text{for } t = \troot
\end{align}
where $D_t \in \text{Skew}(k_t)$, the set of $k_t \times k_t$ skew symmetric matrices. A straightforward computation shows that $D\phi(x)\big |_{\vertspace} \equiv 0$, and therefore for every $dx^v \in \vertspace$, $\phi(x) = \phi(x + dx^v)$. From an optimization point of view, moving from the point $x$ to $x + dx^v$ will not change the current tensor $\phi(x)$ and therefore for any search direction $p$, we must filter out the corresponding component in $\vertspace$ in order to compute the gradient correctly. We accomplish this by projecting on to a \emph{horizontal space}, which is any complementary subspace to $\vertspace$. One such choice is \cite[Eq. 4.8]{htuckgeom}, 
\begin{align}
\label{eq:horizspace}
\horizspace = \left\{ (\delta U_t^h, \delta \tensor{B}_t^h) : \begin{aligned}
    (\delta U_t^h)^T U_t = 0_{k_t} & \; \text{for } t \in L \\
    (\delta B_t^{(1,2)})^T B_t^{(1,2)} = 0_{k_t} & \; \text{for } t \not\in L\cup \troot
\end{aligned} \right\}.
\end{align}
Note that there is no restriction on $B_{\troot}^h$, which is a
matrix. \\
This choice has the convenient property that $\horizspace$ is \emph{invariant} under the action of $\theta$, i.e., \cite[Prop. 4.9]{htuckgeom}
\begin{align}
\label{eq:horizspacetheta}
D\theta(x,\mathcal{A})[\horizspace,0] = \mathcal{H}_{\theta_x(\mathcal{A})} \paramspaceclean,
\end{align}
which we shall exploit for our upcoming discussion of a Riemannian metric. 
The horizontal space $\horizspace$ allows us to uniquely represent abstract tangent vectors in $T_{\pi(x)} \quotspace$ with concrete vectors in $\horizspace$. 

\section{Riemannian Geometry of the HT Format}
\label{sec:riemgeom}
In this section, we introduce a Riemannian metric on the parameter space $\paramspaceclean$ that will allow us to use parameters $x$ as representations for their equivalence class $\pi(x)$ in a well-defined manner while performing numerical optimization. 
\subsection{Riemannian metric}
Since each distinct equivalence class $\pi(x)$ is uniquely identified
with each distinct value of $\phi(x)$, the quotient manifold
$\paramspaceclean / \liegroup$ is really our manifold of interest for
the purpose of computations---i.e, we would like to formulate our
optimization problem over the equivalence classes
$\pi(x)$. Unfortunately, $\paramspaceclean / \liegroup$ is an abstract
mathematical object and thus hard to implement numerically. By  
introducing a Riemannian metric on $\paramspaceclean$ that respects
its quotient structure, we can formulate concrete
optimization algorithms in terms of the HT parameters without being affected
by the non-uniqueness of the format---i.e., by optimizing over parameters $x$ while implicitly performing optimization over equivalence classes $\pi(x)$. Below, we explain how to
explicitly construct this Riemannian metric for the HT format.

Let $\eta_x = (\delta U_t, \delta \tensor{B}_t), \zeta_x = (\delta
V_t, \delta \tensor{C}_t) \in \paramspacecleanT$ be tangent vectors at the point $x = (U_t, \tensor{B}_t) \in \paramspaceclean$. Then we define the inner product $g_x
(\cdot,\cdot)$ at $x$ as
\begin{align}
\label{eq:riemmetric}
g_x(\eta_x, \zeta_x) & := \sum_{t \in T} \tr((U_t^TU_t)^{-1} \delta U_t^T \delta V_t)\\
&+ \sum_{t \not\in L \cup \troot} \langle \delta \tensor{B}_t , (U_{t_l}^T U_{t_l}) \times_1 (U_{t_r}^T U_{t_r}) \times_2 (U_t^TU_t)^{-1} \times_3 \delta \tensor{C}_t \rangle \notag\\
&+ \text{tr}( U_{(\troot)_r}^T U_{(\troot)_r} \delta B_{\troot}^T U_{(\troot)_l}^T U_{(\troot)_l}
\delta C_{\troot} ). \notag 
\end{align}

By the full-rank conditions on $U_t$ and $\tensor{B}_t$ at each node, by
definition of the HT format, each $U_t^T U_t$ for $t \in T$ is
symmetric positive definite and varies smoothly with $x = (U_t,
\tensor{B}_t)$. As a result, $g_x$ is a smooth, symmetric
positive definite, bilinear form on $\paramspacecleanT$, i.e., a
Riemannian metric. Note that when $x$ is in OHT, as it is in our case, $g_x$ reduces to the
standard Euclidean product on the parameter space $\paramspacecleanT$,
making it straightforward to compute in this case. 

\begin{proposition}\label{prop:isometric} On the Riemannian manifold $(\paramspaceclean,g)$,
  $\theta$ defined in (\ref{eq:groupaction}) acts \emph{isometrically}
  on $\mathcal{M}$, i.e., for every $\fancyA \in \liegroup$, $\xi_x,
  \zeta_x \in \horizspace$
\begin{align*}
  g_x(\xi_x,\zeta_x) = g_{\theta_{\mathcal{A}}(x)}(\theta^*\xi_x, \theta^*\zeta_x)
\end{align*}
where $\theta^*$ is the \emph{push-forward} map, $\theta^*v = D\theta(x,\mathcal{A})[v]$.
\end{proposition}

\begin{proof}
Let $x = (U_t, \tensor{B}_t) \in \paramspaceclean$, $y = (V_t, \tensor{C}_t) = \theta_{\mathcal{A}}(x) = (U_tA_t,A_{t_l}^{T} \times_1 A_{t_r}^{T} \times_2 A_t^T\times_3 \tensor{B}_t)$ for $\mathcal{A} \in \liegroup$.

If we write $\eta_x = (\delta U_t, \delta \tensor{B}_t)$, $\zeta_x = (\delta V_t, \delta \tensor{C}_t)$ for $\eta_x,\zeta_x \in \horizspace$, then, by (\ref{eq:horizspacetheta}), it follows that 
\begin{align*}
\eta_{y} = \theta^* \eta_x = (\delta U_t A_t, A_{t_l}^{T} \times_1 A_{t_r}^{T} \times_2 A_t^T \times_3 \delta \tensor{B}_t)
\end{align*}
and similarly for $\zeta_{y}$.

We will compare each component of the sum of (\ref{eq:riemmetric}) term by term. For ease of presentation, we only consider interior nodes $t \not\in L \cup \troot$, as leaf nodes and the root node are handled in an analogous manner.

For $t \not\in L \cup \troot$, let $(\xi_y)_t $ be the component of $\xi_y$ at the node $t$, i.e.,
\begin{align*}
(\xi_y)_t &= A_{t_l}^{T} \times_1 A_{t_r}^{T} \times_2 A_t^T \times_3 \delta \tensor{B}_t \\
& := \widetilde{\delta \tensor{B}_t},
\end{align*}
and similarly let $\widetilde{\delta \tensor{C}_t} := (\zeta_y)_t$. 

A straightforward computation based on (\ref{eq:phix}) and  (\ref{eq:groupaction}) yields that 
\begin{align*}
 V_{t_l}^T V_{t_l} \times_1 V_{t_r}^T V_{t_r} \times_2 (V_t^T V_t)^{-1} \times_3 \tensor{Y}  = (A_{t_l}^TU_{t_l}^T U_{t_l} A_{t_l})\times_1 (A_{t_r}^T U_{t_r}^T U_{t_r} A_{t_r}) \times_2 (A_{t}^{T}(U_t^T U_t)^{-1} A_{t}) \times_3 \tensor{Y}
\end{align*}
for appropriately sized $\tensor{Y}$. In particular, for $\tensor{Y} = \widetilde{\tensor{\delta C}_t}$, we have that
\begin{align*}
&\langle \widetilde{\delta \tensor{B}_t}, V_{t_l}^T V_{t_l} \times_1 V_{t_r}^T V_{t_r} \times_2 (V_t^T V_t)^{-1} \times_3 \widetilde{\delta \tensor{C}_t} \rangle \\
& = \langle A_{t_l}^{T} \times_1 A_{t_r}^{T} \times_2 A_t^T \times_3 \delta \tensor{B}_t, \\
& ((A_{t_l}^TU_{t_l}^T U_{t_l} A_{t_l}) \times_1 (A_{t_r}^T U_{t_r}^T U_{t_r} A_{t_r}) \times_2 (A_{t}^{T}(U_t^T U_t)^{-1} A_{t}) \times_3) \circ (A_{t_l}^{T} \times_1 A_{t_r}^{T} \times_2 A_t^T \times_3 \delta \tensor{C}_t) \rangle \\
&= \langle \delta \tensor{B}_t, (U_{t_l}^T U_{t_l}) \times_1 (U_{t_r}^T U_{t_r}) \times_2 (U_t^T U_t)^{-1} \times_3 \delta \tensor{C}_t \rangle \quad \text{using Prop. 1.2 and } A_t \in O(k_t). \\
\end{align*} 
Adding the terms for each $t \in T$, we obtain
\begin{align*}
g_x(\xi_x, \zeta_x) = g_{\theta_{\mathcal{A}}(x)}(\theta^*\xi_x, \theta^*\zeta_x).
\end{align*}
\end{proof}

Although the above computation uses the fact that $A_t \in O(k_t)$, an almost identical calculation yields \Cref{prop:isometric} when $x$ is non orthogonalized, as considered in \cite{htuckgeom}.  As we are interested in carrying out
our optimization using the HT parameters $x$ as proxies for their equivalence
classes $\pi(x)$, this proposition states that if we measure inner
products between two tangent vectors at the point $x$, we obtain the
same result as if we had measured the inner product between two
tangent vectors transformed by $\theta_{\fancyA}$ at the point
$\theta_{\fancyA} (x)$. In this sense, once we have a unique association of tangent vectors in $\quotspace$ with a subspace of $\paramspacecleanT$, we can use the actual
representatives, the parameters $x$, instead of the abstract equivalence class $\pi(x)$, in a well-defined way 
during our optimization. 
This shows that $\quotspace$, endowed with the Riemannian metric
\begin{align*}
  g_{\pi(x)}(\xi,\zeta) := g_x(\xi^h_x,\zeta^h_x)
\end{align*}
where $\xi^h, \zeta^h_x$ are the horizontal lifts at $x$ of $\xi,\zeta$, respectively, is a Riemannian quotient manifold of $\mathcal{M}$ (i.e., $\pi :
\mathcal{M} \to \mathcal{M} / \liegroup$ is a Riemannian submersion) \cite[Sec. 3.6.2]{optmatrixmanifold}. 

In summary, by using this Riemannian metric and restricting our
optimization to only consider horizontal tangent vectors, we can
implicitly formulate our algorithms on the abstract quotient space by
working with the concrete HT parameters. Below, we will derive the
Riemannian gradient in this context.

\begin{remark}
It should be noted that although the horizontal space
(\ref{eq:horizspace}) is complementary to the vertical space
(\ref{eq:vertspace}), it is demonstrably \emph{not} perpendicular to
$\vertspace$ under the Riemannian metric
(\ref{eq:riemmetric}). Choosing a horizontal space which is
perpendicular to $\vertspace$ under the standard Euclidean product
(i.e., (\ref{eq:riemmetric}) when $x$ is orthogonalized) is beyond the
scope of this paper. Suffice to say, it can be done, as a
generalization of the approach outlined in \cite{prevwork:subspacemc},
resulting in a series of symmetry conditions on various multi-way
combinations of parameters. The resulting projection operators involve
solving a number of coupled Lyapunov equations, increasing with the
depth of $T$. It remains to be
seen whether such equations can be solved efficiently when $d$ is
large. We will not dwell on this point here, as we will not be needing
orthogonal projections for our computations in the following.
\end{remark}

\subsection{Riemannian gradient}
\label{sec:riemanniangradient}
The problem we are interested in solving is 
\begin{align*}
\min_{x \in \paramspaceclean} f(\phi(x)) 
\end{align*}
for a smooth objective function $f : \fullspace \to  \mathbb{R}$. We write $\hat{f} : \paramspaceclean \to \mathbb{R}$, where $\hat{f}(x) = f(\phi(x))$. 

We need to derive expressions for the Riemannian
gradient to update the HT parameters as part of local optimization
procedures. Therefore, our primary quantity of interest is the
\emph{Riemannian gradient} of $\hat{f}$.
\begin{defn} \cite[Sec. 3.6]{optmatrixmanifold}
  Given a smooth scalar function $\hat{f}$ on a Riemannian manifold
  $\mathcal{N}$, the \emph{Riemannian gradient} of $\hat{f}$ at $x \in \mathcal{N}$,
  denoted $\nabla^R \hat{f}(x)$, is the unique element of
  $\mathcal{T}_x \mathcal{N}$ which satisfies
\begin{align*}
  g_x( \nabla^R \hat{f}(x), \xi ) = D\hat{f}(x)\left[\xi\right] \quad \forall \xi \in \mathcal{T}_x \mathcal{N}
\end{align*}
with respect to the Riemannian metric $g_x( \cdot, \cdot )$.
\end{defn}
Our manifold of interest in this case is $\mathcal{N} = 
\quotspace$, with the corresponding horizontal space
$\horizspace$ in lieu of the abstract tangent space $T_{\pi(x)} \quotspace$. Therefore, in the above equation, we can consider the 
horizontal lift $\xi^h$ of the tangent vector $\xi$ and instead write
\begin{align*}
g_x( \nabla^R \hat{f}(x), \xi^h ) = D\hat{f}(x)[\xi^h].
\end{align*}
Our derivation is similar to that of \cite[Sec 6.2.2]{htuckgeom}, except our derivations are more streamlined and cheaper computationally since we reduce the operations performed at the interior nodes $t \in T \setminus L$. By a slight abuse of notation in this section, we denote variational quantities associated to node $t$ as $\delta Z_t \in \mathbb{R}^{n_{t_l}n_{t_r} \times k_t}$ and $\delta \tensor{Z}_t \in \mathbb{R}^{n_{t_l} \times n_{t_r} \times k_t}$ where $(\delta Z_t)_{(1,2)} = \delta \tensor{Z}_t$ is the reshaping of $\delta Z_t$ in to a $3-$tensor. The Riemannian gradient will be denoted $(\delta U_t, \delta \tensor{B}_t)$ and a general horizontal vector will be denoted by $(\delta V_t, \delta \tensor{C}_t)$.

Since $x = (U_t, \tensor{B}_t)$ is orthogonalized, we use $\langle \cdot, \cdot \rangle$ to denote the Euclidean inner product. By the chain rule, we have that, for any $\xi = (\delta V_t, \delta \tensor{C}_t) \in \horizspace$, 
\begin{align*}
D\hat{f}(x)[\xi] &= Df(\phi(x))[D\phi(x)[\xi]]\\
&= \langle \nabla_{\phi(x)} f(\phi(x)), D\phi(x)[\xi] \rangle.
\end{align*}
Then each tensor $\delta \tensor{V}_t \in \mathbb{R}^{n_{t_l} \times n_{t_r} \times k_t}$, with $\delta V_{\troot} = D\phi(x)[\xi]$, satisfies the recursion  
\begin{align}
\label{eq:dphi}
\delta \tensor{V}_t = \delta V_{t_l} \times_1 U_{t_r} \times_2 \tensor{B}_t + U_{t_l} \times_1
\delta V_{t_r} \times_2 \tensor{B}_t + U_{t_l} \times_1 U_{t_r} \times_2 \delta \tensor{C}_t,
\end{align}
for matrices $\delta V_{t_l} \in \mathbb{R}^{n_{t_l} \times k_{t_l}}$, $\delta V_{t_r} \in \mathbb{R}^{n_{t_r} \times k_{t_r}}$ and tensor $\tensor{\delta
C_t} \in \mathbb{R}^{k_{t_l} \times k_{t_r} \times k_t }$ satisfying \cite[Lemma 2]{htuckgeom}
\begin{align}
\label{eq:orthogonalitycontraints}
  \delta V_{t_l}^T U_{t_l} = 0 \quad \delta V_{t_r}^T U_{t_r} = 0
  \quad (\delta C_t^{(1,2)})^T B_t^{(1,2)} = 0.
\end{align}
The third orthogonality condition is omitted when $t = \troot$. 

Owing to this recursive structure, we compute $\langle \delta \tensor{U}_t, \delta \tensor{V}_t \rangle$, where $\delta \tensor{U}_t$ is the component of the Riemannian gradient at the current node and recursively extract the components of the Riemannian gradient associated to the children, i.e., $\delta U_{t_l}, \delta U_{t_r}$, and $\delta \tensor{B}_t$. Here we let $\delta U_{\troot} = \nabla_{\phi(x)} f(\phi(x))$ be the Euclidean gradient of $f(\phi(x))$ at $\phi(x)$, reshaped into a matrix of size $n_{(\troot)_l} \times n_{(\troot)_r}$. 

Let $\dfrac{\partial U_t}{\partial U_{t_l}}$ be the linear operator
such that $\dfrac{\partial U_t}{\partial U_{t_l}} \delta V_{t_l} =
\delta V_{t_l} \times_1 U_{t_r} \times_2 \tensor{B}_t$ and similarly for
$\dfrac{\partial U_t}{\partial U_{t_l}}$, $\dfrac{\partial
  U_t}{\partial \tensor{B_t}}$.\\
Then $\langle \delta \tensor{U}_t, \delta \tensor{V}_t \rangle$ is equal to
\begin{align}
\label{eq:dotproddecomp}
\langle \delta \tensor{U}_t, \dfrac{\partial U_t}{\partial U_{t_l}} \delta V_{t_l}
\rangle + \langle  \delta \tensor{U}_t, \dfrac{\partial U_t}{\partial U_{t_r}} \delta
V_{t_r} \rangle + \langle  \delta \tensor{U}_t, \dfrac{\partial U_t}{\partial \tensor{B}_t} \tensor{\delta
C}_t \rangle.
\end{align}
If we set 
\begin{align}
\label{eq:recursivegrad}
\delta U_{t_l} = P_{U_{t_l}}^{\perp} \left(\dfrac{\partial U_t}{\partial U_{t_l}}\right)^T \delta \tensor{U}_t, \; \delta U_{t_r} = P_{U_{t_r}}^{\perp} \left(\dfrac{\partial U_t}{\partial U_{t_r}}\right)^T \delta \tensor{U}_t, \;
\delta \tensor{B}_t	 = \left(P_{M}^{\perp} \left( \left(\dfrac{\partial U_t}{\partial B_t}\right)^T \delta \tensor{U}_t \right)^{(1,2)}\right)_{(1,2)}
\end{align}
where 
\begin{align*}
  P_M^{\perp} &= I_{k_t} - B_t^{(1,2)} (B_t^{(1,2)})^T \quad & \text{ if } t \neq \troot \\
  P_M^{\perp} &= I \quad & \text{ if } t = \troot
\end{align*}
and $P_{U_t} = (I_{k_t} - U_t U_t^T)$ is the usual projection on to $\text{span}(U_t)^{\perp}$, then we have that (\ref{eq:dotproddecomp}) is equal to %
\begin{align*}
  \langle \delta U_{t_l}, \delta V_{t_l} \rangle + \langle \delta U_{t_r},
  \delta V_{t_r} \rangle + \langle \delta \tensor{B}_t, \delta \tensor{C}_t \rangle
\end{align*}
and $\delta U_{t_l}, \delta U_{t_r}$, and $\delta \tensor{B}_t$ satisfy (\ref{eq:orthogonalitycontraints}). Their recursively
decomposed factors will therefore be in the horizontal space $\horizspace$.

$\delta \tensor{B}_t$ is the component of the Riemannian gradient at node $t$. If $t_l$ is a leaf node, then we have extracted the component of the Riemannian gradient associated to $t_l$, namely $\delta U_{t_l}$. Otherwise, we set $\delta \tensor{U}_{t_l} = (\delta U_{t_l})_{(1,2)}$ and apply the above recursion. We make the same considerations for the right children. 

We compute the adjoint partial derivatives via
\begin{align}
\label{eq:adjpartialderivatives}
  \dfrac{\partial U_t}{\partial U_{t_l}}^T \delta \tensor{U}_t = \langle U_{t_r}^T \times_2 \delta \tensor{U}_t, \tensor{B}_t \rangle_{(2, 3), (2,3)} , \;
  \dfrac{\partial U_t}{\partial U_{t_r}}^T \delta \tensor{U}_t = \langle U_{t_l}^T \times_1 \delta \tensor{U}_t, \tensor{B}_t \rangle_{(1, 3), (1,3)}, \;
  \dfrac{\partial U_t}{\partial \tensor{B}_t}^T \delta \tensor{U}_t =
  U_{t_l}^T \times_1 U_{t_r}^T \times_2 \delta \tensor{U}_t
\end{align}
For the general case of computing these adjoint operators, we refer to \ref{app:adjops}. In the above computations, the multilinear product operators are never formed explicitly and instead each operator is applied to various reshapings of the matrix or tensor of interest, see \cite{spottoolbox} for a reference Matlab implementation.

In order to minimize the number of computations performed on intermediate tensors, which are much larger than $\text{dim}(\paramspaceclean)$, we first note that in computing the terms
\begin{align*}
P_{U_{t_l}}^{\perp} \langle U_{t_r}^T \times_2 \delta \tensor{U}_t, \tensor{B}_t \rangle_{(2,3),(2,3)},
\end{align*}
that $\delta \tensor{U}_t = (P_{U_t}^{\perp} \delta \tilde{U}_t)_{(1,2)}$ for a matrix $\delta \tilde{U}_t \in \mathbb{R}^{n_{t_l} n_{t_r} \times k_t}$. Using (\ref{eq:multilinearproperty}), the above expression can be written as 
\begin{align}
\label{eq:simplifiedderivative}
 \langle P_{U_{t_l}}^{\perp} \times_1 U_{t_r}^T \times_2 (P_{U_t}^{\perp} \delta \tilde{U}_t)_{(1,2)}, \tensor{B}_t \rangle_{(2,3),(2,3)}.
\end{align}
We note that in the above, $P_{U_{t_l}}^{\perp} \times_1 U_{t_r}^T \times_2 (P_{U_t}^{\perp} \delta \tilde{U}_t)_{(1,2)} = (U_{t_r}^{T} \otimes P_{U_{t_l}}^{\perp} P_{U_t}^{\perp} \delta \tilde{U}_t)_{(1,2)}$, and the operator applied to $\delta \tilde{U_t}$ satisfies
\begin{align*}
U_{t_r}^T \otimes P_{U_{t_l}}^{\perp} P^{\perp}_{U_t} &= U_{t_r}^T \otimes P_{U_{t_l}}^{\perp} (I_{n_t} - U_t U_t^T) \\
&= U_{t_r}^T \otimes P_{U_{t_l}}^{\perp} (I_{n_t} - U_{t_r} \otimes U_{t_l} B_{t}^{(1,2)} (B_t^{(1,2)})^T U_{t_r}^T \otimes U_{t_l}^T)\\
&= U_{t_r}^T \otimes P_{U_{t_l}}^{\perp}.
\end{align*}
This means that, using (\ref{eq:multilinearproperty}), we can write (\ref{eq:simplifiedderivative}) as 
\begin{align*}
P_{U_{t_l}}^{\perp} \langle U_{t_r}^T \times_2 (\delta \tilde{U}_t)_{(1,2)}, \tensor{B}_t \rangle_{(2,3),(2,3)}
\end{align*}
i.e., we do not have to apply $P_{U_t}^{\perp}$ to the matrix $\delta \tilde{U}_t$ at the parent node of $t$. Applying this observation recursively and to the other terms in the Riemannian gradient, we merely need to orthogonally project the resulting extracted parameters $(\delta U_t, \delta \tensor{B}_t)$ on to $\horizspace$ after applying the formula (\ref{eq:recursivegrad}) \emph{without} applying the intermediate operators $P^{\perp}_{U_t}$, reducing the overall computational costs. We summarize our algorithm for computing the Riemannian gradient in \Cref{alg:riemmanngrad}.

\begin{algorithm}
\caption{The Riemannian gradient $\nabla^R f$ at a point $x = (U_t, \tensor{B}_t) \in \paramspaceclean$}
\label{alg:riemmanngrad}
\begin{algorithmic}
\REQUIRE $x = (U_t, \tensor{B}_t)$ parameter representation of the current point
\STATE Compute $\tensor{X} = \phi(x)$ and $\nabla_{\tensor{X}} f(\tensor{X})$, the Euclidean gradient of $f$, a $n_1 \times \dots n_d$ tensor.
\STATE $\delta U_{\troot} \gets (\nabla_{\tensor{X}} f(\mathbf{X}))_{(1,2)}$
\FOR{$t \in T \setminus L$, visiting parents before their children}
\STATE $\delta \tensor{U}_t \gets (\delta U_t)_{(1,2)}$
\STATE $\delta U_{t_l} \gets \langle U_{t_r}^T \times_2 \delta \tensor{U}_t, \tensor{B}_t \rangle_{(2,3),(2,3)} $, $\delta U_{t_r} \gets \langle U_{t_l}^T \times_1 \delta \tensor{U}_t, \tensor{B}_t \rangle_{(1,3),(1,3)}$
\STATE $\delta \tensor{B}_t \gets U_{t_l}^T \times_1 U_{t_r}^T \times_2 \delta \tensor{U}_t$
\IF{$t \neq \troot$}
\STATE $\tensor{\delta B_t} \gets (P_{B_t^{(1,2)}}^{\perp} (\delta B_t)^{(1,2)} )_{(1,2)}$ 
\ENDIF
\ENDFOR
\FOR{$t \in L$}
\STATE $\delta U_t \gets P_{U_t}^{\perp} \delta U_t$
\ENDFOR
\RETURN $\nabla^R f \gets (\delta U_t, \delta \tensor{B}_t)$
\end{algorithmic}
\end{algorithm}
\Cref{alg:riemmanngrad} is computing the operator
$D\phi(x)^* : \mathcal{T}_{\phi(x)} \mathcal{H} \to \horizspace $ applied to the Euclidean gradient $\nabla_{\phi(x)} f(\phi(x))$. The
forward operator $D\phi(x) : \horizspace \to \mathcal{T}_{\phi(x)}
\mathcal{H}$ can be computed using a component-wise orthogonal
projection $P^H_x : \mathcal{T}_x \paramspaceclean \to \horizspace $ followed by
applying (\ref{eq:dphi}) recursively.

\subsection{Tensor Completion Objective and Gradient }
In this section, we specialize the computation of the objective and Riemannian gradient in the HT format to the case where the Euclidean gradient of the objective function is sparse, in particular for tensor completion. This will allow us to scale our method to high dimensions in a straightforward fashion as opposed to the inherently dense considerations in Algorithm \ref{alg:riemmanngrad}. Here for simplicity, we suppose that our dimension tree $T$ is \emph{complete}, that is a full binary tree up to level $\text{depth}(T)-1$ and all of the leaves at level $\text{depth}(T)$ are on the leftmost side of $T$, as in Figure~\ref{fig:dimtree}. This will ease the exposition as well as allow for a more efficient implementation compared to a noncomplete tree. 

We consider a separable, smooth objective function on the HT manifold,
\begin{align}
\label{eq:separableopt}
\hat{f}(x) = f(\phi(x)) =  \sum_{\mathbf{i} \in \Omega} f_{\mathbf{i}}( \phi(x)_{\mathbf{i}} ),
\end{align}
where $f_{\mathbf{i}} : \mathbb{R} \to \mathbb{R}$ is a smooth, single variable function. For the least-squares tensor completion problem, $f_\mathbf{i}(a) = \frac{1}{2}(a - b_{\mathbf{i}})^2$.

We denote $\mathbf{i} = (i_1, i_2, \dots, i_d)$ and let $\mathbf{i}_{t}$ be the subindices of $\mathbf{i}$ indexed by $t \in T$. In this section, we also use the Matlab notation for indexing in to matrices, i.e., $A(m,n)$ is the $(m,n)$th entry of $A$, and similarly for tensors. Let $K = \max_{t \in T} k_t$. 
\subsubsection{Objective function}
With this notation in mind, we write each entry of $P_{\Omega} \phi(x)$, indexed by $\mathbf{i} \in \Omega$, as 
\begin{align*}
(P_{\Omega} \phi(x))(\mathbf{i}) &= \sum_{r_l = 1}^{k_{t_l}} \sum_{r_r = 1}^{k_{t_r}} (U_{t_l})(\mathbf{i}_{t_l}, r_l) \cdot (U_{t_r})(\mathbf{i}_{t_r},r_r) \cdot B_{\troot}(r_l,r_r), \quad \text{where } t = \troot.
\end{align*}
Each entry of $U_{t_l}, U_{t_r}$ can be computed by applying the recursive formula (\ref{eq:phix}), i.e.,
\begin{align*}
U_{t}(\mathbf{i}_t,r) = \sum_{r_{l}=1}^{k_{t_l}} \sum_{r_{r}=1}^{k_{t_r}} (U_{t_l})(\mathbf{i}_{t_l},r_l) \cdot (U_{t_r})(\mathbf{i}_{t_r},r_{r}) \cdot \tensor{B}_t(r_l, r_r, r)
\end{align*}
with the substitutions of $t\to t_l, t_r$ as appropriate. 

At each node $t \in T$, we perform at most $K^3$ operations and therefore the computation of $P_\Omega \phi(x)$ requires at most $2|\Omega|dK^3$ operations. The least squares objective, $\frac{1}{2} \| P_{\Omega} \phi(x) - b \|_2^2$, can be computed in $|\Omega|$ operations. 

\subsubsection{Riemannian gradient}
The Riemannian gradient is more involved, notation-wise, to derive explicitly compared to the objective, so in the interest of brevity we only concentrate on the recursion for computing $\delta U_1$ below.

We let $\tensor{Z} = \nabla_{\phi(x)} f(\phi(x))$ denote the Euclidean gradient of $f(\tensor{X})$ evaluated at $\tensor{X} = \phi(x)$, which has nonzero entries $\tensor{Z}(\mathbf{i})$ indexed by $\mathbf{i} \in \Omega$. By expanding out (\ref{eq:adjpartialderivatives}), for each $\mathbf{i} \in \Omega$, $\delta U_{t_l}$ evaluated at the root node with coordinates $\mathbf{i}_{t_l}, r_{l}$ for $r_{l} = 1, \dots, k_{t_l}$ is
\begin{align*}
\delta U_{t_l}(\mathbf{i}_{t_l}, r_{l}) = \sum_{\mathbf{i} = (\mathbf{i}_{t_l}, \mathbf{i}_{t_r}) \in \Omega} \tensor{Z}(\mathbf{i}) \sum_{r_r = 1}^{k_{t_r}} U_{t_r}(\mathbf{i}_{t_r}, r_r) B_{\troot}(r_{l}, r_{r}), \quad \text{ where } t = \troot.
\end{align*}
For each $t \in T \setminus L \cup \troot$, we let $\widetilde{\delta U_{t_l}}$ denote the length $k_{t_l}$ vector, which depends on $\mathbf{i}_{t}$, satisfying, for each $\mathbf{i} \in \Omega, r_{l} = 1, ..., k_{t_l}$,
\begin{align*}
(\widetilde{\delta U_{t_l}})(\mathbf{i}_t, r_{t_l}) = \sum_{r_{r} = 1}^{k_{t_r}} \sum_{r_t = 1}^{k_t} U_{t_r}(\mathbf{i}_{t_r}, r_{r}) \tensor{B}_{t}(r_{l}, r_{r}, r_t)\widetilde{\delta U_{t}}(\mathbf{i}_{t}, r_{t}).
\end{align*}

This recursive construction above, as well as similar considerations for the right children for each node, yield Algorithm \ref{alg:riemmgradsparse}. For each node $t \in T \setminus \troot $, we perform $3|\Omega|K^3$ operations and at the root where we perform $3|\Omega|K^2$ operations. The overall computation of the Riemannian gradient requires at most $6d|\Omega|K^3$ operations, when $T$ is complete, and a negligible $O(dK)$ additional storage to store the vectors $\widetilde{\delta U_t}$ for each fixed $\mathbf{i} \in \Omega$. The computations above are followed by componentwise orthogonal projection on to $\horizspace$, which requires $O(d(NK^2 + K^4))$ operations and are dominated by the $O(d|\Omega|K^3)$ time complexity when $|\Omega|$ is large.

Therefore for large $|\Omega|$, each evaluation of the objective, with or without the Riemannian gradient, requires $O(d |\Omega|K^3)$ operations. Since $f(\tensor{X})$ exhibits this separable structure and the parameters $x = (U_t, \tensor{B}_t)$ are typically very small, it is straightforward to compute the objective and its gradient in an embarrassingly parallel manner for very large problems. 

By comparison, the gradient in the Tucker tensor completion case \cite{tuckercompletion} requires $O(d(|\Omega| + N)K^d + K^{d+1} )$ operations, which scales much more poorly when $d\ge 4$ compared to using Algorithm \ref{alg:riemmgradsparse}. This discrepancy is a result of the structural differences between Tucker and Hierarchical Tucker tensors, the latter of which allows one to exploit  additional low-rank behaviour of the core tensor in the Tucker format.

In certain situations, when say $|\Omega| = pN^d$ for some $p \in [10^{-3},1]$ and $d$ is sufficiently small, say $d = 4, 5$, it may be more efficient from a computer hardware point of view to use the dense linear algebra formulation in \Cref{alg:riemmanngrad} together with an efficient dense linear algebra library such as BLAS, rather than \Cref{alg:riemmgradsparse}. The dense formulation requires $O(N^d K)$ operations when $T$ is a balanced tree, which may be smaller than the $O(d|\Omega|K^3)$ operations needed in this case. 

\begin{algorithm}
\caption{ Objective \& Riemannian gradient for separable objectives }
\label{alg:riemmgradsparse}
\begin{algorithmic}
\REQUIRE $x = (U_t, \tensor{B}_t)$ parameter representation of the current point
\STATE $f_x \gets 0$, $\delta U_t, \delta \tensor{B}_t \gets 0$, $\widetilde{\delta U_t} \gets 0$
\FOR {$\mathbf{i} \in \Omega$}
\FOR {$t \in T \setminus L$, visiting children before their parents }
\FOR{$z = 1, 2, \dots, k_t$ }
\STATE $U_t(\mathbf{i}_t,z) \gets \sum_{w=1}^{k_l} \sum_{y=1}^{k_r} (U_{t_l})(\mathbf{i}_{t_l}, w) \cdot (U_{t_r})(\mathbf{i}_{t_r}, y) \cdot \tensor{B}_t(w, y, z) $
\ENDFOR
\ENDFOR
\STATE $f_x \gets f_x + f_{\mathbf{i}}( U_{\troot}(\mathbf{i}) )$
\STATE $\widetilde{\delta U_{\troot}} \gets \nabla f_{\mathbf{i}}( U_{\troot}(\mathbf{i}) )$
\FOR {$t \in T \setminus L$, visiting parents before their children }
\FOR {$w = 1, \dots, k_{t_l}$, $y = 1, \dots, k_{t_r}$, $z = 1, \dots, k_{t}$ } 
\STATE $\delta \tensor{B}_t (w, y, z) \gets \delta \tensor{B}_t(w, y, z) + \widetilde{\delta U_t}( z )  \cdot (U_{t_l})( \mathbf{i}_{t_l} , w ) \cdot (U_{t_r}) (\mathbf{i}_{t_r}, y) $
\ENDFOR
\FOR{$w =1,\dots, k_{t_l}$}
\STATE $\widetilde{\delta  U_{t_l} }( w ) \gets \sum_{y=1}^{k_{t_r}} \sum_{z=1}^{k_t} (U_{t_r})(\mathbf{i}_{t_r}, y) \cdot \tensor{B}_{t}(w, y, z) \cdot \widetilde{ \delta U_t } (z) $
\ENDFOR
\FOR{$y=1,\dots,k_{t_r}$}
\STATE $\widetilde{\delta U_{t_r}}( y ) \gets \sum_{w=1}^{k_{t_l}} \sum_{z=1}^{k_t} (U_{t_l})(\mathbf{i}_{t_l}, w) \cdot \tensor{B}_{t}(w, y, z) \cdot \widetilde{ \delta U_t } (z) $
\ENDFOR
\ENDFOR
\FOR{$t \in L$ }
\FOR{$z = 1, 2, \dots, k_t$ }
\STATE $\delta U_t( \mathbf{i}_t, z) \gets \delta U_t( \mathbf{i}_t, z) + \widetilde{\delta U_t}(z) $
\ENDFOR
\ENDFOR
\STATE Project $(\delta U_t, \delta \tensor{B}_t)$ componentwise on to $\horizspace$ 
\ENDFOR
\RETURN $f(x) \gets f_x$, $\nabla^R f(x) \gets (\delta U_t, \delta \tensor{B}_t)$
\end{algorithmic}
\end{algorithm}

\section{Optimization}
\label{sec:optimization}
\subsection{Reorthogonalization as a retraction}
The exponential mapping on a Riemannian manifold captures the notion
of ``minimal distance'' movement in a particular tangent
direction. Although it has many theoretically desirable properties,
the exponential mapping is often numerically difficult or expensive to
compute as it involves computing matrix exponentials or
solving ODEs. The strict, distance-minimizing properties of the
exponential mapping can be relaxed, while still preserving
algorithmic convergence, resulting in the notion of a
\emph{retraction} on a manifold.

\begin{defn} 
\label{def:retraction}
A retraction on a manifold $\mathcal{N}$ is a smooth mapping $R$ from
the tangent bundle $\mathcal{T}\mathcal{N}$ onto $\mathcal{N}$ with the following properties: Let
$R_x$ denote the restriction of $R$ to $\mathcal{T}_x \mathcal{N}$.
\begin{itemize}
\item $R_x(0_x) = x$, where $0_x$ denotes the zero element of $\mathcal{T}_x
  \mathcal{N}$
\item With the canonical identification $\mathcal{T}_{0_x} \mathcal{T}_x \mathcal{N}
  \simeq \mathcal{T}_x \mathcal{N}$, $R_x$ satisfies
\begin{align*}
  DR_x(0_x) = \text{id}_{\mathcal{T}_x \mathcal{N}}
\end{align*}
where $\text{id}_{\mathcal{T}_x \mathcal{N}}$ denotes the identity mapping on
$\mathcal{T}_x \mathcal{N}$ (Local rigidity condition).
\end{itemize}
\end{defn}
A retraction approximates the action of the exponential mapping to
first order and hence much of the analysis for algorithms utilizing
the exponential mapping can also be immediately carried over to those
using retractions.

A computationally feasible retraction on the HT parameters is given by
QR- or square-root-based reorthogonalization (\ref{app:sqrt}). The QR-based orthogonalization of (potentially nonorthogonal) parameters $x = (U_t, \tensor{B}_t)$, denoted $QR(x)$, is given in \Cref{alg:qr} \cite[Alg. 3]{htuckersvd}. 

\begin{proposition} \label{prop:retraction} Given $x \in \paramspaceclean$, $\eta \in \mathcal{T}_x \paramspaceclean$, let $\text{QR}(x)$ be the QR-based orthogonalization defined in \Cref{alg:qr}. Then $R_x(\eta) := \text{QR}(x + \eta) $ is a retraction on $\paramspaceclean$.
\end{proposition}

We refer to \ref{proof:qrretraction} for the proof of this
proposition. 

As before for the Riemannian metric, we can treat the retractions on
the HT parameter space as implicitly being retractions on the quotient
space as outlined below. \\
Since $R_x(\eta)$ is a retraction on the parameter space
$\paramspaceclean$, and our horizontal space is invariant under the
Lie group action, by the discussion in
\cite[4.1.2]{optmatrixmanifold}, we have the following

\begin{proposition}
The mapping 
\begin{align*}
  R_{\pi(X)}(z) = \pi(R_X(Z))
\end{align*}
is a retraction on $\quotspace$, where $R_X(Z)$ is the QR or
square-root based retraction (\ref{app:sqrt}) previously defined on
$\paramspaceclean$, $\pi : \paramspaceclean \to \quotspace$ is the
projection operator, and $Z$ is the horizontal lift at $X$ of the
tangent vector $z$ at $\pi(X)$.
\end{proposition}
\begin{algorithm}
\caption{QR-based orthogonalization}
\label{alg:qr}
\begin{algorithmic}
\REQUIRE HT parameters $x = (U_t, \tensor{B}_t)$
\RETURN $y = (V_t, \tensor{C}_t)$ orthogonalized parameters such that $\phi(x) = \phi(y)$
\FOR{$t \in L$}
\STATE $Q_t R_t = U_t$, where $Q_t$ is orthogonal and $R_t$ is upper triangular
\STATE $V_t \gets Q_t$
\ENDFOR
\FOR{$t \in T \setminus (L \cup \troot)$, visiting children before their parents}
\STATE $Z_t \gets R_{t_l} \times_1 R_{t_r} \times_2 \tensor{B}_t$
\STATE $Q_t R_t = Z_t^{(1,2)}$, where $Q_t$ is orthogonal and $R_t$ is upper triangular
\STATE $\tensor{C}_t \gets (Q_t)_{(1,2)}$
\ENDFOR
\STATE $C_{\troot} \gets R_{(\troot)_l} B_{\troot} R_{(\troot)_r}^T$
\end{algorithmic}
\end{algorithm}

\subsection{Vector transport}
Now that we have a method for ``moving'' in a particular direction along
the HT manifold, we need a means of mapping tangent vectors from one
point to another. For this purpose, we use the notion of vector
transport, which relaxes the isometry constraints of parallel transport to decrease computational complexity. Even
though we make this approximation, we still enjoy increased
convergence rates compared to steepest descent (see \cite[Sec.
8.1.1]{optmatrixmanifold} for more details).

Since our parameter space $\paramspaceclean$ is a subset of Euclidean
space, given a point $x \in \paramspaceclean$ and a horizontal vector $\eta_x \in \horizspace$, we take our vector transport $\mathcal{T}_{x,\eta_x} : \horizspace \to \mathcal{H}_{R_{x}(\eta_x)} \paramspaceclean$ of the vector $\xi_x \in \horizspace$ to be
\begin{align*}
  \mathcal{T}_{x,\eta_x} \xi_x := P_{R_{x}(\eta_x)}^h \xi_x
\end{align*} 
where $P^h_x$ is the component-wise projection onto the horizontal
space at $x$ \cite[Sec. 8.1.4]{optmatrixmanifold}. This mapping is
well defined on $\quotspace$ since $\horizspace$ is invariant under
$\theta$, and induces a vector transport on the quotient space.

\subsection{Smooth optimization methods}
Now that we have established the necessary components for manifold
optimization, we present a number of concrete optimization algorithms for solving 
\begin{align*}
\min_{x \in \paramspaceclean} f(\phi(x)).
\end{align*}

\subsubsection{First-order methods}
\label{sec:first-order-methods}
Given the expressions for the Riemannian gradient and retraction, it
is straightforward to implement the classical Steepest Descent
algorithm with an Armijo line search on this Riemannian manifold, specialized from the general Riemannian manifold case \cite{optmatrixmanifold} to the HT manifold in \Cref{alg:cg}. This
algorithm consists of computing the Riemannian gradient, followed by a
line search, HT parameter update, and a reorthogonalization. Since
this algorithm has a poor convergence rate, we rely on more
sophisticated optimization algorithms such as the nonlinear conjugate
gradient method as outlined in Algorithm \ref{alg:cg}. 

Here $g_i$ denotes the Riemannian gradient at iteration $i$ of the algorithm, $p_i$ is the search direction for the optimization method, and $\alpha_i$ is the step length.

We choose
the Polak-Ribiere approach
\begin{align*}
  \beta_i = \dfrac{\langle g_i, g_i -
    \mathcal{T}_{x_{i-1},\alpha_{i-1} p_{i-1}}(g_{i-1}) \rangle}{\langle
    g_{i-1}, g_{i-1} \rangle}
\end{align*}
to compute the CG-parameter $\beta_i$, so that the search direction $p_i$ satisfies 
\begin{align*}
p_i = - g_i + \beta_i \mathcal{T}_{x_{i-1},\alpha_{i-1} p_{i-1}} p_{i-1}
\end{align*}
and $p_1 = - g_1$.
\renewcommand{\algorithmiccomment}[1]{\hfill$\triangleright$ #1}
\begin{algorithm}
\caption{General Nonlinear Conjugate Gradient method for minimizing a function $f$ over $\htuckspaceclean$}
\label{alg:cg}
\begin{algorithmic}
\REQUIRE Initial guess $x_0 = (U_t, \tensor{B}_t)$, $0 < \sigma < 1$ sufficient decrease parameter for the Armijo line search, $0 < \theta < 1$ step size decrease parameter, $\gamma > 0$ CG restart parameter
\STATE $p_{-1} \gets 0$
\STATE $i \gets 0$
\FOR{$\itrvar = 0,1,2,\dots$ until convergence}
\STATE $\mathbf{X}_i \gets \phi(x_i)$
\STATE $f_i \gets f(\mathbf{X}_i)$
\STATE $g_i \gets \nabla^R \hat{f}(x_i)$ 
\COMMENT{ Riemannian gradient of $\hat{f}(x)$ at $x_i$ }
\STATE $s_{i} \gets \mathcal{T}_{x_{i-1},\alpha_{i-1} p_{i-1}} \alpha_{i-1}p_{i-1}$ 
\COMMENT{ Vector transport the previous search direction }
\STATE $y_i \gets g_i -  \mathcal{T}_{x_{i-1},\alpha_{i-1} p_{i-1}} g_{i-1} $
\STATE $L_i \gets y_i^T s_i / \|s_i\|^2$ 
\COMMENT{ Lipschitz constant estimate }
\STATE $p_i \gets -g_i + \beta_i \mathcal{T}_{x_{i-1},\alpha_{i-1} p_{i-1}} p_{i-1}$ 
\IF {$\langle p_i, g_i \rangle > -\gamma$ }
\STATE  $p_i = -g_i$ 
\COMMENT{ Restart CG direction }
\ENDIF
\IF {$y_i^T s_i > 0$}
\STATE $\alpha \gets - g_i^T p_i/(L_i \|p_i\|_2^2) $ 
\ELSE
\STATE $\alpha \gets \alpha_{i-1} $
\ENDIF
\STATE Find $m \in \mathbb{Z}$ such that $\alpha_i = \alpha \theta^m$ and
\STATE \qquad $f(x_i +  \alpha_i p_i) - f_i \le \sigma  \alpha_i g_i^T p_i$
\STATE \qquad $f(x_i +  \alpha_i p_i) < \min\{ f(x_i +  \alpha_i \theta p_i), f(x_k + \alpha_i \theta^{-1} p_i) \}$
\COMMENT{Find a quasi-optimal minimizer }
\STATE $x_{i+1} \gets R_{x_i}( \alpha_i p_i) $ 
\COMMENT{Reorthogonalize}
\STATE $i \gets i+1$
\ENDFOR
\end{algorithmic}
\end{algorithm}

\subsubsection{Line search}
As any gradient based optimization scheme, we need a good initial step
size and a computationally efficient line search. Following
\cite{Shi:2005gz}, we use a variation of the limited-minimization line
search approach to set the initial step length based on the previous
search direction and gradient that are vector transported to the
current point---i.e, we have
\begin{align*}
  s_i &= \mathcal{T}_{x_{i-1},\alpha_{i-1} p_{i-1}} \alpha_{i-1} p_{i-1}\\
  y_i &= g_i - \mathcal{T}_{x_{i-1},\alpha_{i-1} p_{i-1}} g_{i-1}.
\end{align*}
In this context, $s_i$ is the manifold analogue for the Euclidean difference between iterates, $x_i - x_{i-1}$ and $y_i$ is the manifold analogue for the difference of gradients between iterates, $g_i - g_{i-1}$, which are standard optimization quantities in optimization algorithms set in $\mathbb{R}^{n}$. 

Our initial step size for the direction $p_i$ is given as
\begin{align*}
\alpha_0 = - g_i^T p_i / (L_i \|p_i\|_2^2)
\end{align*}
where $L_i = y_i^T s_i / \|s_i\|_2^2$ is the estimate of the Lipschitz
constant for the gradient \cite[Eq. 16]{shi2004convergence}. Because we are operating in the HT
parameter space, the above computations require $O(\text{dim}(M)) = O(dNK + (d-1)K^3)$ operations, much less than the $2|\Omega|(d+1)K^d$ operations used in \cite{tuckercompletion} to initialize their line search. We justify this choice because
we are working on large-scale problems where we have to limit the
number of operations in the full tensor space, even when $|\Omega|$ is small. 

Moreover, computing the gradient is much more expensive than
evaluating the objective. For this reason, we use a simple Armijo-type
back-/forward-tracking approach that only involves function
evaluations and seeks to minimize the $1D$ function $f(x + \alpha
p_i)$ quasi-optimally, i.e., to find $m \in \mathbb{Z}$ such that
$\alpha = \theta^m \alpha_0 $ for $\sigma > 0$
\begin{align}
\label{eq:armijo}
f(x_i + \alpha p_i) - f(x_i) \le \sigma \alpha g_i^T p_i  \\
f(x_i + \alpha p_i) \le \min\{f(x_i + \theta\alpha p_i),f(x_i +
\theta^{-1} \alpha p_i)\} \notag
\end{align}
so $\alpha \approx \alpha^* = \argmin_{\alpha} f(x_i + \alpha p_i)
$ in the sense that increasing or decreasing $\alpha$ by a factor of $\theta$ will increase $f(x_i + \alpha p_i)$. After the first few iterations of our optimization procedure, we
observe empirically that our line search only involves two or three
additional function evaluations to verify the second inequality in
(\ref{eq:armijo}), i.e., our initial step length $\alpha_0$ is
quasi-optimal.

Because $\phi(R_x(\alpha \eta)) = \phi(x + \alpha\eta)$ for any $x
\in \paramspaceclean$ and horizontal vector $\eta$, where $R_x$ is
either the QR or square-root based retraction, Armijo linesearches do
not require reorthogonalization, which further reduces computational
costs.

\subsubsection{Gauss-Newton Method}
\label{sec:gaussnewton}
Because of the least-squares structure of our tensor completion
problem (\ref{eq:tensorcompletion}), we can approximate the Hessian by the Gauss-Newton Hessian
\begin{align*}
  H_{GN}:= D\phi^*(x) D\phi(x) : \horizspace \to \horizspace.
\end{align*}
Note that we do not use the ``true'' Gauss-Newton Hessian, $D\phi^*(x)
P_{\Omega}^*P_{\Omega} D\phi(x)$, for the tensor completion case, since for even
moderate subsampling ratios, $P_{\Omega}^* P_{\Omega}$ is close to the zero
operator and this Hessian is very poorly conditioned as a result.

Since $D\phi(x) : \horizspace \to \mathcal{T}_{\phi(x)} \mathcal{H}$ is an
isomorphism, it is easy to see that $H_{GN}$ is symmetric and positive
definite on $\horizspace$.  The solution to the Gauss-Newton equation
is then
\begin{align*}
H_{GN} \xi = - \nabla^R f(x)
\end{align*}
for $\xi \in \horizspace$. 
We can simplify the computation of $H_{GN}$ by exploiting the recursive structure of $D\phi^*(x)$ and $D\phi(x)$, thereby avoiding intermediate vectors of size $\fullspace$ in the process. We write at the root 
\begin{align*}
  \delta U'_{t_l} &= (I - U_{t_l} U_{t_l}^T) \dfrac{\partial U_t}{\partial U_{t_l}}^T D\phi(x)[\xi], \\
  \delta U'_{t_r} &= (I - U_{t_r} U_{t_r}^T)  \dfrac{\partial U_t}{\partial U_{t_r}}^T D\phi(x)[\xi], \\
  \delta B'_t &= \dfrac{\partial U_t}{\partial \tensor{B}_t}^T D\phi(x)[\xi]\end{align*}
where
\begin{align*}
D\phi(x)[\xi] = \delta U_{t_l} \times_1
  U_{t_r} \times_2  B_{t} + U_{t_l} \times_1 \delta U_{t_r} \times_2
  B_{t} + U_{t_l} \times_{1} U_{t_r} \times_2 \delta
  B_{t}, \quad t = \troot.
\end{align*}
In the above expression, $D\phi(x)$ is horizontal, so that for each $t
\in T \setminus \troot$, $\delta U_t$ is perpendicular to $U_t$ (\ref{eq:orthogonalitycontraints}). A straightforward
computation simplifies the above expression to
\begin{align*}
  \delta U'_{t_l} = \left(\dfrac{\delta U_{\troot}}{\delta U_{t_l}}\right)^T \delta U_{\troot} &= \delta U_{t_l} B_{\troot} B_{\troot}^T \\
  &:= \delta U_{t_l} G_{t_l},\\
  \delta U'_{t_r} = \left(\dfrac{\delta U_{\troot}}{\delta U_{t_r}}\right)^T \delta U_{\troot}  &= \delta U_{t_r} B_{\troot}^T B_{\troot} \\
  &:= \delta U_{t_r} G_{t_r},\\
  \delta B'_t = \left(\dfrac{\delta U_{\troot}}{\delta B_{\troot}}\right)^T \delta
  U_{\troot} &= \delta B_{\troot}.
\end{align*}
This expression gives us the components of the horizontal vector
$\delta U'_{t_l}, \delta U'_{t_r}$ sent to the left and right
children, respectively, as well as the horizontal vector $\delta
B'_t$.

We proceed recursively by considering a node $t \in T \setminus L \cup
\troot$ and let $\delta U_t G_t$ be the contribution from the parent
node of $t$. By applying the adjoint partial derivatives, followed by
an orthogonal projection on to $\horizspace$, we arrive at a simplified form for the
Gauss-Newton Hessian
\begin{align*}
  P_{U_{t_l}}^{\perp} \dfrac{\delta U_{t}}{\delta U_{t_l}}^T \delta U_{t} G_t &= \langle \delta U_{t_l} \times_1 G_t \times_3 \tensor{B}_t, \tensor{B}_t \rangle_{(2,3),(2,3)} \\
  &:= \delta U_{t_l} G_{t_l},\\
  P_{U_{t_r}}^{\perp} \dfrac{\delta U_{t}}{\delta U_{t_r}}^T \delta U_{t} G_t &= \delta U_{t_r} G_{t_r}, \\
  P_{B_t^{(1,2)}}^{\perp}  \dfrac{\delta U_{t}}{\delta \tensor{B}_t}^T
  \delta U_{t} G_t &= \delta G_t \times_3 \tensor{B}_t .
\end{align*}
In these expressions, the matrices $G_t$ are the Gramian matrices
associated to the HT format, initially introduced in \cite{htuckersvd}
and used for truncation of  a general tensor to the HT format as in
\cite{htucktoolboxthesis}. They satisfy, for $x = (U_t, \tensor{B}_t) \in
\paramspaceclean$,
\begin{align}
\label{eq:gramian}
G_{\troot} & = 1\\
G_{t_l} & = \langle G_t \times_3 \tensor{B}_t, \tensor{B}_t \rangle_{(2,3),(2,3)} \notag\\
G_{t_r} & = \langle G_t \times_3 \tensor{B}_t, \tensor{B}_t \rangle_{(1,3),(1,3)}, \notag
\end{align}
i.e., the same recursion as $G_t$ in the above derivations. Each $G_t$
is a $k_t \times k_t$ symmetric positive definite matrix (owing to the
full rank constraints of the HT format) and also satisfies
\begin{align}
\label{eq:grameig}
\lambda_j(G_t) = \sigma_j(X^{(t)})^2
\end{align}
where $\lambda_j(A)$ is the $j$th eigenvalue of the matrix $A$ and $\sigma_j(A)$ is the $j$th singular value of $A$. 

Assuming that each $G_t$ is well conditioned, applying the inverse of
$H_{GN}$ follows directly, summarized in \Cref{alg:hgninv}.
\begin{algorithm}
\caption{$H_{GN}^{-1} \zeta$} 
\begin{algorithmic}
\label{alg:hgninv}
\REQUIRE Current point $x = (U_t, \tensor{B}_t)$, horizontal vector $\zeta = (\delta U_t, \delta \tensor{B}_t)$
\STATE Compute $(G_t)_{t \in T}$ using (\ref{eq:gramian})
\FOR{ $t \in T \setminus \troot$ }
\IF{ $t \in L$ }
\STATE $\widetilde{\delta U_t} \gets \delta U_t G_t^{-1}$
\ELSE
\STATE $\widetilde{\tensor{\delta B}_t} \gets G_t^{-1} \times_3 \tensor{\delta B}_t $
\ENDIF
\ENDFOR
\RETURN $H_{GN}^{-1} \zeta \gets (\widetilde{\delta U_t}, \widetilde{\delta \tensor{B}_t}) $
\end{algorithmic}
\end{algorithm} 
For the case where our solution HT tensor exhibits
\emph{quickly}-decaying singular values of the matricizations, as is
typically the assumption on the underlying tensor, the Gauss-Newton
Hessian becomes poorly conditioned as the iterates converge to the
solution, owing to (\ref{eq:grameig}). This can be remedied by
introducing a small $\epsilon > 0$ and applying $(G_t + \epsilon
I)^{-1}$ instead of $G_t^{-1}$ in \Cref{alg:hgninv} or by applying
$H_{GN}^{-1}$ by applying $H_{GN}$ in a truncated PCG method. For
efficiency purposes, we find the former option
preferable. Alternatively, we can also avoid ill-conditioning via
regularization, as we will see in the next section.

\begin{remark}
  We note that applying the inverse Gauss-Newton Hessian to a tangent
  vector is akin to ensuring that the projection on to the horizontal
  space is orthogonal, as in \cite[6.2.2]{htuckgeom}. Using this
  method, however, is much faster than the previously proposed method,
  because applying \Cref{alg:hgninv} only involves matrix-matrix
  operations on the small parameters, as opposed to operations
  on much larger intermediate matrices that live in the spaces between
  the full tensor space and the parameter space.
\end{remark}

\subsection{Regularization} 
\label{subsec:gramregularization}
In the tensor completion case, when there is little data available,
interpolating via the Gauss-Newton method is susceptible to
overfitting if one chooses the ranks $(k_t)_{t \in T}$ for the
interpolated tensor too high. In that case, one can converge to
solutions in $\text{null}(P_{\Omega})$ that try leave the current
manifold, associated to the ranks $(k_t)_{t \in T}$, to another nearby
manifold corresponding to higher ranks. This can lead to degraded
results in practice, as the actual ranks for the solution tensor are
almost always unknown. One can use cross-validation techniques to
estimate the proper internal ranks of the tensor, but we still need to
ensure that the solution tensor has the predicted ranks for this approach to be successful -- i.e., the iterates $x$ must stay away from the boundary of $\htuckspaceclean$.

To avoid our HT iterates converging to the manifold boundary, we
introduce a regularization term on the singular values of the HT
tensor $\phi(x) = \tensor{X}$. To accomplish this, we exploit the hierarchical
structure of $\tensor{X}$ and specifically the property of the Gramian
matrices $G_t$ in (\ref{eq:grameig}) to ensure that \emph{all}
matricizations of $\tensor{X}$ remain well-conditioned \emph{without}
having to perform SVDs on each matricization $X^{(t)}$. The latter approach 
would be prohibitively expensive when $d$ or $N$ are even moderately large. 

Instead, we penalize the growth of the Frobenius norm of $X^{(t)}$ and
$(X^{(t)})^\dagger$, which indirectly controls the largest and smallest singular
values of $X^{(t)}$.  We implement this regularization via the Gramian matrices in
the following way.
From (\ref{eq:grameig}), it follows that $\text{tr}(G_t) = \|G_t\|_* = \|X^{(t)}\|^2_F$
and likewise $\text{tr}(G_t^{-1}) = \|G_t^{-1}\|_* =
\|(X^{(t)})^\dagger\|^2_F$.
Our regularizer is then
\begin{align*}
  R((\tensor{B}_{t'})_{t' \in T}) = \sum_{t \in T} \text{tr}(G_t) + \text{tr}(G_t^{-1}).
\end{align*}
A straightforward calculation shows that for $\fancyA \in \liegroup$
\begin{align*}
  (G_t)_{t \in T, x} = (A_t^T G_t A_t)_{t \in T, \theta_{\fancyA}(x)}
\end{align*}
for $A_t$ orthogonal. Therefore, our regularizer $R$ is well-defined on
the quotient manifold in the sense that it is $\theta-$invariant on
the parameter space $\paramspaceclean$. This is the same
regularization term considered in \cite{tuckercompletion}, used for
(theoretically) preventing the iterate from approaching the boundary
of $\htuckspaceclean$. In our case, we can leverage the structure of the
Gramian matrices to implement this regularizer in a computationally
feasible way.

Since in the definition of the Gramian matrices (\ref{eq:gramian}),
$G_t$ is computed recursively via tensor-tensor contractions (which
are smooth operations), it follows that the mapping $g : (\tensor{B}_t)_{t \in
  T \setminus L} \to (G_t)_{t \in T}$ is smooth. In order to compute its
derivatives, we consider a node $t \in T \setminus \troot$ and
consider the variations of its left and right children, i.e.,
\begin{align}
\label{eq:dGramian}
\delta G_{t_r} &= \dfrac{\partial G_{t_r}}{\partial \tensor{B}_t} \delta \tensor{B}_t + \dfrac{\partial G_{t_r}}{\partial G_t} \delta G_t \\
\delta G_{t_l} &= \dfrac{\partial G_{t_l}}{\partial \tensor{B}_t} \delta \tensor{B}_t +
\dfrac{\partial G_{t_l}}{\partial G_t} \delta G_t . \notag
\end{align}
We can take the adjoint of this recursive formulation, and thus obtain
the gradient of $g$, if we compute the adjoint partial derivatives in
(\ref{eq:dGramian}) as well as taking the adjoint of the recursion
itself. To visualize this process, we consider the relationship
between input variables and output variables in the recursion as a
series of small directed graphs.
\captionsetup[subfigure]{labelformat=empty}

\begin{figure}[H]
\centering
\subfloat[]{ 
\includegraphics[scale=0.5]{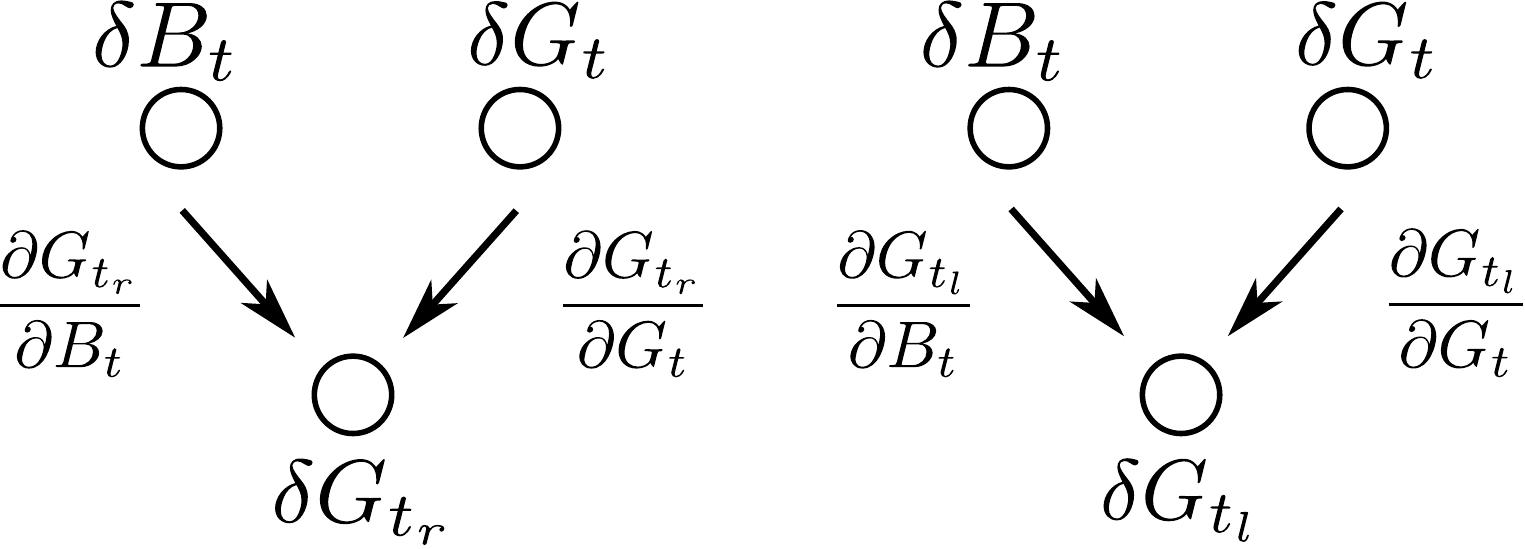} 
} 
\caption{ Forward Gramian derivative map } 
\end{figure}
These graphs can be understood in the context of Algorithmic
Differentiation, whereby the forward mode of this derivative map
propagates variables up the tree and the adjoint mode propagates
variables down the tree and adds (accumulates) the contributions of
the relevant variables. 
\begin{figure}[H]
\centering
\subfloat[]{ 
\includegraphics[scale=0.5]{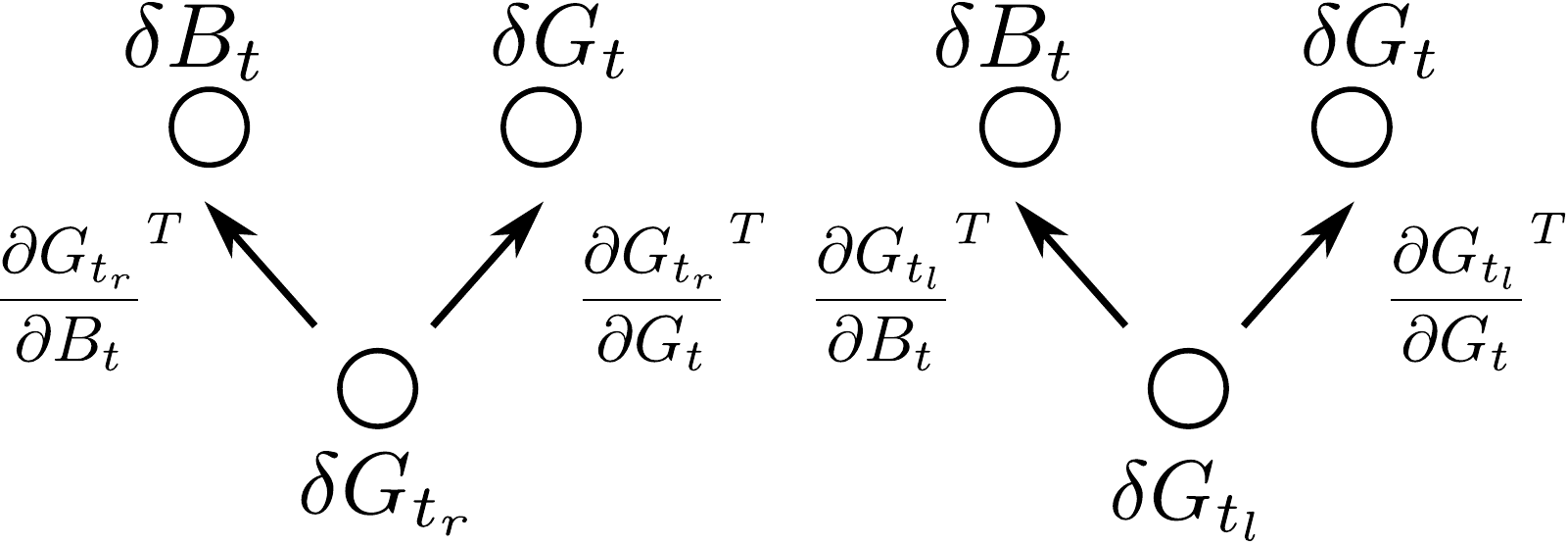}
} 
\caption { Adjoint Gramian derivative map }
\end{figure}
Since we only consider tangent vectors $\delta B_t$ that are in the
horizontal space at $x$, each extracted component is projected on to
$(B_t^{(1,2)})^{\perp}$. We summarize our results in the following
algorithms.

\begin{algorithm}
\caption{$Dg[\delta \tensor{B}_t]$} 
\begin{algorithmic}
\label{alg:dGramian}
\REQUIRE Current point $x = (U_t, \tensor{B}_t)$, horizontal vector $dx = (\delta U_t, \delta \tensor{B}_t)$
\STATE Compute $(G_t)_{t \in T}$ using (\ref{eq:gramian})
\STATE $\delta G_{\troot} \gets 0$
\FOR{ $t \in T \setminus L$, visiting parents before children }
\STATE $\delta G_{t_l} \gets \langle \delta G_t \times_3 \tensor{B}_t, \tensor{B}_t \rangle_{(2,3),(2,3)} + 2 \langle G_t \times_3 \delta \tensor{B}_t, \tensor{B}_t \rangle_{(2,3),(2,3)}$
\STATE $\delta G_{t_r} \gets \langle \delta G_t \times_3 \tensor{B}_t, \tensor{B}_t \rangle_{(1,3),(1,3)} + 2 \langle G_t \times_3 \tensor{\delta B_t}, \tensor{B}_t \rangle_{(1,3),(1,3)}$
\ENDFOR
\RETURN $Dg[\delta \tensor{B}_t] \gets (\delta G_t)_{t \in T} $
\end{algorithmic}
\end{algorithm} 
\begin{algorithm}
\caption{$Dg^*[\delta G_t] $} 
\begin{algorithmic}
\label{alg:dGramianadj}
\REQUIRE Current point $x = (U_t, \tensor{B}_t)$, Gramian variations $(\delta G_t)_{t \in T}$, $\delta G_{\troot} = 0$
\STATE Compute $(G_t)_{t \in T}$ using (\ref{eq:gramian})
\FOR{ $t \in T$ }
\STATE $\widetilde{\delta G_t} \gets \delta G_t$
\ENDFOR
\FOR{ $t \in T \setminus L$, visiting children before parents }
\STATE $\delta \tensor{B}_t \gets (\widetilde{\delta G_{t_l}} + \widetilde{\delta G_{t_l}}^T) \times_1 G_t \times_3 \tensor{B}_t + (\widetilde{\delta G_{t_r}} + \widetilde{\delta G_{t_r}}^T) \times_2 G_t \times_3 \tensor{B}_t$
\IF { $t \neq \troot$ }
\STATE $\delta \tensor{B}_t \gets (P_{B^{(1,2)}}^{\perp} \delta \tensor{B}_t^{(1,2)})_{(1,2)}$
\STATE $\widetilde{\delta G_{t}} \gets \delta G_t + \langle \widetilde{G_{t_l}} \times_1 \tensor{B}_t, \tensor{B}_t \rangle_{(1,2),(1,2)} + \langle \widetilde{G_{t_r}} \times_2 \tensor{B}_t, \tensor{B}_t \rangle_{(1,2),(1,2)} $
\ENDIF
\ENDFOR
\RETURN $Dg^*[\delta G_t] \gets (\delta \tensor{B}_t)_{t \in T} $
\end{algorithmic}
\end{algorithm} 
Applying \Cref{alg:dGramianadj} to the gradient of $R(\tensor{B}_t)$,
\begin{align*}
  \nabla R(\tensor{B}_t) = ( V_t (I_{k_t} - S_t^{-2}) V_t^T ),
\end{align*}
where $G_t = V_t S_t V_t^T$ is the eigenvalue decomposition of $G_t$, yields the Riemannian gradient of the
regularizer. Note that here, we avoid having to compute SVDs of
\emph{any} matricizations of the full data $\phi(x)$, resulting in a
method which is much faster than other tensor completion methods that
require the SVDs on tensors in $\fullspace$ \cite{prevwork:tuckercompletion}. Note that the cost of computing this
regularizer $R(\tensor{B}_t)$ and its gradient are almost negligible compared
to the cost of computing the objective and its Riemannian gradient.

Finally, we should also note that the use of this regularizer is not
designed to improve the recovery quality of problem instances with a
relatively large amount of data and is useful primarily in the case
where there is very little data so as to prevent overfitting, as we shall
see in the numerical results section.

\subsection{Convergence analysis}
Our analysis here follows from similar considerations in \cite[Sec. 3.6]{tuckercompletion}.
\begin{thm} 
  Let $\{x_i\}$ be an infinite sequence of iterates, with $x_i$ generated at iteration $i$,
  generated from \Cref{alg:cg} for the Gramian-regularized objective
  with $\lambda>0$
\begin{align*}
  f(x) = \dfrac{1}{2} \|P_{\Omega} \phi(x) - b\|_2^2 + \lambda^2 \sum_{t \in T \setminus \troot} \text{\rm{tr}}(G_t(x)) + \text{\rm{tr}}(G_t^{-1}(x)).
\end{align*}
Then $\lim_{i \to \infty} \|\nabla^R f(x_i) \| = 0$. 
\end{thm}
\begin{proof}
  To show convergence, we merely need to show that the iterates remain
  in a sequentially compact set, since any accumulation point of $\{ x_i \}$ is a critical point of $f$, by  \cite[Thm 4.3.1]{optmatrixmanifold}. But this follows because by construction, since  $f(x_i)
  \le f(x_0) := C^2$ for all $i$. Letting $\tensor{X}_i := \phi(x_i)$
\begin{align*}
  &\dfrac{1}{2} \| P_{\Omega} \phi(x_i) - b\|_2^2 + \lambda^2 \sum_{t \in T \setminus \troot} \text{tr}(G_t(x_i)) + \text{tr}(G_t^{-1}(x_i)) = \\
  & \dfrac{1}{2} \| P_{\Omega} \tensor{X}_i - b\|_2^2 + \lambda^2 \sum_{t \in
    T \setminus \troot} \|X_i^{(t)}\|_{F}^2 + \|(X_i^{(t)})^{\dagger}\|_{F}^2 \le C^2
\end{align*}
This shows, in particular, that 
\begin{align*}
  \lambda^2 \sum_{t \in T\setminus \troot} \|X_i^{(t)}\|_{F}^2 \le C^2 \quad \lambda^2
  \sum_{t \in T \setminus \troot} \|(X_i^{(t)})^{\dagger}\|_{F}^2 \le C^2
\end{align*}
and therefore we have upper and lower bounds on the maximum and
minimum singular values of $X_i^{(t)}$
\begin{align*}
  \sigma_{\max}(X_i^{(t)}) \le \|X_i^{(t)}\|_{F} \le C/\lambda \quad
  \sigma^{-1}_{\min}(X_i^{(t)}) \le \|(X_i^{(t)})^{\dagger}\|_{F} \le
  C/\lambda
\end{align*}
and therefore the iterates $X_i$ stay within the compact set 
\begin{align*}
  \mathcal{C} = \{ \tensor{X} \in \htuckspaceclean : \sigma_{\min}(X_k^{(t)})
  \ge \lambda/C, \sigma_{\max}(X_k^{(t)}) \le C/\lambda, t \in T
  \setminus \troot \}.
\end{align*}
One can show, as a modification of the proof in \cite[]{htuckgeom},
that $\hat{\phi} : \quotspace \to \htuckspaceclean$ is a
homeomorphism on to its image, so that $\hat{\phi}^{-1}(C)$ is compact
in $\quotspace$.
We can introduce a metric on $\quotspace$, which generates the
topology on the quotient space, as
\begin{align}
\label{eq:quotmetricfull}
d( \pi(x), \pi(y) ) &= \inf_{\mathcal{A},\mathcal{B} \in \liegroup }
\|\theta_{\mathcal{A}}(x) - \theta_{\mathcal{B}}(y) \|_T
\end{align}
where $\|x - y\|_{T} = \sum_{t \in L} \|U_t - V_t\|_{F} + \sum_{t \in
  T \setminus L} \|\tensor{B}_t - \tensor{C}_t\|_{F} $ is the natural metric on
$\paramspaceclean$ and $x = (U_t, \tensor{B}_t)$, $y = (V_t, \tensor{C}_t)$. 

Note that this pseudo-metric is a metric which generates the topology
on $\quotspace$ by \cite[Thm 2.1]{quotmetric} since $\{
\theta_{\mathcal{A}} \}_{\mathcal{A} \in \liegroup } $ is a group of
isometries acting on $\paramspaceclean$ and the orbits of the action
are closed by \cite[Thm 2]{htuckgeom}. Note that this metric is
equivalent to 
\begin{align}
\label{eq:quotmetric}
d( \pi(x), \pi(y) ) = \inf_{\mathcal{A} \in \liegroup } \|x-
\theta_{\mathcal{A}}(y) \|_T
\end{align}
which is well-defined and equal to (\ref{eq:quotmetricfull}) since $\|
\theta_{\mathcal{A}}(x) - \theta_{\mathcal{B}}(y)\|_T = \| x -
\theta_{\mathcal{A}^{-1} \mathcal{B}}(y)\|_{T}$ and $\mathcal{A},
\mathcal{B}$ vary over $\liegroup$. \\
Therefore, if we have a sequence $\{x_i\}$ in $\pi^{-1}
(\hat{\phi}^{-1}(C))$, by compactness of $\phi^{-1}(C)$, without loss
of generality we have
\begin{align*}
  \pi(x_i) \to \pi(y)\in \hat{\phi}^{-1}(C).
\end{align*}
Then, by the characterization (\ref{eq:quotmetric}), there exists a
sequence $\fancyA_i \subset \liegroup$ such that
\begin{align*}
  d( x_i, \theta_{\fancyA_i}(y) ) \to 0
\end{align*}
Since $\liegroup$ is compact, there exists a subsequence
$\{\fancyA_{i_j}\}$ that converges to $\fancyA \in \liegroup$. It then
follows that
\begin{align*}
  d( x_{i_j}, \theta_{\fancyA}(y) ) \le d( x_{i_j},
  \theta_{\fancyA_{i_j}}(y) ) + d( \theta_{\fancyA_{i_j}}(y),
  \theta_{\fancyA}(y) ) \to 0 \quad \text{ as } j\to\infty
\end{align*}
And so $\pi^{-1}( \hat{\phi}^{-1}( C ) ) $ is sequentially compact in
$\paramspaceclean$.
Therefore since the sequence $x_k$ generated by \Cref{alg:cg} stays
inside $\pi^{-1}( \hat{\phi}^{-1}( C ) )$ for all $i$, a subsequence
of $x_i$ converges to some $x \in \pi^{-1}( \hat{\phi}^{-1}( C ) )$,
and so $x$ is a critical point
of $f$.
\end{proof}

\section{Numerical examples}
To address the challenges of large-scale tensor completion problems,
as encountered in exploration seismology, we implemented the approach
outlined in this paper in a highly optimized parallel Matlab toolbox
entitled HTOpt (available at \url{http://www.math.ubc.ca/~curtd/software.html} for academic use).
Contrary to the HT toolbox \cite{htucktoolbox}, whose primary function is
performing operations on known HT tensors, our toolbox is designed to
solve optimization problems in the HT format such as the seismic
tensor completion problem. Our package includes the general optimization on HT manifolds detailed in \Cref{alg:riemmanngrad} as well as sparsity-exploiting objective \& Riemannian gradient in \Cref{alg:riemmgradsparse}, implemented in Matlab. We also include a parallel implementation using the Parallel Matlab toolbox for both of these algorithms.  All of the following experiments were run
on a single IBM x3550 workstation with 2 quad-core Intel 2.6Ghz
processors with 16GB of RAM running Linux 2.6.18. 

A simplified variation of the experiments below using the seismic data were presented previously in \cite{DaSi1307:Hierarchical}. The experiments in our conference proceedings subsamples sources and use a Conjugate-Gradient method to solve the interpolation problem. In this paper, we use receiver subsampling and our subsequently developed Gauss-Newton and regularization methods to improve the recovery substantially while simultaneously greatly reducing the number of iterations required.

To the best of our knowledge, there is no other existing method which
is able to interpolate HT tensors from a fixed sampling set $\Omega$. As
such, we compare our Gauss-Newton
method with the interpolation scheme detailed in
\cite{tuckercompletion}, denoted geomCG, for interpolating tensors
with missing entries on the Tucker manifold. We have implemented a
completely Matlab-based version of geomCG, which does not take
advantage of the sparsity of the residual when computing the objective
and Riemannian gradient, but uses Matlab's internal calls to LAPACK
libraries to compute matrix-matrix products and is much more efficient
for this problem. To verify that our implementation of the Tucker-based interpolation scheme is correct, we compare our method to the reference mex implementation for geomCG included in \cite{tuckercompletion}, see the HTOpt package for more details. Since we take advantage of dense linear algebra routines, we find that our Matlab implementation is significantly faster than the mex code of \cite{tuckercompletion} when $K \ge 20$ and $|\Omega|$ is a significant fraction of $N^d$, as is the case in the examples below. 

\subsection{Seismic data}
We briefly summarize the structure of seismic data in this section. Seismic data is collected via a boat equipped with an airgun and, for our purposes, a 2D array of receivers positioned on the ocean floor. The boat periodically fires a pressure wave in to the earth, which reflects off of subterranean discontinuities and produces a returning wave that is measured at the receiver array. The resulting data volume is five-dimensional, with two spatial source coordinates, denoted $x_{src}, y_{src}$, two receiver coordinates, denoted $x_{rec}, y_{rec}$, and time. For these experiments, we take a Fourier transform along the time axis and extract a single 4D volume by fixing a frequency and let $\tensor{D}$ denote the resulting \emph{frequency slice} with dimensions $n_{src} \times n_{src} \times n_{rec} \times n_{rec}$. 

From a practical point of view, the acquisition of seismic data from a physical system only allows us to subsample \emph{receiver} coordinates, i.e., $\Omega = [n_{src}] \times [n_{src}] \times \mathcal{I}$ for some $\mathcal{I} \subset [n_{rec}]\times [n_{rec}]$ with $|\mathcal{I}| < n_{rec}^2$, rather than the standard tensor completion approach, which assumes that $\Omega \subset [n_{src}] \times [n_{src}] \times [n_{rec}] \times [n_{rec}]$ is random and unstructured. As a result, we use the dimension tree 
\begin{figure}[H]
\centering
\includegraphics[scale=0.75]{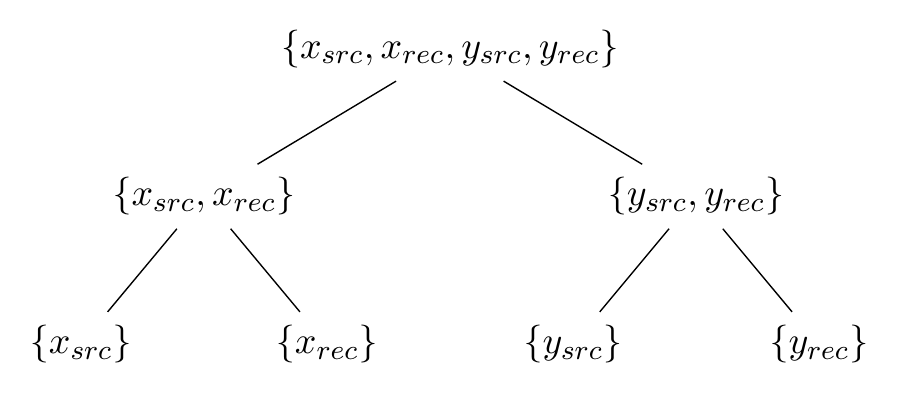}
\end{figure}
\noindent for completing seismic data. With this choice, the fully sampled data $\tensor{D}$ has \emph{quickly} decaying singular values in each matricization $D^{(t)}$ and is therefore represented well in the HT format. Additionally, the subsampled data $P_{\Omega} \tensor{D}$ has increased singular values in all matricizations, and is poorly represented as a HT tensor with fixed ranks $\mathbf{k}$ as a result. We examine this effect empirically in \cite{DaSi1307:Hierarchical} and note that this data organization is used in \cite{Demanet:2006td} in the context of solution operators of the wave equation. Although this approach is limited to considerations of seismic data, for larger dimensions/different domains, potentially the method of \cite{ballani2013tree} can choose an appropriate dimension tree automatically. In the next section, we also include the case when $\Omega \subset  [n_{src}] \times  [n_{src}]  \times  [n_{rec}] \times  [n_{rec}] $, i.e. the ``missing points'' scenario, to demonstrate the added difficulty of the ``missing receivers'' case described above.

\subsection{Single reflector data}
For this data set, we generate data from a very simple seismic model
consisting of two horizontal layers with a moderate difference in wavespeed and
density between them. We generate this data with $n_{src} = n_{rec} = 50$ and extract a frequency slice at 4.21Hz, rescaled to have unit norm. 

We consider the two sampling scenarios discussed in the previous section: we remove random points from the tensor, with results shown in
\Cref{fig:singlereflector-missingpts}, and we remove random
receivers from the tensor, with results shown in
\Cref{fig:singlereflector-missingrecs}. Here
geomCG($r_{\text{leaf}}$) - $w$ denote the Tucker interpolation
algorithm with rank $r_{\text{leaf}}$ in each mode and $w$ rank
continuation steps, i.e., the approach proposed in \cite{tuckercompletion}. We also let HT($r_{\text{leaf}},
r_{x_\text{src} x_\text{rec}}$) denote the HT interpolation method with
rank $r_{\text{leaf}}$ as in the Tucker interpolation and rank
$r_{x_\text{src} x_\text{rec}}$ as the internal rank for the dimension
tree. As is customary in the seismic literature, we measure recovery quality in
terms of SNR, namely
\begin{align*}
  \text{SNR}(\tensor{X},\tensor{D}) = -20\log_{10} \left( \dfrac{\|\tensor{X}_{\Omega^c} -
    \tensor{D}_{\Omega^c}\|}{\|\tensor{D}_{\Omega^c}\|} \right) \text{dB}, 
\end{align*}
where $\tensor{X}$ is our interpolated signal, $\tensor{D}$ is our reference solution, and $\Omega^c = [n_{src}] \times [n_{src}] \times [n_{rec}] \times [n_{rec}] \setminus \Omega$.
As we can see in \Cref{fig:singlereflector-missingpts}, the HT
formulation is able to take advantage of low-rank separability of the
seismic volume to produce a much
higher quality solution than that of the Tucker tensor completion. The
rank continuation scheme does not seem to
be improving the recovery quality of the Tucker solution to the same
degree as in \cite{tuckercompletion}, although it does seem to mitigate
some of the overfitting errors for geomCG($30$). We display slices for
fixed source coordinates and varying receiver coordinates in \Cref{fig:singlereflector-missingpts} for
randomly missing points and \Cref{fig:singlereflector-missingrecs} for
randomly missing receivers. By exploiting the low-rank structure of
the HT format compared to the Tucker format, we are able to achieve
much better results than Tucker tensor completion, especially for the
realistic case of missing receiver samples.

In all instances for these experiments, the HT tensor completion
outperforms the conventional Tucker approach both in terms of recovery
quality and recovery speed. We note that geomCG does not scale as well
computationally as our HT algorithm for $d > 3$, as the complexity analysis in \cite{tuckercompletion} predicts. As such, we only
consider the HT interpolation for the next sections, where we will solve
the tensor completion problem for much larger data volumes.

\begin{figure}
\captionsetup[subfigure]{labelformat=parens,width=110pt}
\centering
\subfloat[True data] {
\includegraphics[width=0.23\columnwidth]{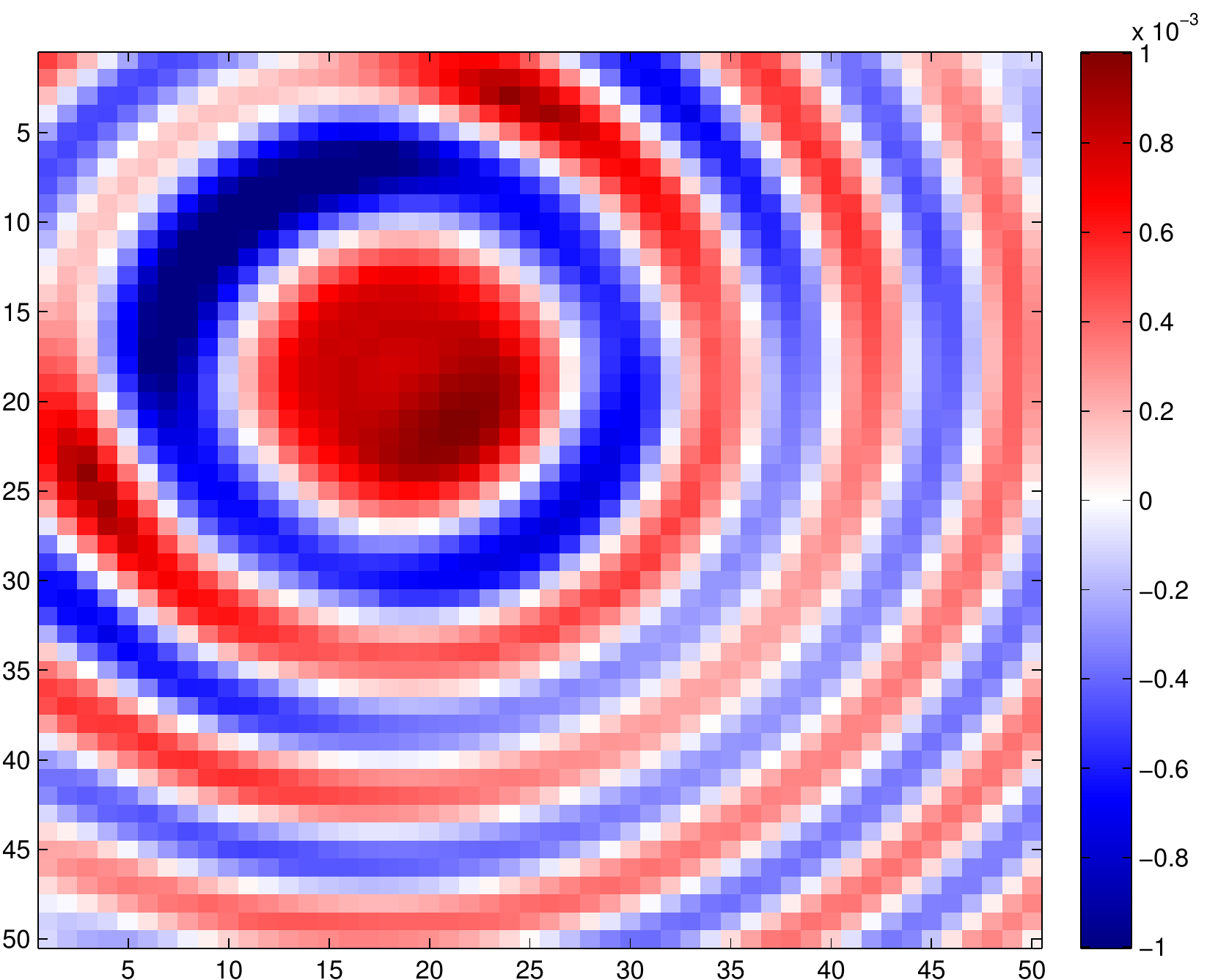}
} ~
\subfloat[ Input data $b = P_{\Omega} \tensor{X^*}$ ]{
\includegraphics[width=0.23\columnwidth]{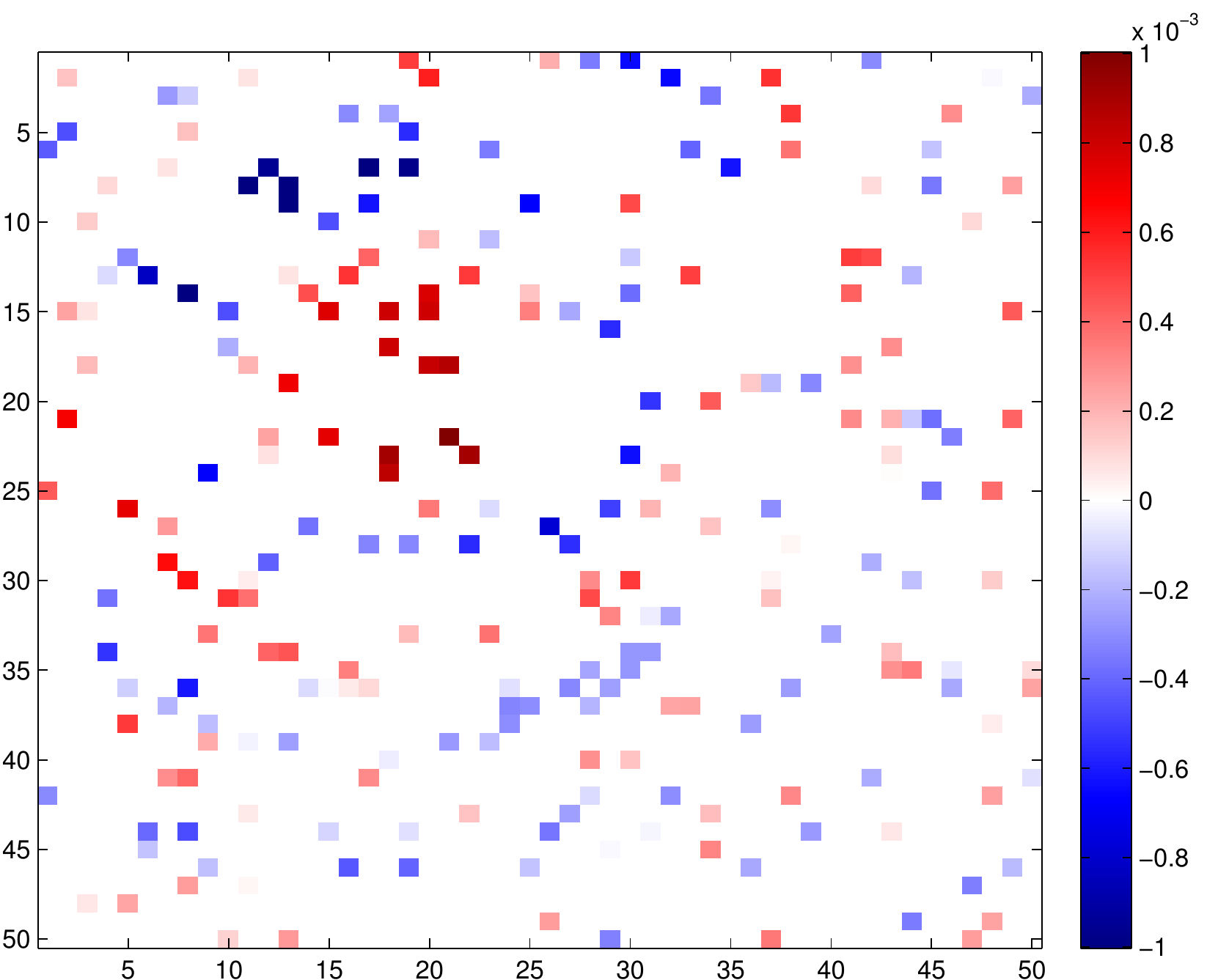}
} ~
\subfloat[HT(30,80) - SNR $30.5$ dB ]{
\includegraphics[width=0.23\columnwidth]{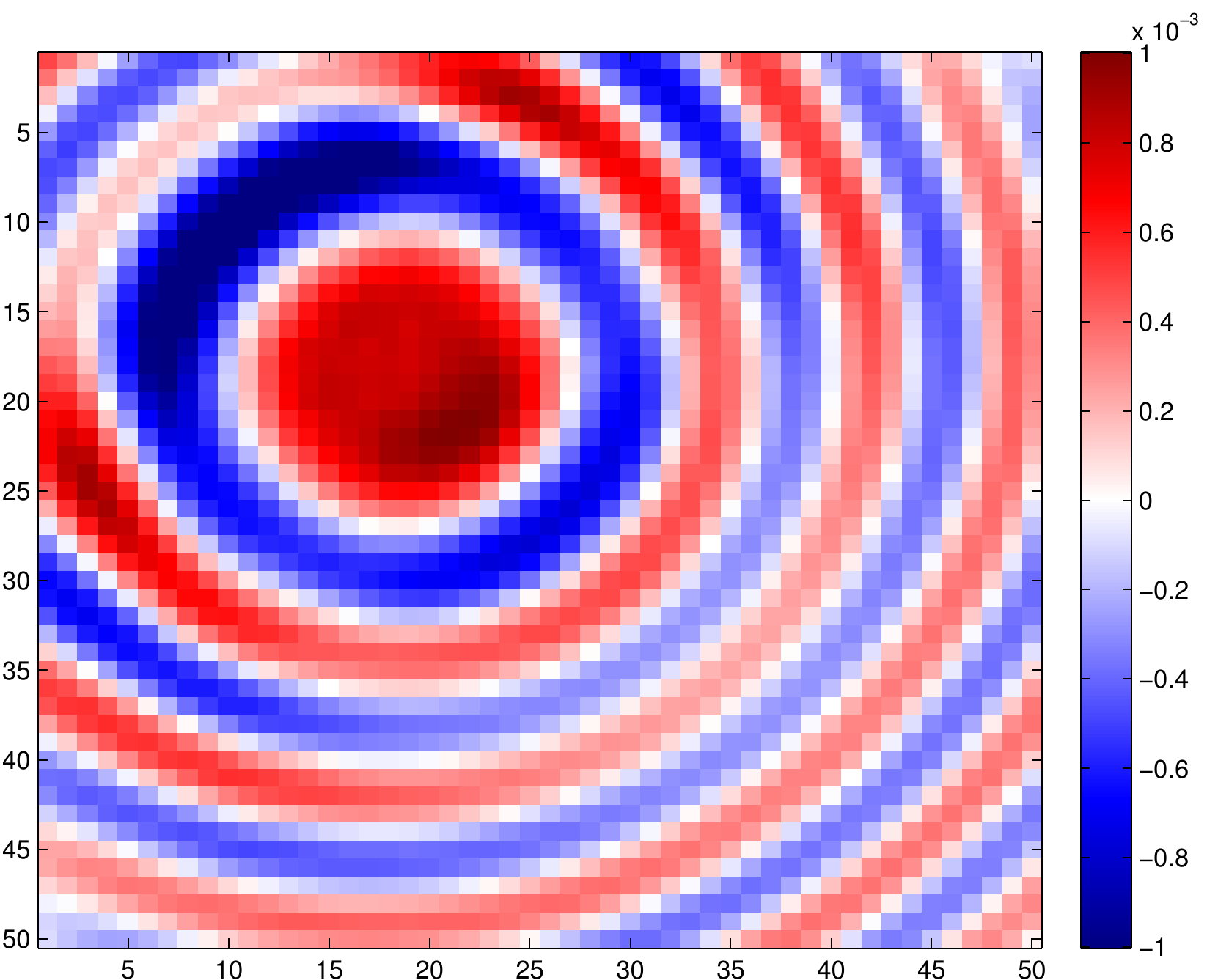}
} ~
\subfloat[geomCG(20)-0 - SNR $29.4$ dB ]{
\includegraphics[width=0.23\columnwidth]{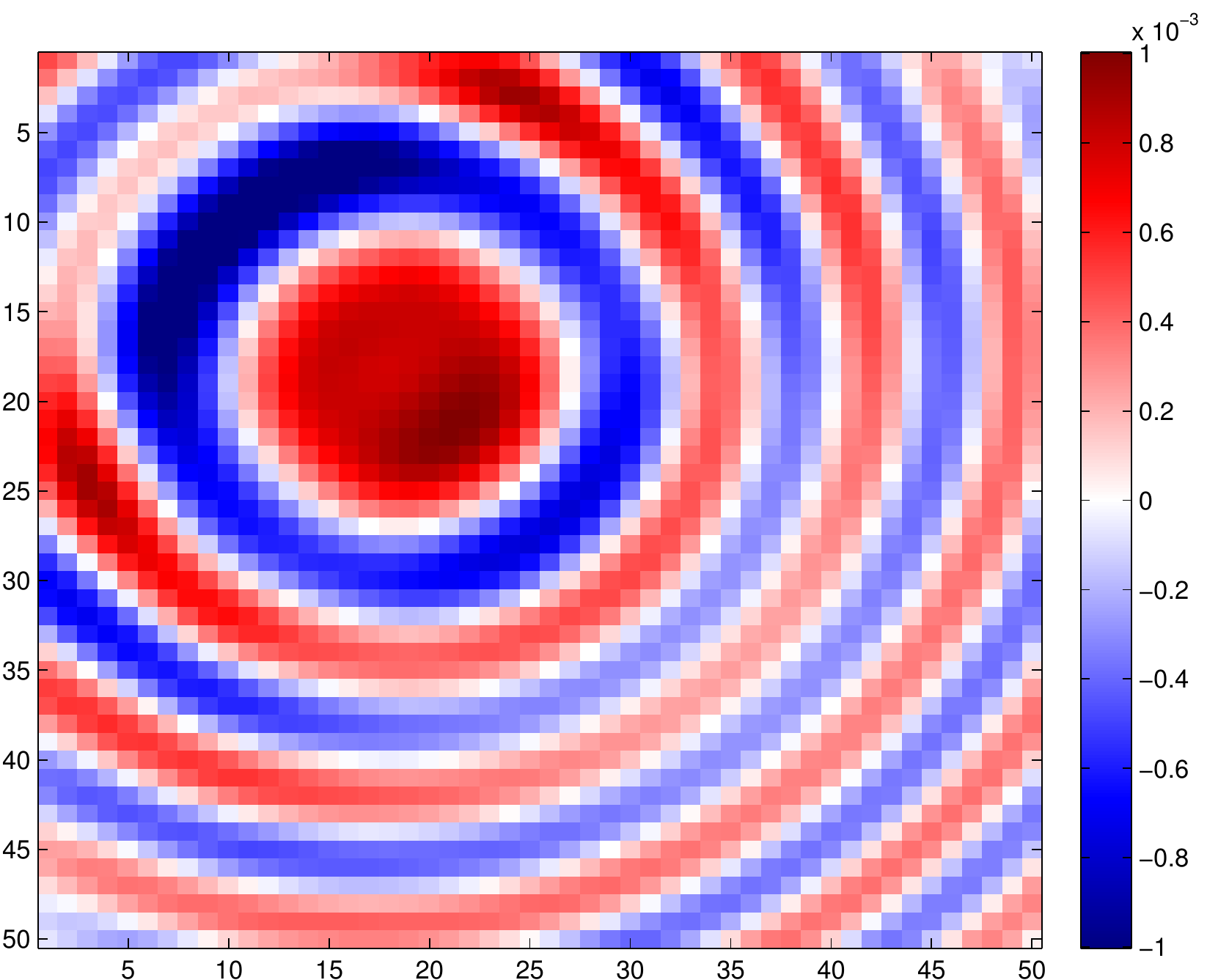}
} 
\caption{Reconstruction results for 90\% missing points, best results for geomCG and HTOpt. }
\label{fig:singlereflector-missingpts}
\end{figure}
\begin{figure}
\captionsetup[subfigure]{labelformat=parens,width=130pt}
\centering
\subfloat[ 90\% missing receivers ]{
\includegraphics[width=0.3\linewidth]{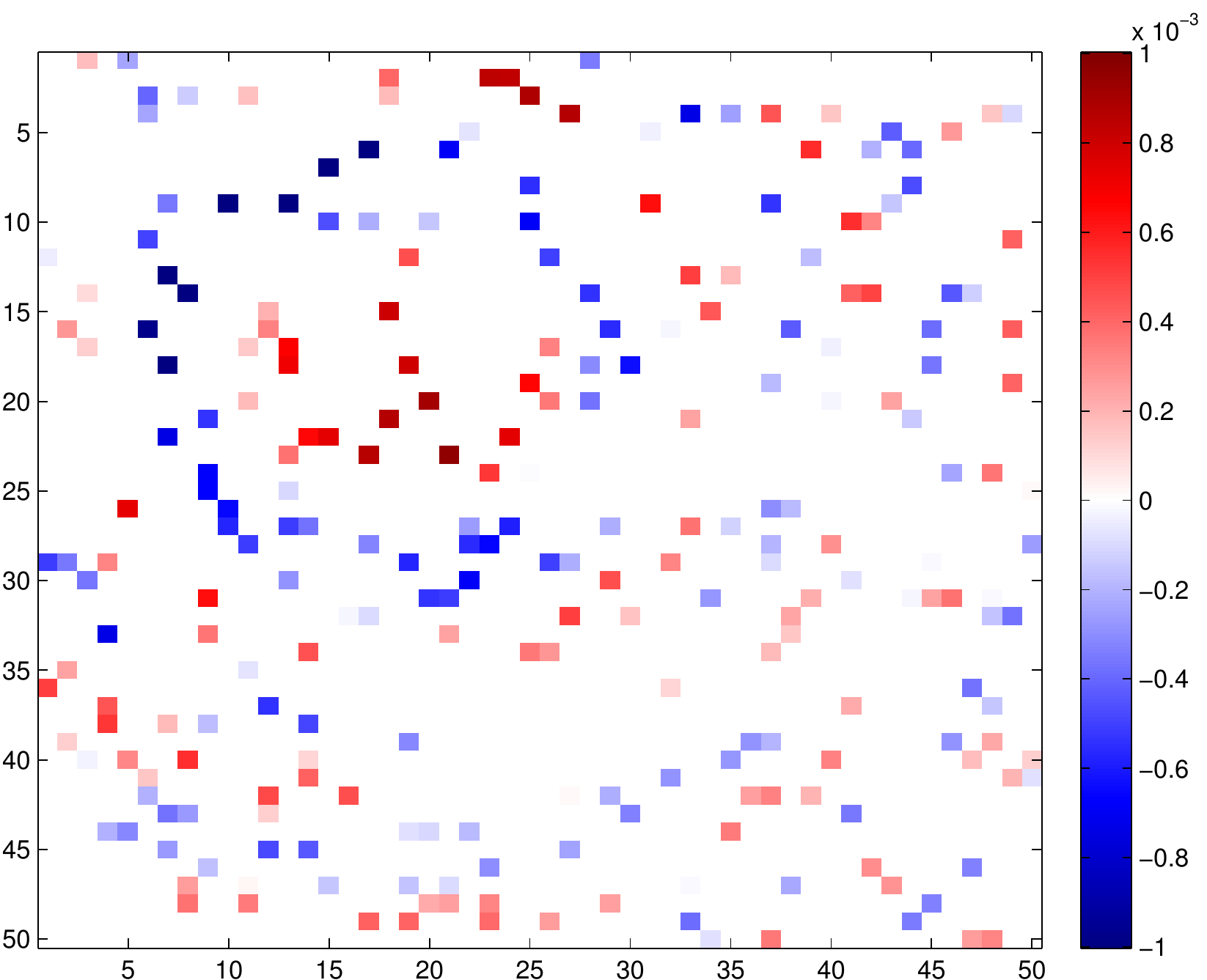}
} ~
\subfloat[HT(20,20) - SNR $7.04$ dB ]{
\includegraphics[width=0.3\linewidth]{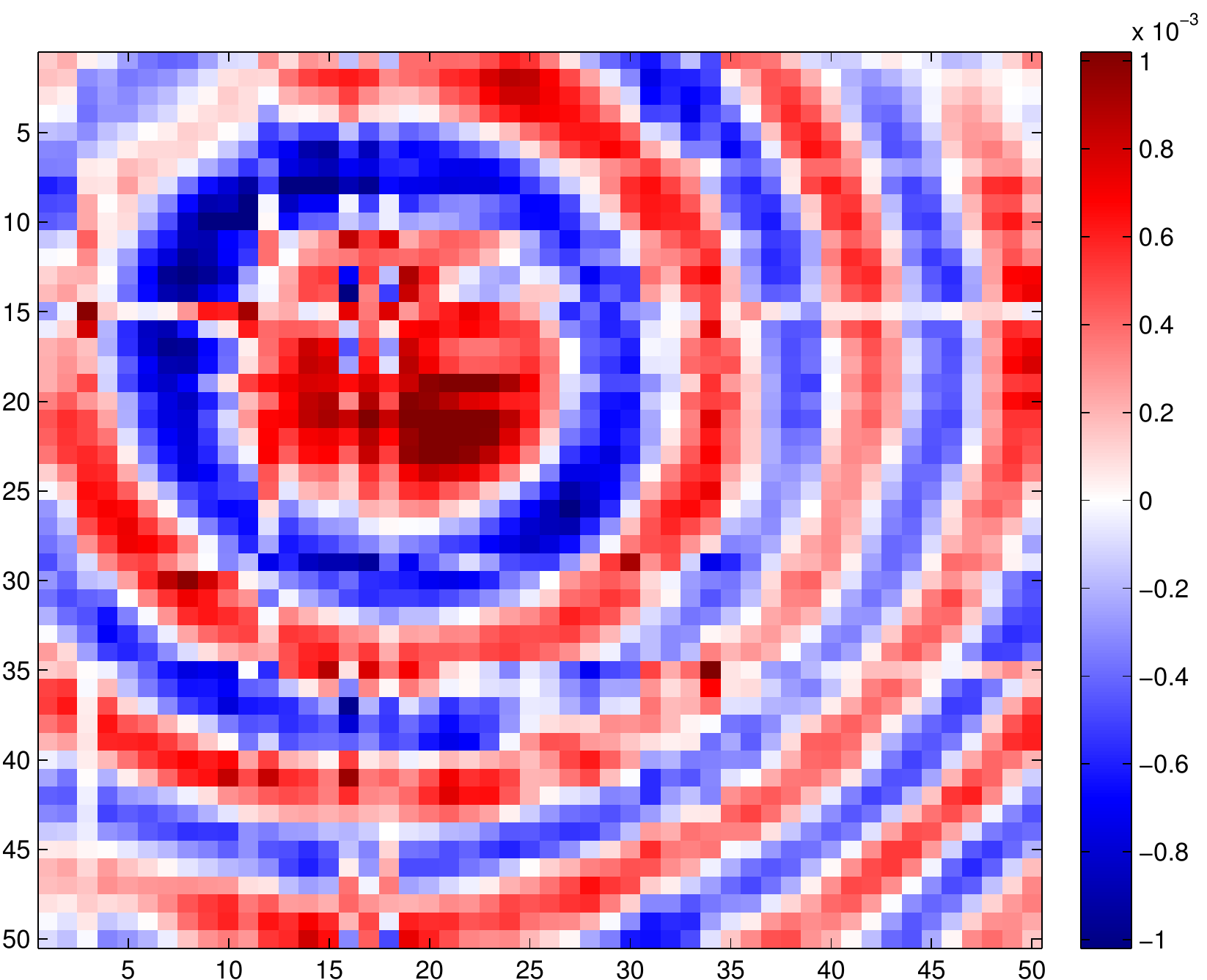}
} ~
\subfloat[geomCG(20)-0 - SNR $-1.92$ dB ]{
\includegraphics[width=0.3\linewidth]{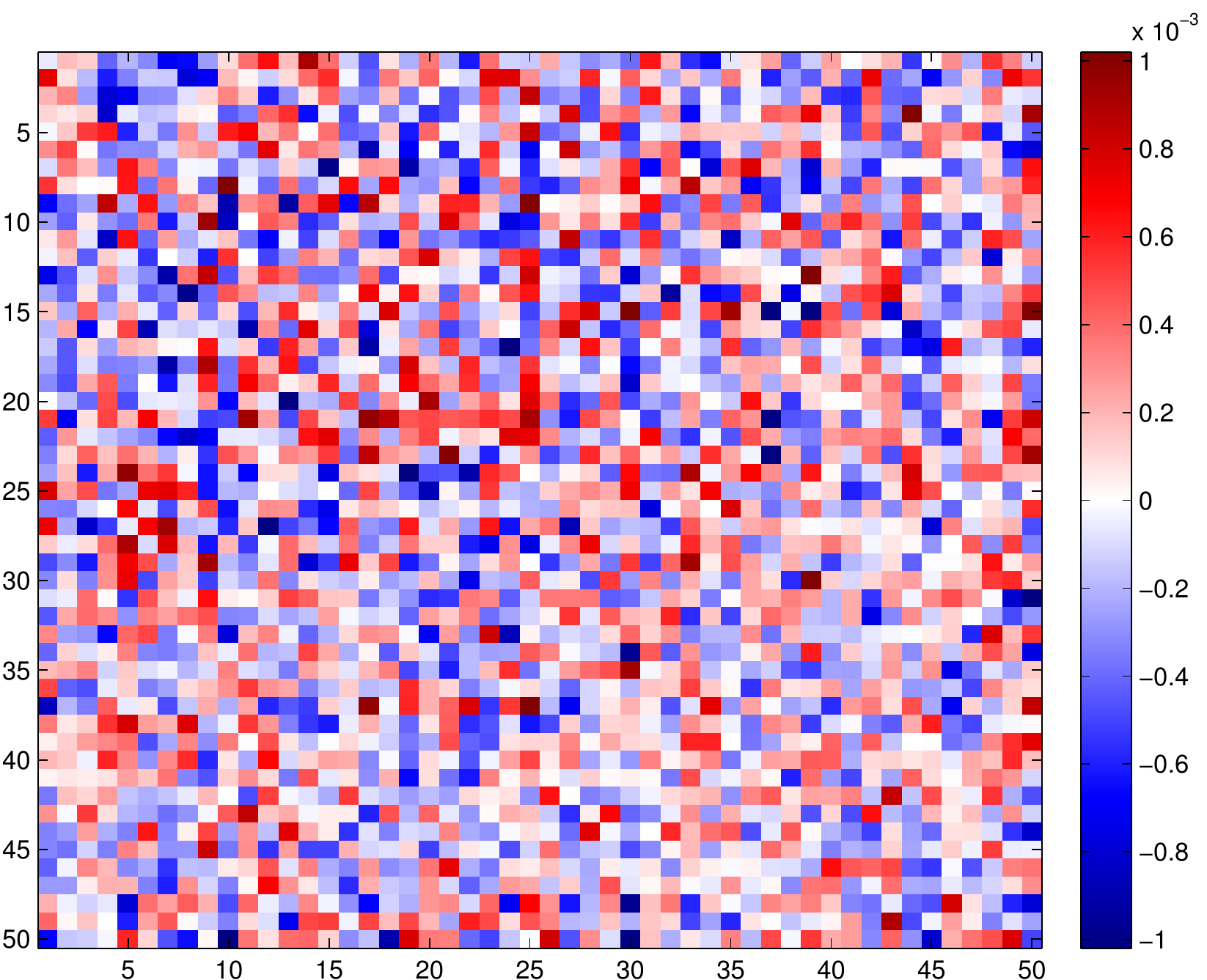}
} \\
\subfloat[ 70\% missing receivers ]{
\includegraphics[width=0.3\columnwidth]{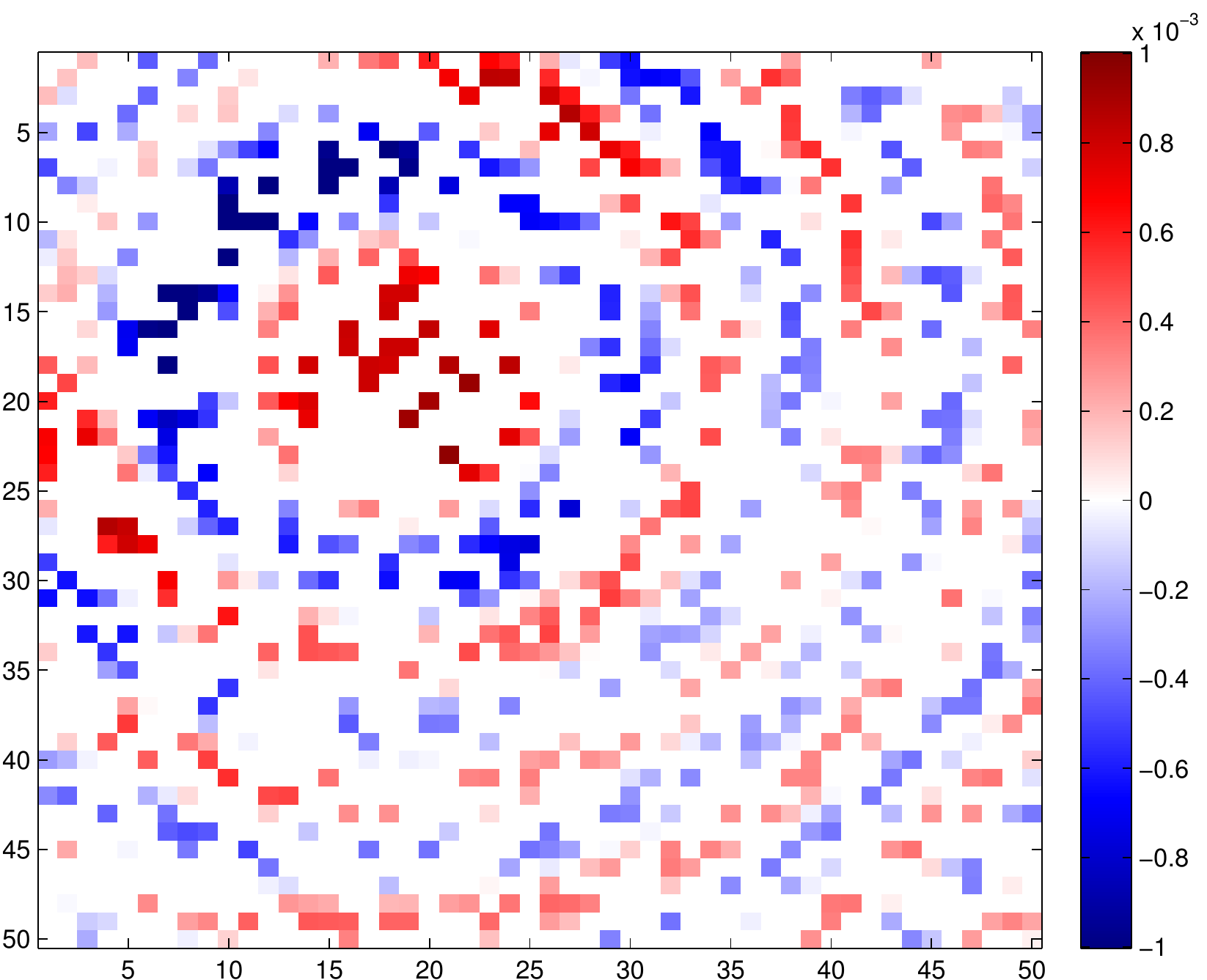}
} ~
\subfloat[HT(20,20) - SNR $20.4$ dB ]{
\includegraphics[width=0.3\columnwidth]{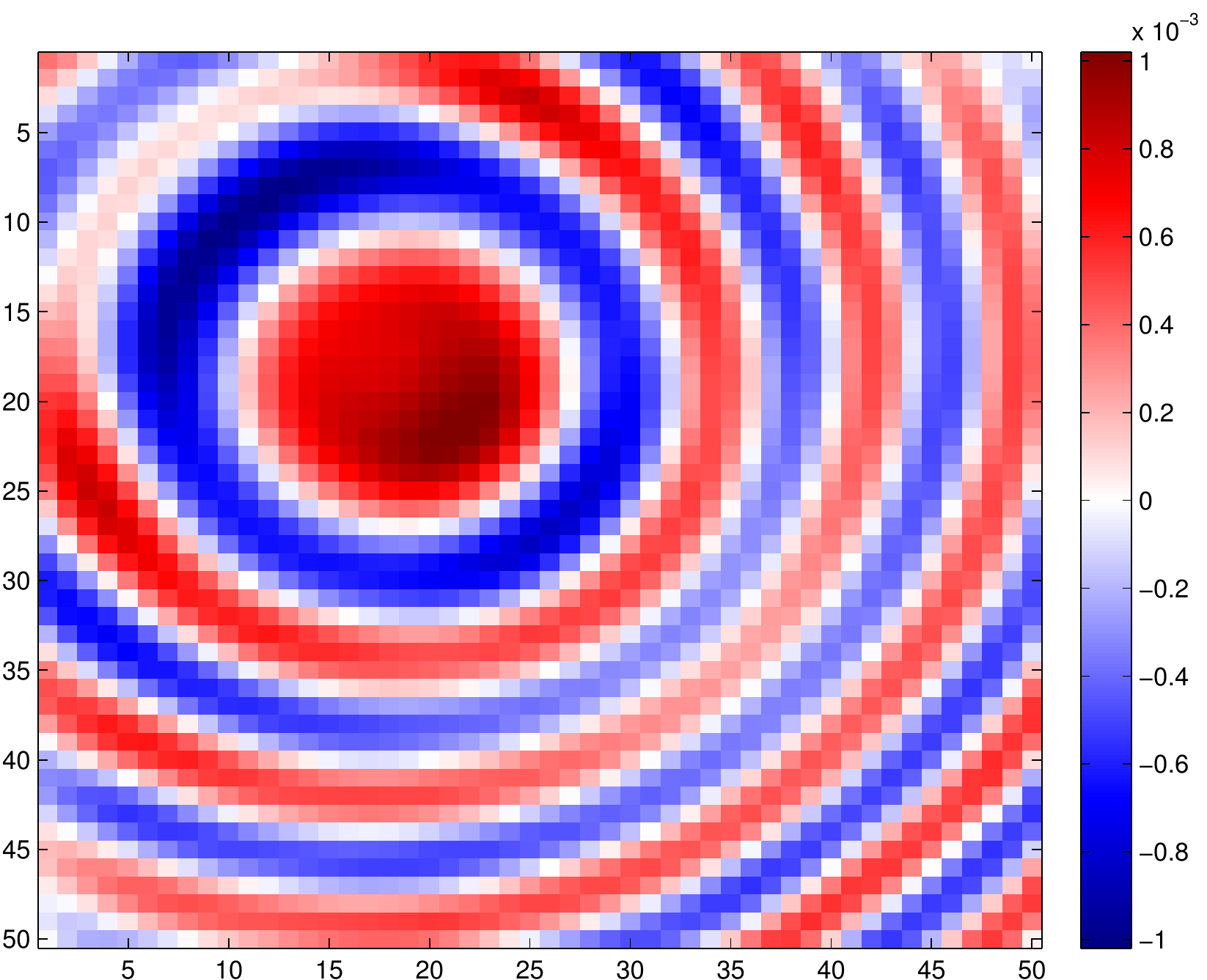}
} ~
\subfloat[geomCG(30)-5 - SNR $16.8$ dB ]{
\includegraphics[width=0.3\columnwidth]{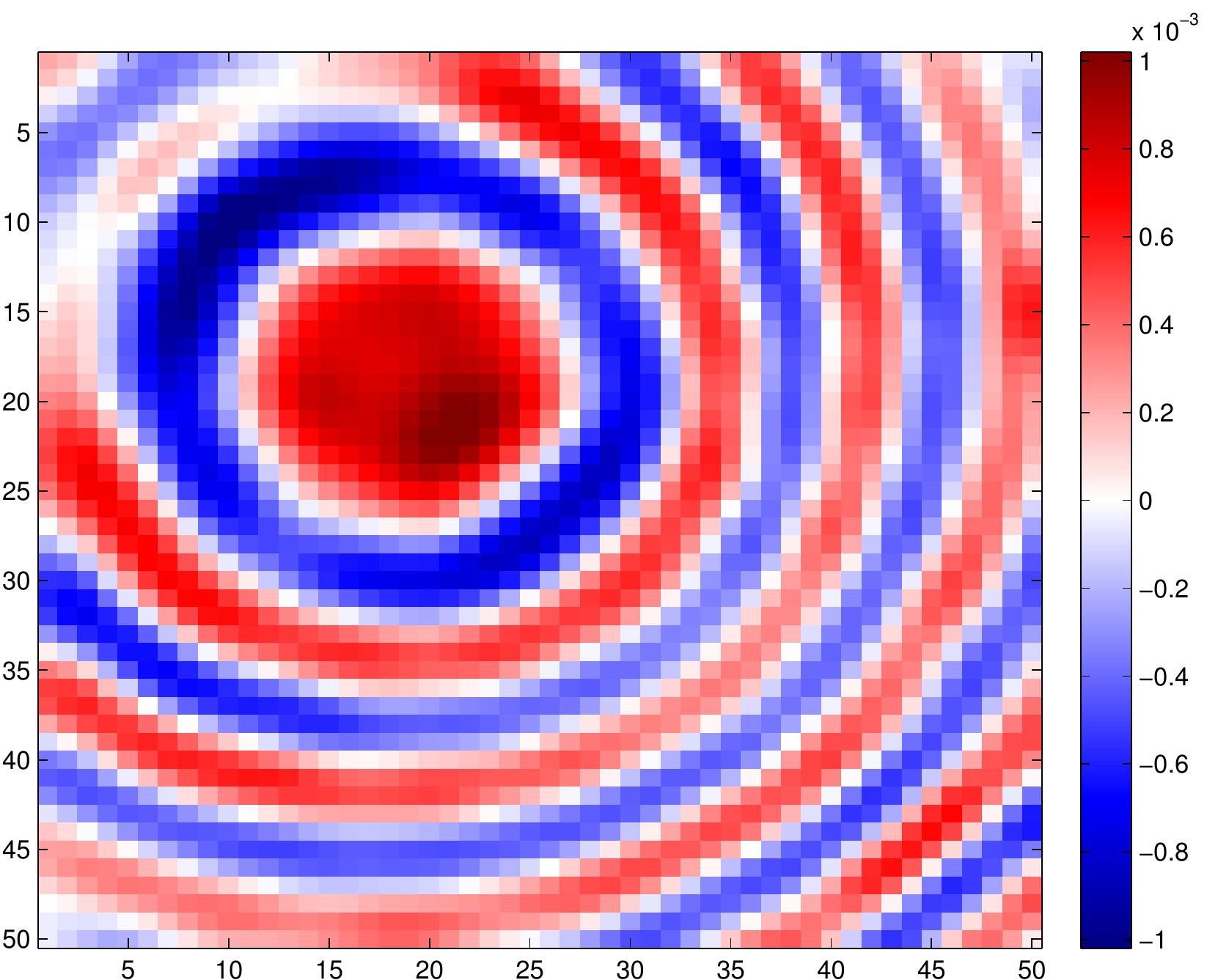}
} 
\caption{Reconstruction results for sampled receiver coordinates, best
  results for geomCG and HTOpt.  \emph{(a-c)} 90\% missing
  receivers. \emph{(d-f)}: 70\% missing receivers. }
\label{fig:singlereflector-missingrecs}
\end{figure}

\begin{table}
 \centering 
\begin{tabular}{c | c c | c c | c c } 
\Xhline{2\arrayrulewidth}  \\ 
& \multicolumn{6}{c}{\centering Single reflector data - sampling percentage (missing points) }\\ 
&\multicolumn{2}{c}{\centering 10\%} & \multicolumn{2}{c}{\centering 30\%} & \multicolumn{2}{c}{\centering 50\%} \\ 
& SNR [dB] & time [s] & SNR [dB] & time [s] & SNR [dB] & time [s] \\ 
\hline
geomCG(20) - 0 & 28.5 & 1023 & 30.5 & 397 & 30.7 & 340 \\ 
geomCG(30) - 0 & -6.7 & 1848 & 21.8 & 3621 & 31.5 & 2321 \\ 
geomCG(30) - 5 & 16.1 & 492 & 13.8 & 397 & 15.5 & 269 \\ 
HTOpt(20,60) & 30.1 & 83 & 30.4 & 59 & 30.4 & 57 \\ 
HTOpt(20,80) & 30.3 & 121 & 30.8 & 75 & 30.8 & 53 \\ 
HTOpt(30,80) & \textbf{31.6} & 196 & \textbf{32.9} & 133 & \textbf{33.1} & 114 \\ 
\Xhline{2\arrayrulewidth}
\end{tabular}
\caption{Reconstruction results for single reflector data - missing points - mean SNR over 5 random training sets} 
\label{table:singlereflector-missingall} 
\end{table}

\begin{table}
 \centering 
\begin{tabular}{c | c c | c c | c c } 
\Xhline{2\arrayrulewidth}  \\ 
& \multicolumn{6}{c}{\centering Single reflector data - sampling percentage (missing receivers) }\\ 
&\multicolumn{2}{c}{\centering 10\%} & \multicolumn{2}{c}{\centering 30\%} & \multicolumn{2}{c}{\centering 50\%} \\ 
& SNR [dB] & time [s] & SNR [dB] & time [s] & SNR [dB] & time [s] \\ 
\hline
geomCG(20) - 0 & -5.1 & 899 & 9.9 & 898 & 18.5 & 891 \\ 
geomCG(30) - 0 & -3.6 & 1796 & -4.7 & 1834 & 6.1 & 1802 \\ 
geomCG(30) - 5 & -6.4 & 727 & 11.1 & 670 & 14.2 & 356 \\ 
HTOpt(20,20) & \textbf{6.1} & 111 & \textbf{19.8} & 101 & 20.1 & 66 \\ 
HTOpt(30,20) & 2.8 & 117 & 18.1 & 109 & 19.8 & 94 \\ 
HTOpt(30,40) & 0.0 & 130 & 13.4 & 126 & \textbf{21.6} & 108 \\ 
\Xhline{2\arrayrulewidth}
\end{tabular}
\caption{Reconstruction results for single reflector data - missing receivers - mean SNR over 5 random training test sets} 
\label{table:singlereflector-missingrecs} 
\end{table}

\subsection{Performance}
We investigate the empirical performance scaling of our approach as $N, d, K,$ and $|\Omega|$ increase, as well as the number of processors for the parallel case, in Figure \ref{fig:performance} and Figure \ref{fig:performance2}. Here we denote the use of \Cref{alg:riemmanngrad} as the ``dense'' case and Figure \ref{alg:riemmgradsparse} as the ``sparse'' case. We run our optimization code in Steepest Descent mode with a single iteration for the line search, and average the running time over $10$ iterations and $5$ random problem instances.  Our empirical performance results agree very closely with the theoretical complexity estimates, which are $O(N^d K)$ for the dense case and $O(|\Omega| d K^3)$ for the sparse case. Our parallel implementation for the sparse case scales very close to the theoretical time $O(1/\text{\# processors})$.

\begin{figure}
\centering
\subfloat[Fixed $K, d$, varying $N$, $|\Omega| = 1000N$] {
\includegraphics[scale=0.4]{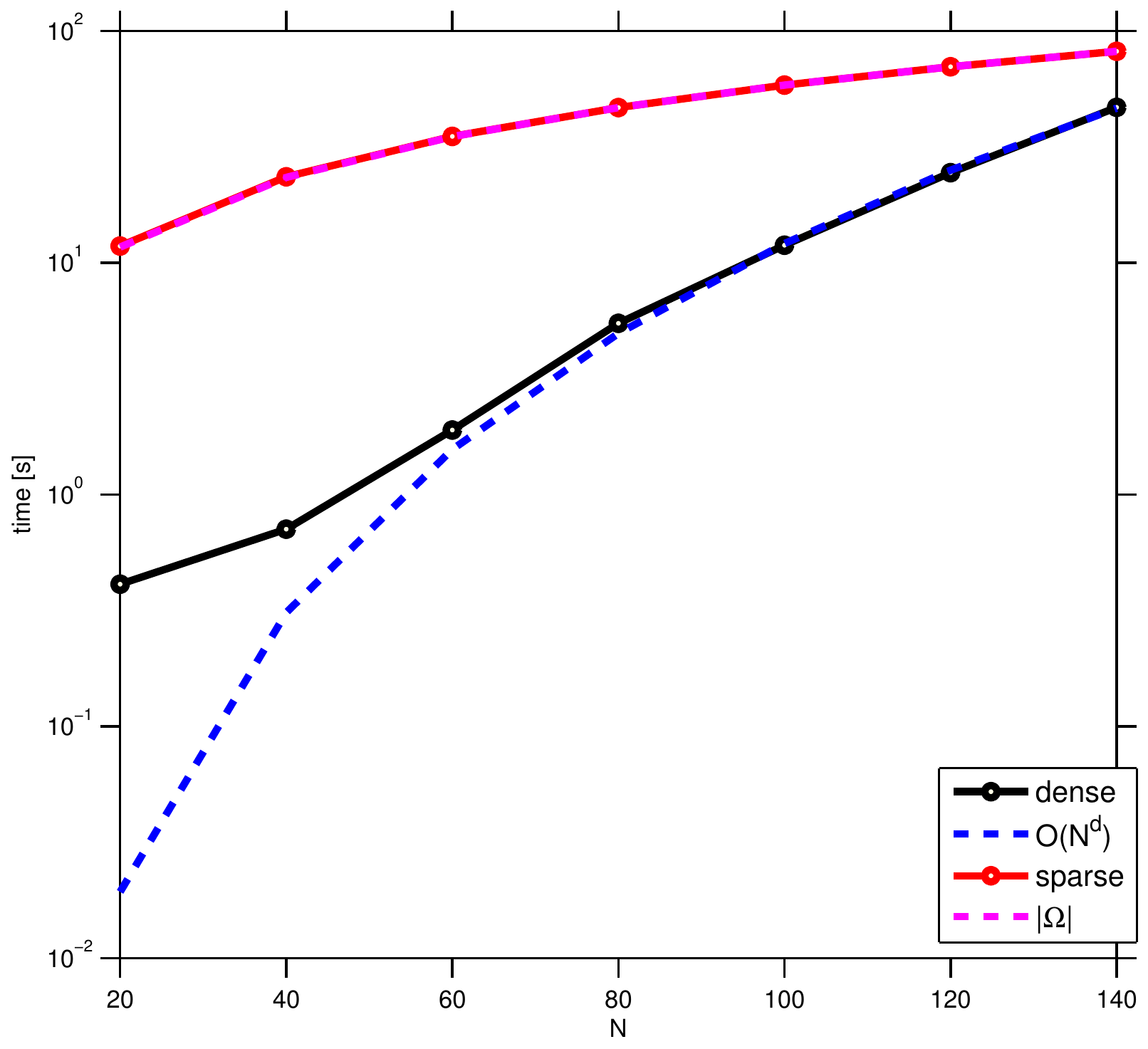}
} ~
\subfloat[Fixed $N, d, |\Omega|$, varying $K$]{
\includegraphics[scale=0.4]{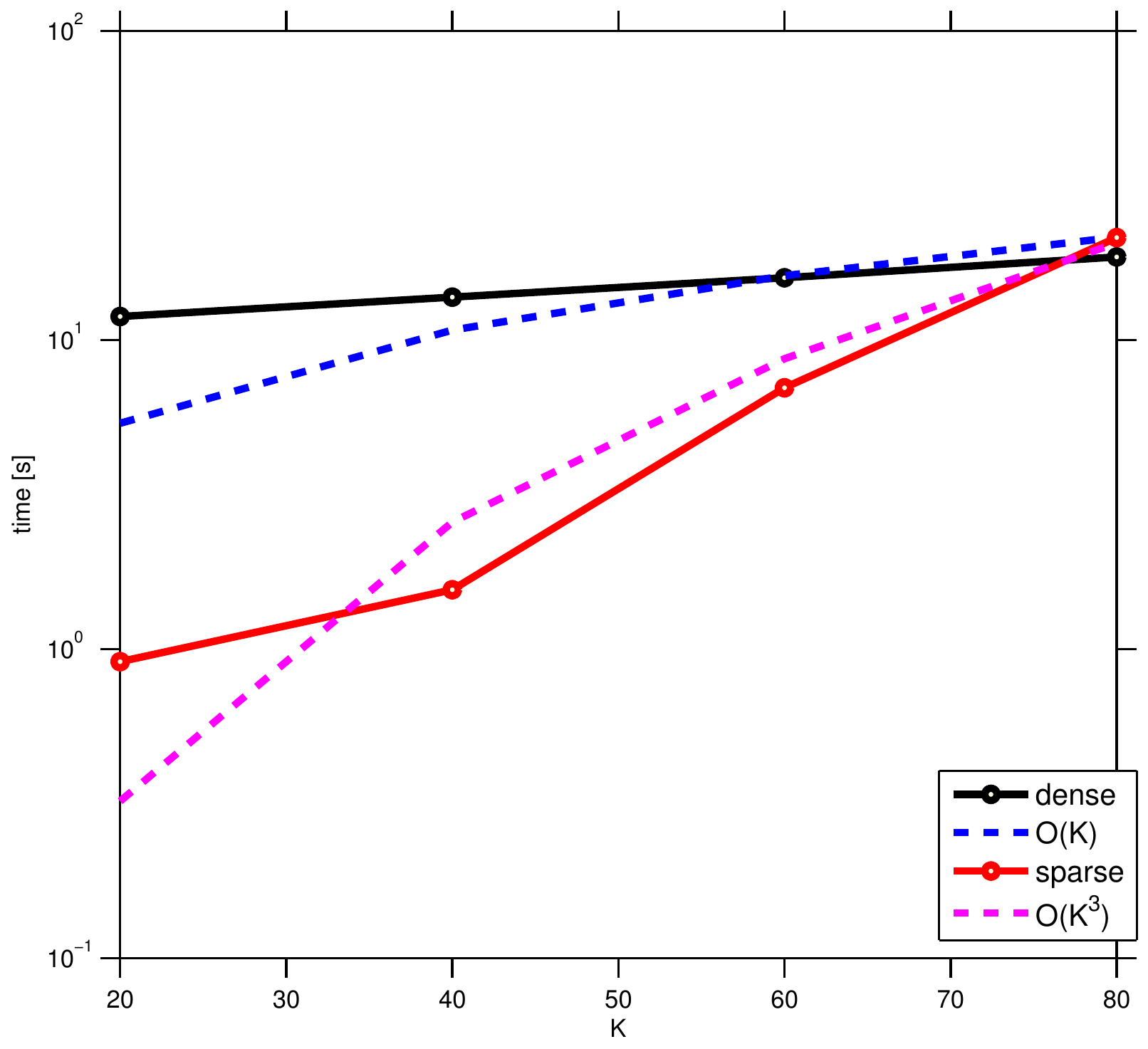}
} 
\caption{ Dense \& sparse objective, gradient performance. }
\label{fig:performance}
\end{figure}

\begin{figure}
\centering
\subfloat[ Fixed $N, K, |\Omega|$, varying $d$]{
		\includegraphics[scale=0.4]{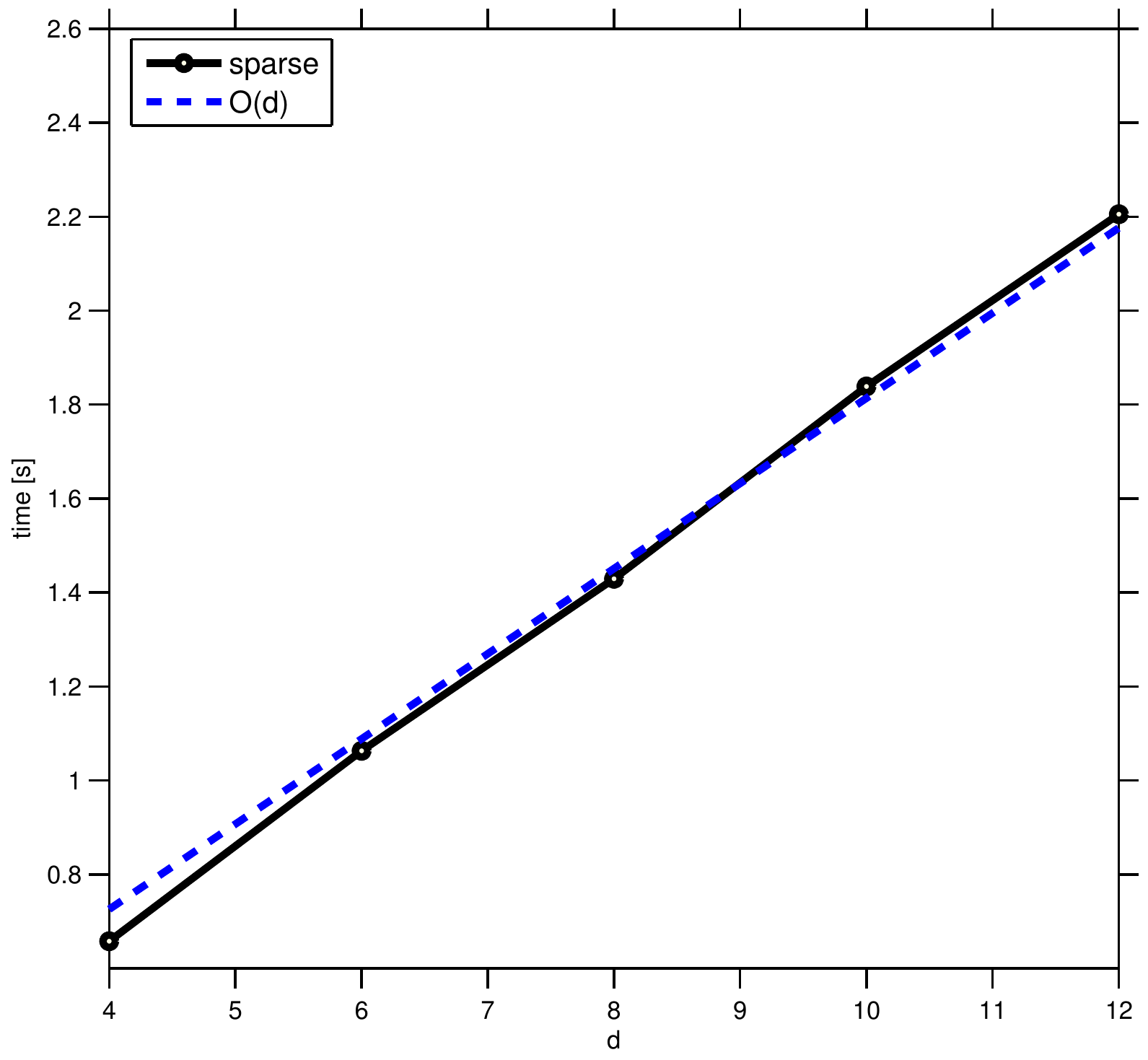}
} ~
\subfloat[ Fixed $N, K, d, \Omega$, varying number of processors ]{
		\includegraphics[scale=0.5]{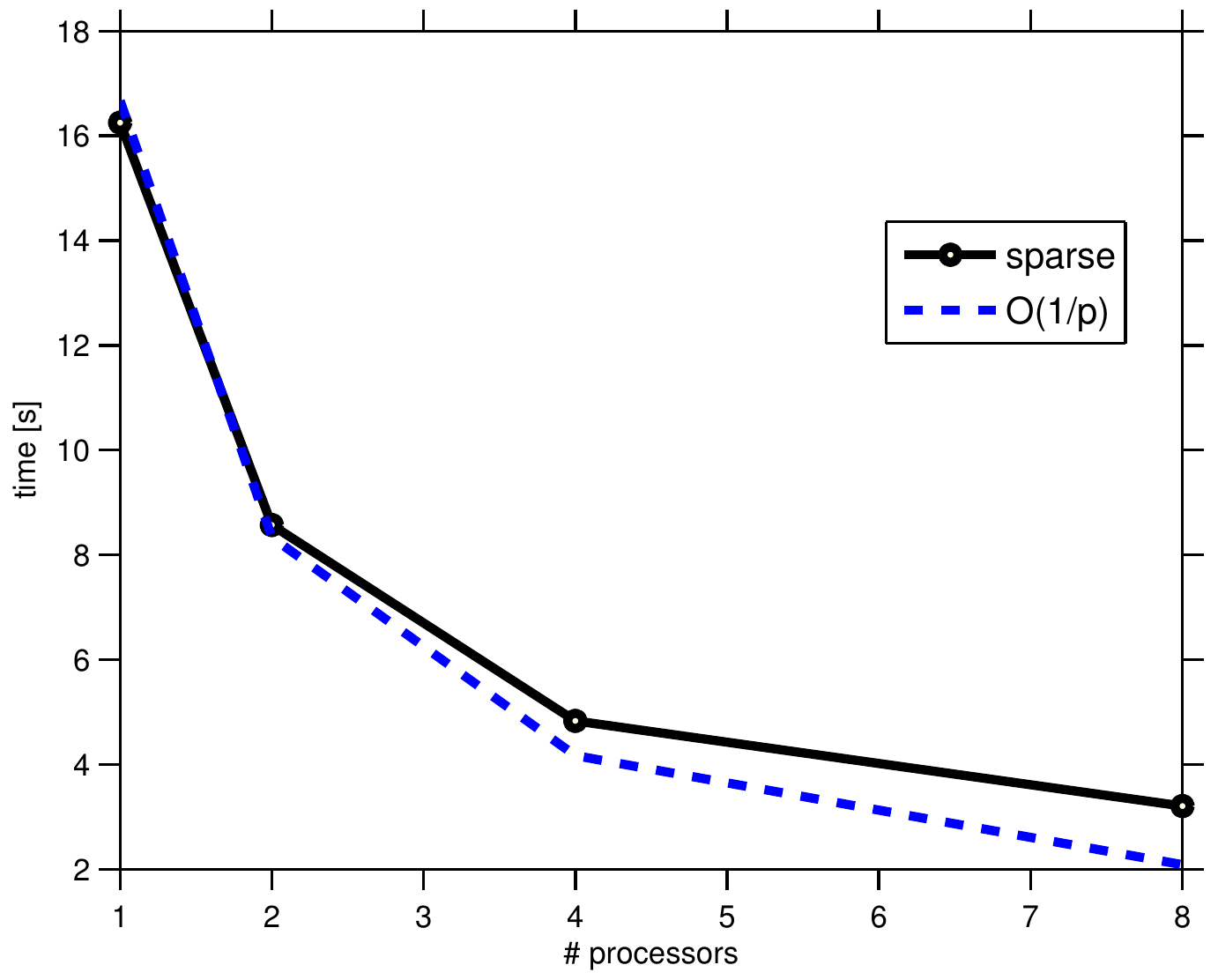}
} 
\caption{ Sparse objective, gradient performance. }
\label{fig:performance2}
\end{figure}

\subsection{Synthetic BG Compass data}
This data set was provided to us by BG and consists of 5D data
generated from an unknown synthetic model. Here $n_{src} = 68$ and
$n_{rec} = 401$ and we extract frequency slices at
4.86 Hz, 7.34 Hz, and 12.3 Hz. On physical grounds, we expect a slower
decay of the singular values at higher frequencies and thus the problem is much more difficult at 12.3 Hz compared to 4.86 Hz. 

At these frequencies, the data has relatively low spatial frequency
content in the receiver coordinates, and thus we subsample the
receivers by a factor of $2$ to $n_{rec} = 201$, for the purposes of
speeding up the overall computation and ensuring that the intermediate
vectors in the optimization are able to fit in memory. Our overall data volume has dimensions $\tensor{D} \in \mathbb{R}^{68 \times 68 \times 201 \times 201}$.

We randomly remove varying amounts of receivers from this reduced data volume and
interpolate using 50 iterations of the GN method discussed earlier. We display several recovered slices for fixed source coordinates and varying receiver coordinates (so-called \emph{common source gathers} in seismic terminology) in \Cref{fig:bgdata75missingrecs}.

We summarize our recovery results for tensor completion on these data
sets from missing receivers in \Cref{table:snrrecs} and the various
recovery parameters we use in \Cref{table:recoveryparams}. When the
subsampling rate is extremely high (90\% missing receivers in these
examples), the recovery can suffer from overfitting issues, which
leads to spurious artifacts in the recovered volume and lower SNRs
overall. Using the Gramian-based regularization method discussed
earlier, we can mitigate some of those artifacts and boost recovered
SNRs, as seen in \Cref{fig:regularization}.

\section{Conclusions and discussion} 
In this work we have developed the algorithmic components to solve
optimization problems on the manifold of fixed-rank Hierarchical Tucker
tensors. By exploiting this manifold structure, we solve the tensor
completion problem where the tensors of interest exhibit low-rank
behavior. Our algorithm is computationally efficient because we mostly
rely on operations on the small HT parameter space. The manifold
optimization itself guarantees that we do not run into convergence
issues, which arise when we ignore the quotient structure of the HT
format. Our application of this framework to seismic
examples confirms the validity of our new approach and outperforms
existing Tucker-based approaches for large data volumes. To stabilize the recovery for high
subsampling ratios, we introduced an additional regularization term
that exploits properties of the Gramian matrices without the need to
compute SVDs in the ambient space. 

While the method clearly performs well on large-scale problems, there
are still a number of theoretical questions regarding the performance
of this approach. In particular, the generalization of matrix
completion recovery guarantees to the HT format remains an open
problem. As in many alternative approaches to matrix/tensor
completion, the selection of the rank parameters and regularization
parameters remain challenging both theoretically and from a practical
point of view. However, the paper clearly illustrates that the HT
format is a viable option to represent and complete high-dimensional
data volumes in a computationally feasible manner.

\section{Acknowledgements}
We would like to thank the sponsors of the SINBAD consortium for their
continued support. We would also like to thank the BG Group for
providing us with the Compass data set.

\begin{table}[H]
\centering
\begin{tabular}{c | c| c | c | c }
Frequency & \% Missing & Train SNR (dB) & Test SNR (dB) & Runtime (s)  \\
\hline
4.86 Hz & 25\% & 21.2  & 21 & 4033 \\
\hline
 & 50\% & 21.3  & 20.9 & 4169 \\
\hline
 & 75\% & 21.5  & 19.9 & 4333 \\
\hline
 & 90\% & 19.9  & 10.4 & 4679  \\
\hline
 & $90\%^*$ & $20.8^*$ & $13.0^*$ & 5043 \\
\hline	
7.34 Hz & 25\% & 17.3  & 17.0 & 4875 \\
\hline
 & 50\% & 17.4 & 16.9 & 4860  \\
\hline
 & 75\% & 17.7  & 16.5 & 5422 \\
\hline
 & 90\% & 16.6  & 9.82 & 4582 \\
\hline
 & $90\%^*$ & $16.6^*$  & $10.5^*$ & 4947 \\
\hline
12.3 Hz & 25\% & 14.9 & 14.2 & 5950 \\
\hline
& 50\% & 15.2 & 13.8 & 7083\\
\hline
& 75\% & 15.8 & 9.9 & 7387 \\
\hline
& 90\% & 13.9 & 5.39 & 4578 \\
\hline
& $90\%^*$ & $14^*$ & $6.5^*$ & 4966
\end{tabular}
\caption{ HT Recovery results - randomly missing receivers. Starred
  quantities are computed with regularization.  }
\label{table:snrrecs}
\end{table}

\begin{table}[H]
\centering
\begin{tabular}{c | c | c | c | c}
Frequency & $k_{x_{src} x_{rec}}$ & $k_{x_{src}}$ & $k_{x_{rec}} $ & HT-SVD SNR (dB) \\
\hline
4.86 Hz & 150 & 68 & 120 & 21.1 \\
\hline
7.34 Hz & 200 & 68 & 120 & 17.0\\
\hline
12.3 Hz & 250 & 68 & 150 & 13.9
\end{tabular}
\caption{ HT parameters for each data set and the corresponding SNR of
  the HT-SVD approximation of each data set. The 12.3 Hz data is
  of much higher rank than the other two data sets and thus is much
  more difficult to recover. }
\label{table:recoveryparams}
\end{table}

\begin{figure}
	\centering
        \captionsetup[subfigure]{labelformat=empty,position=top,width=120pt}
	\subfloat[True Data]{
		\includegraphics[scale=0.23]{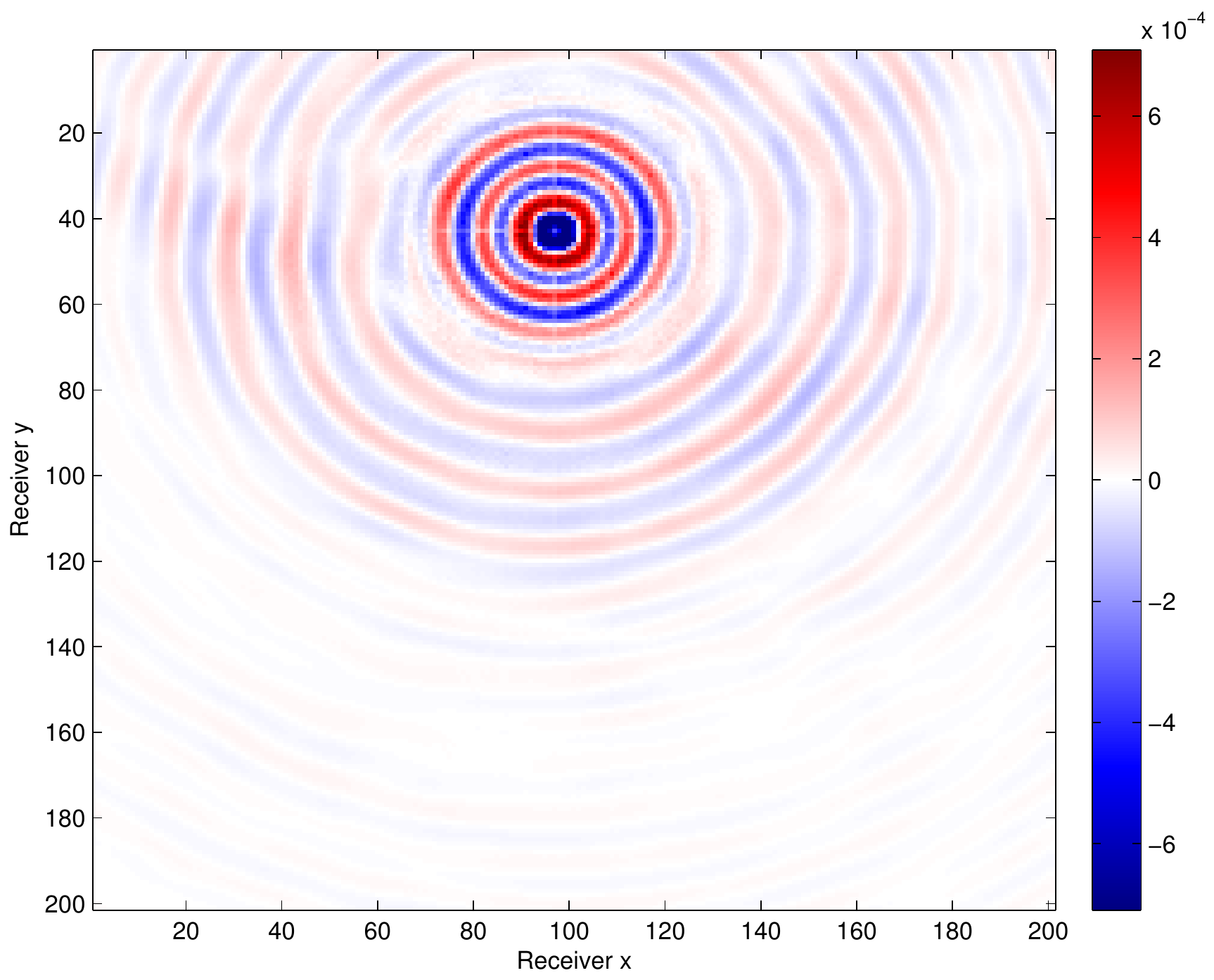}
	}
	\subfloat[Subsampled Data]{
		\includegraphics[scale=0.23]{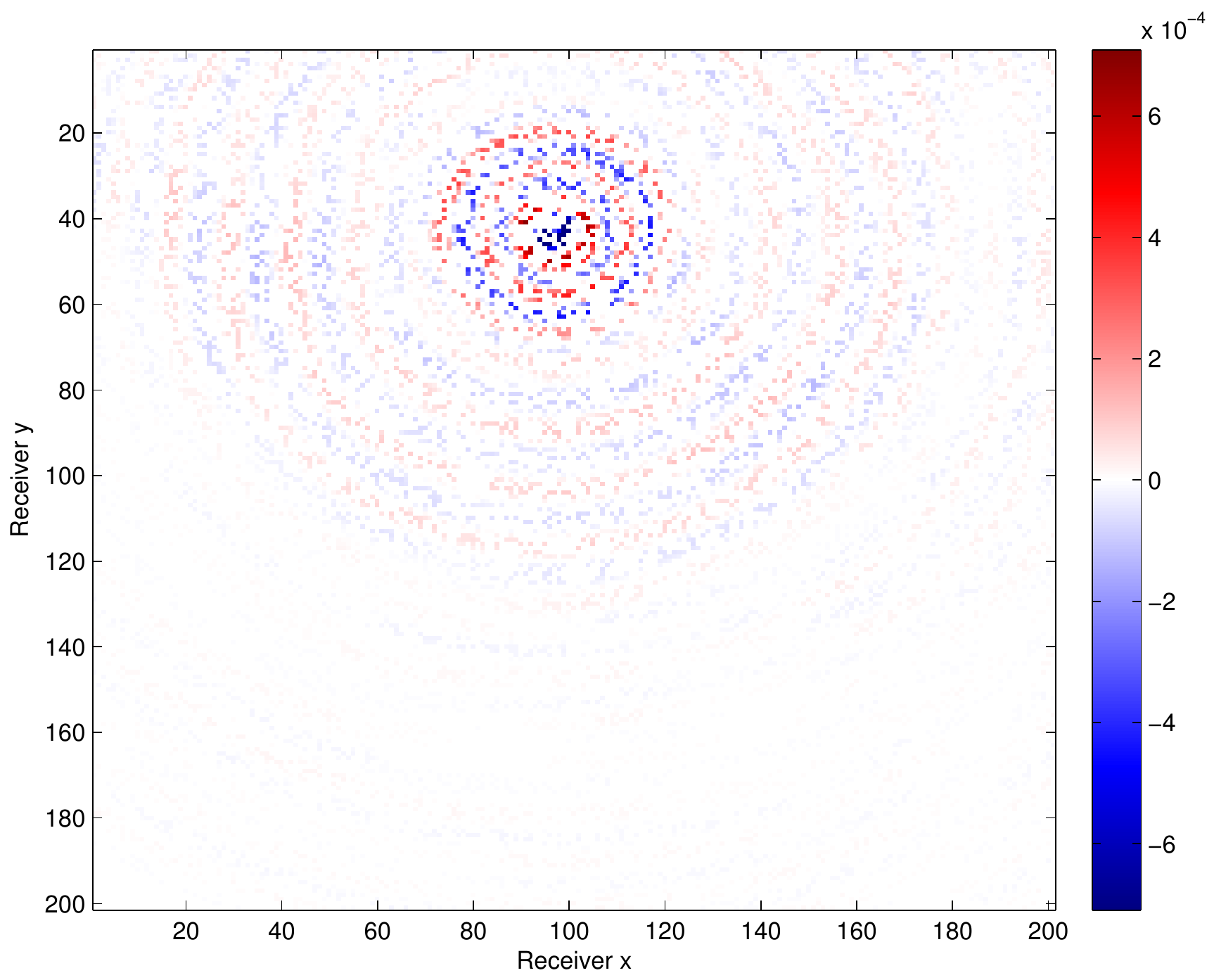}
	} 
	\subfloat[Interpolated Data - SNR 20 dB]{
		\includegraphics[scale=0.23]{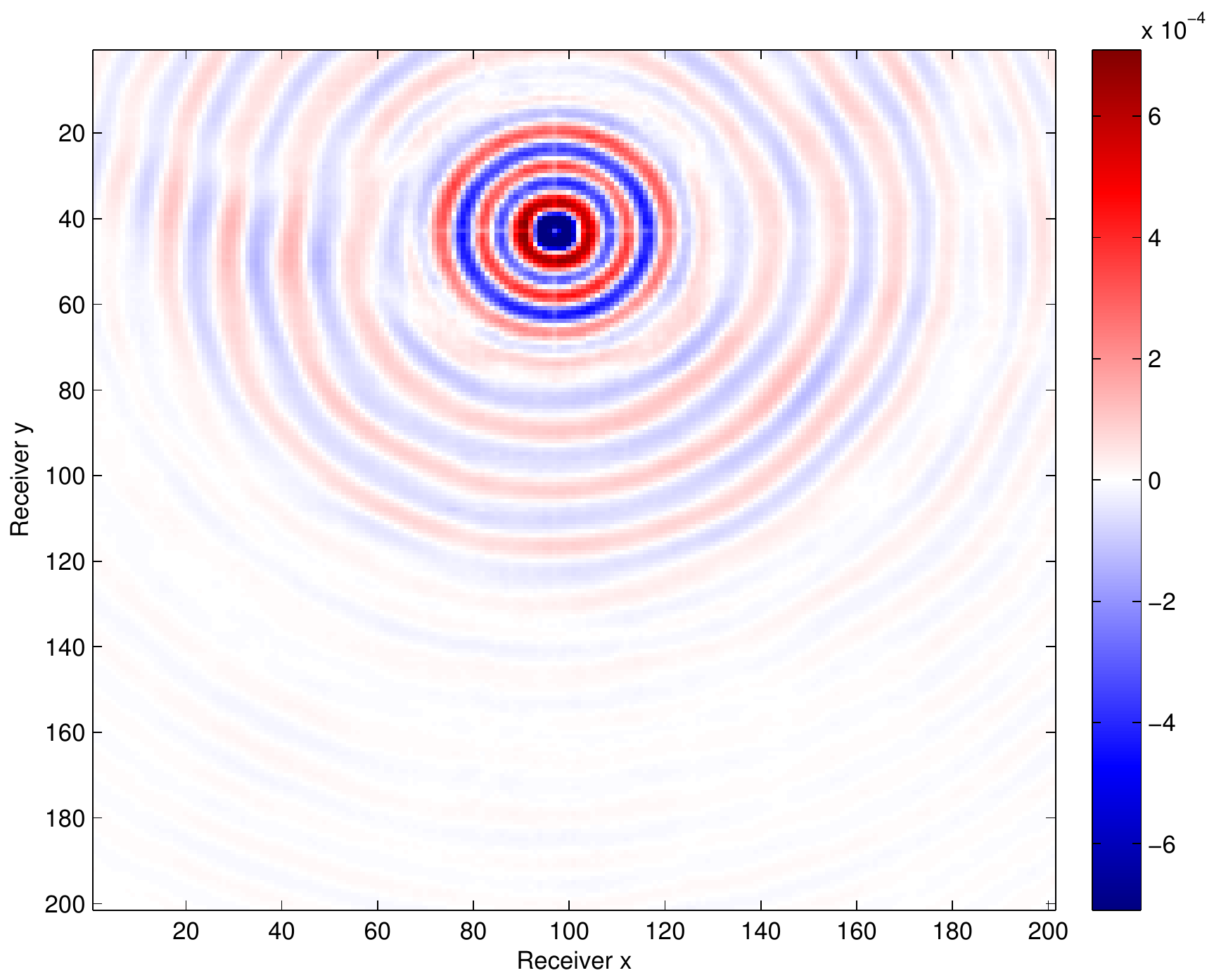}
	}
	\subfloat[Difference]{
		\includegraphics[scale=0.23]{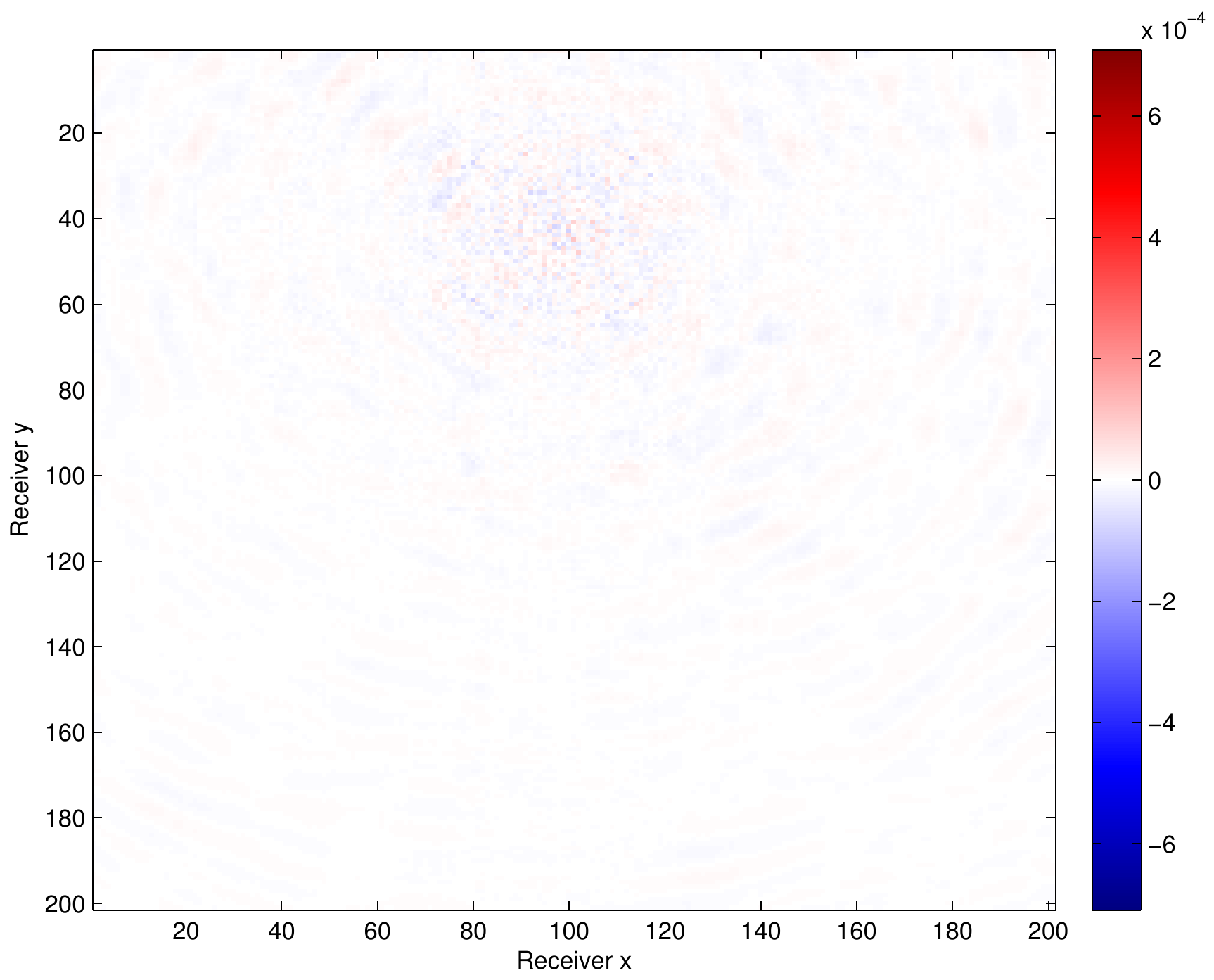}
	} \\
	 \captionsetup[subfigure]{labelformat=empty,position=top}
	\subfloat[True Data]{
		\includegraphics[scale=0.23]{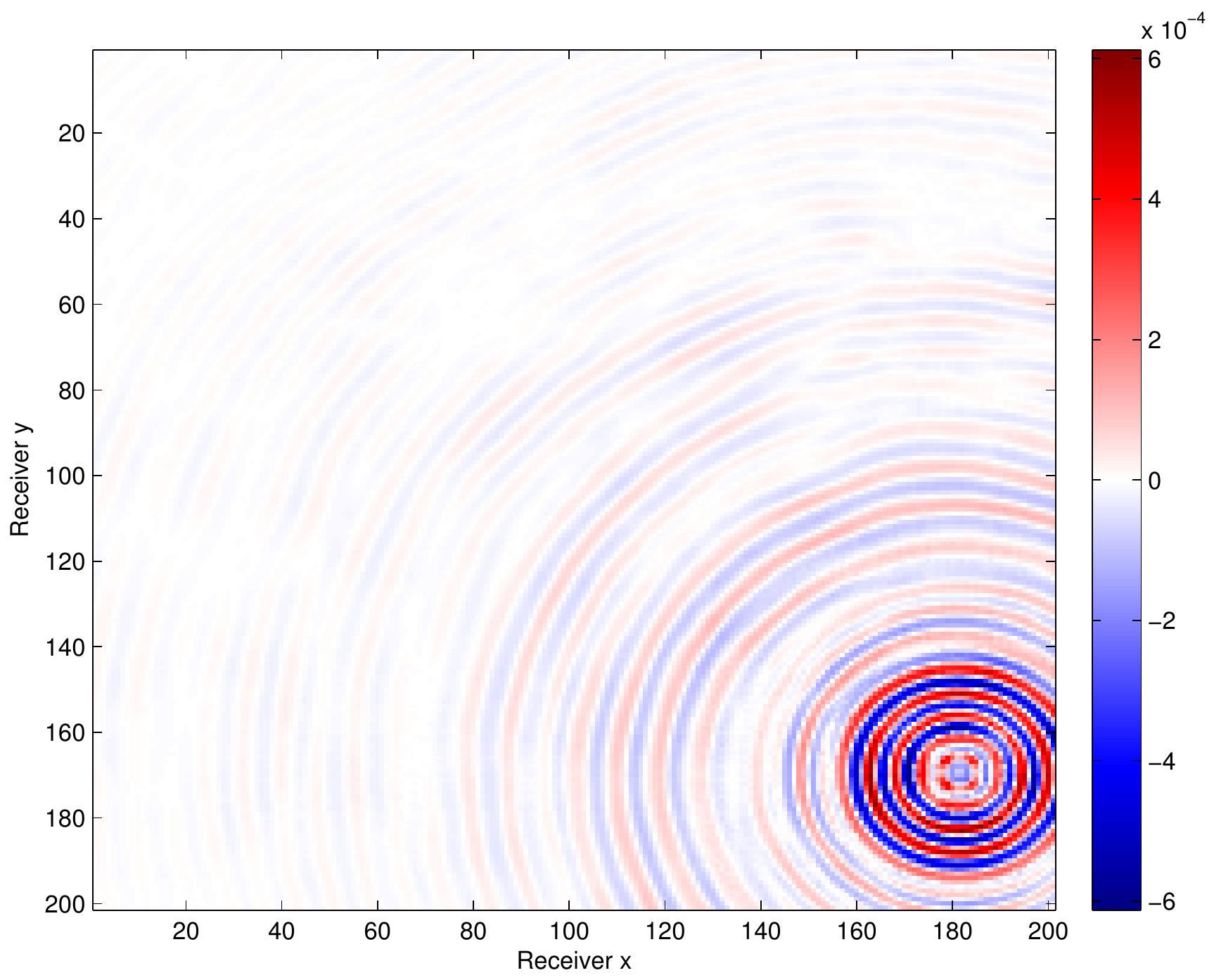}
	}
	\subfloat[Subsampled Data]{
		\includegraphics[scale=0.23]{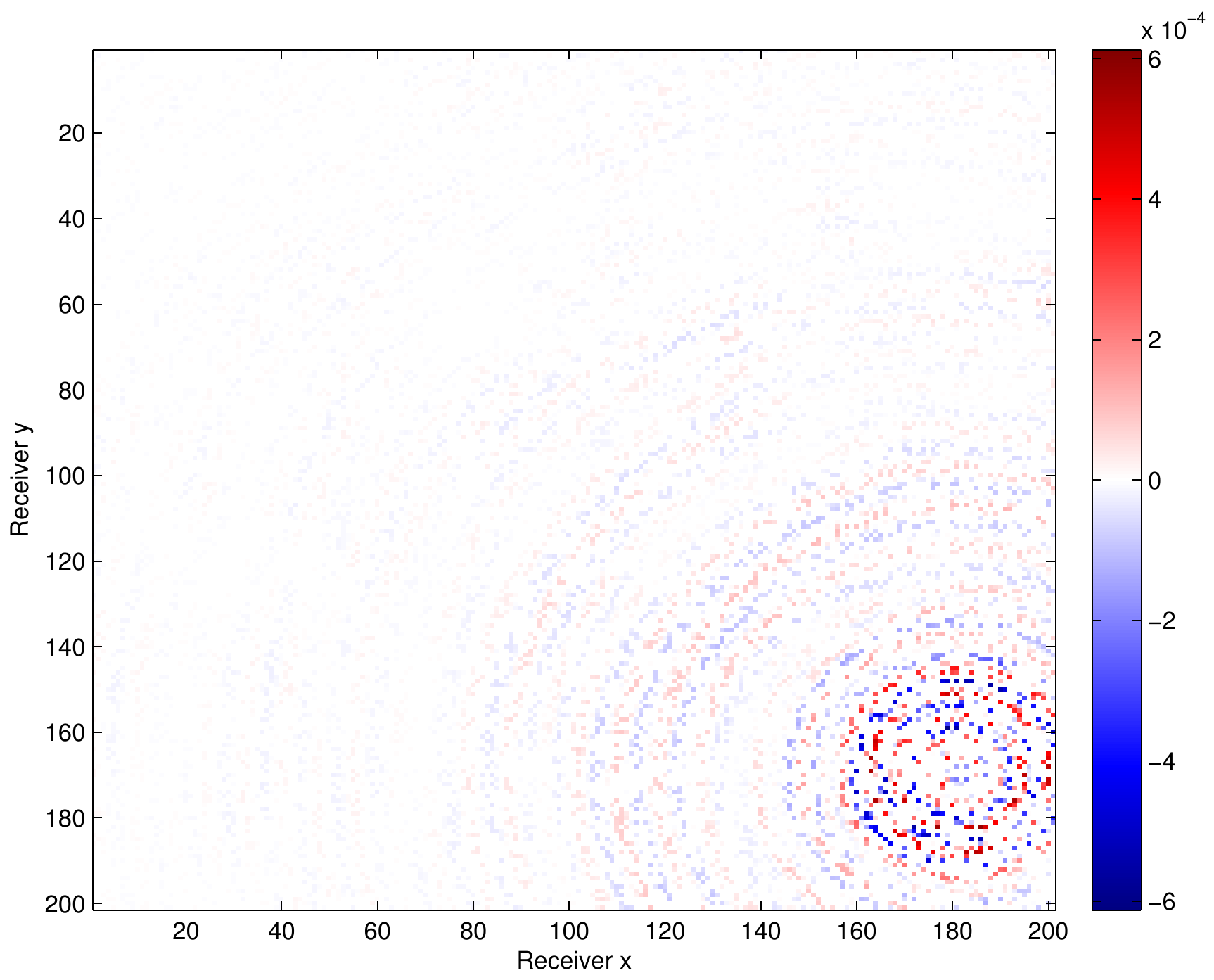}
	} 
	\subfloat[Interpolated Data - SNR 17.7 dB]{
		\includegraphics[scale=0.23]{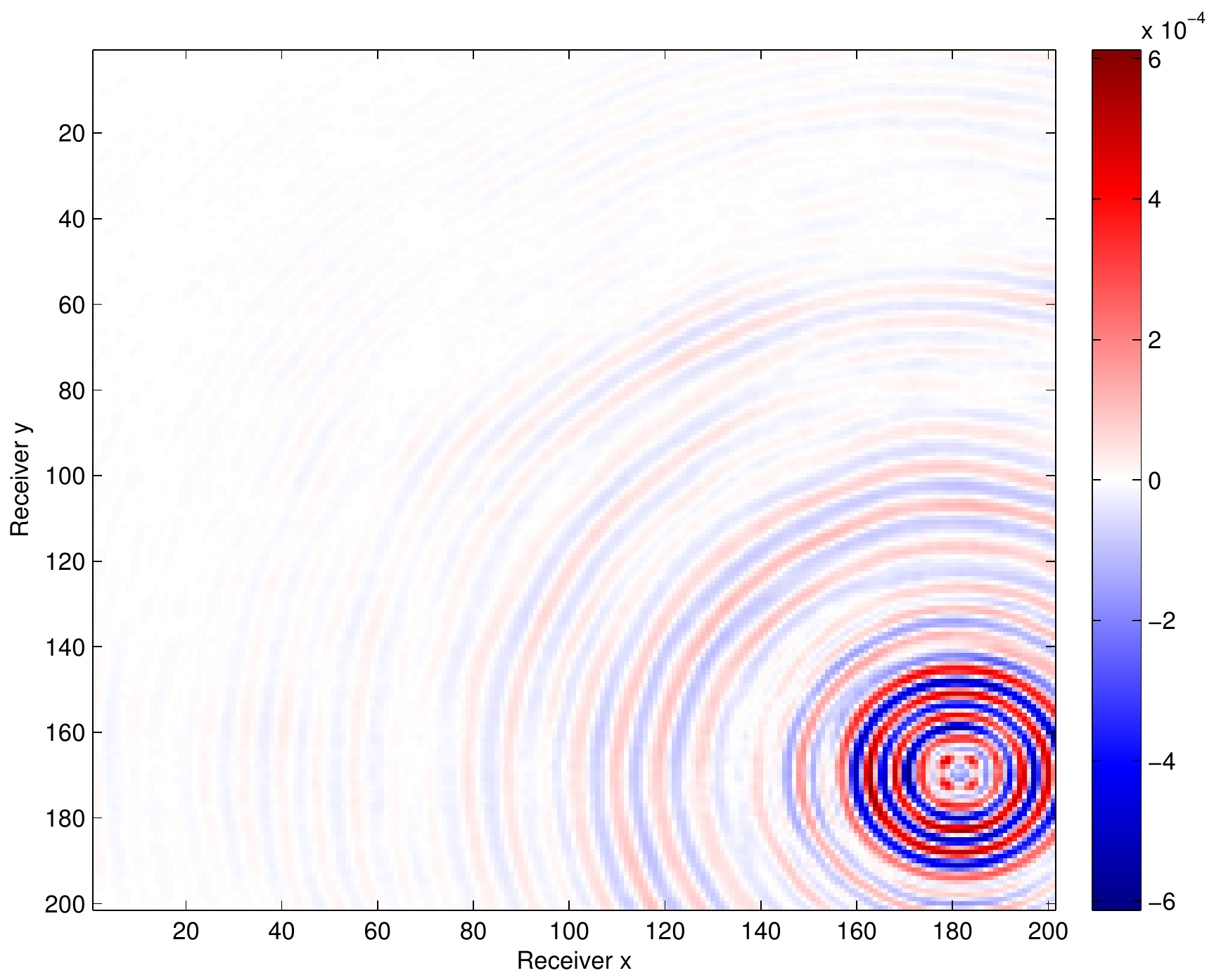}
	}
	\subfloat[Difference]{
		\includegraphics[scale=0.23]{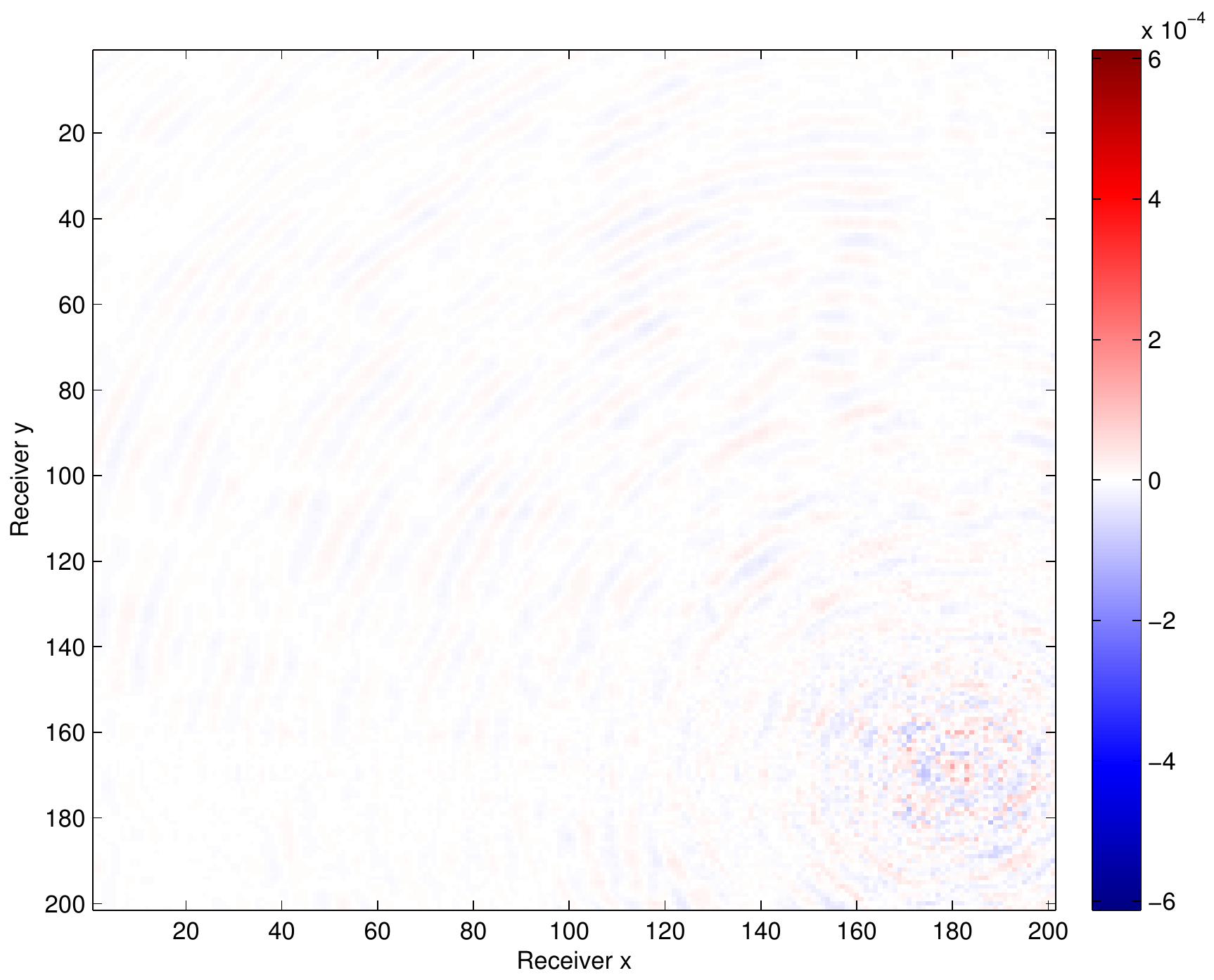}
	}
	\\
	\centering
	\subfloat[True Data]{
		\includegraphics[scale=0.23]{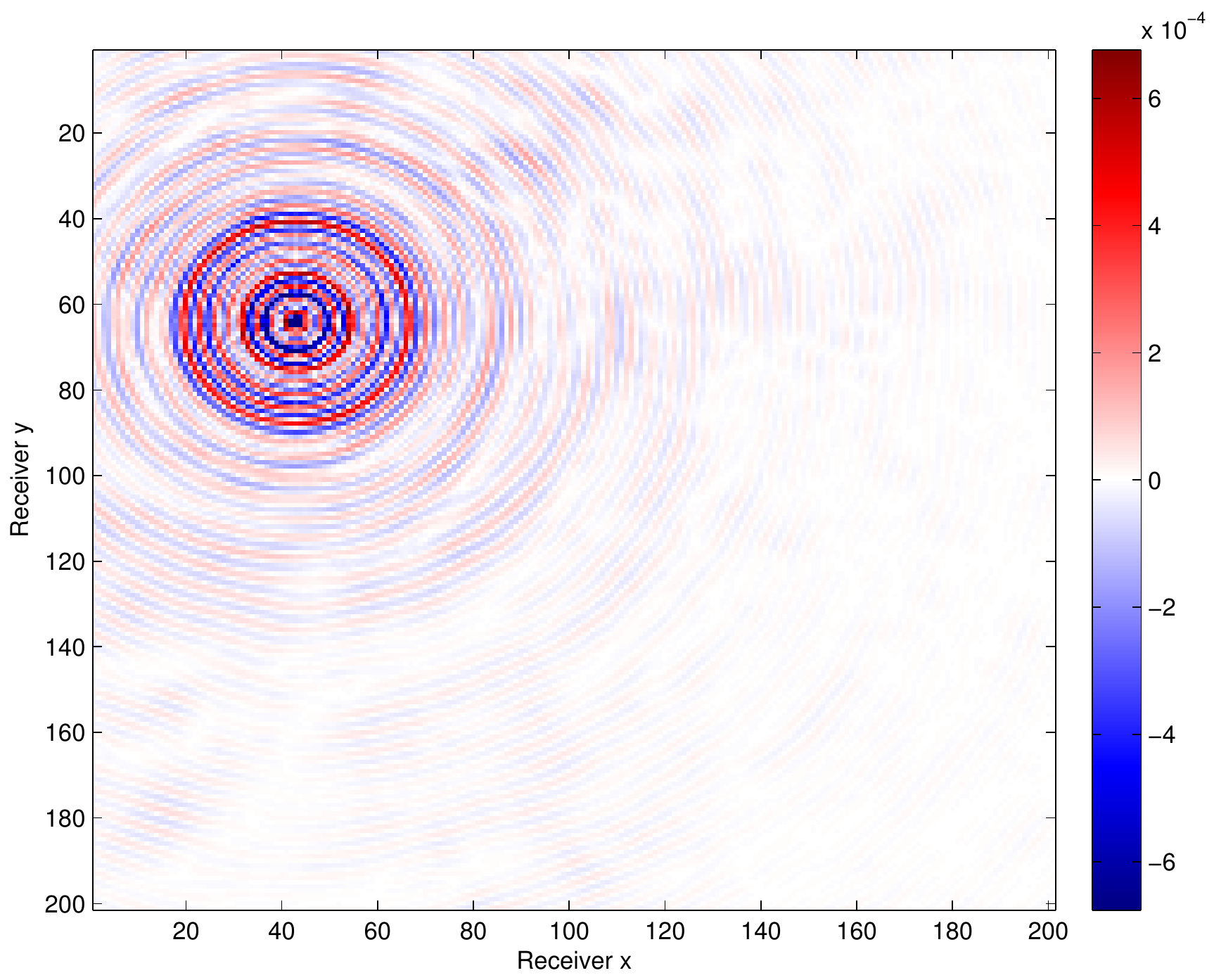}
	}
	\subfloat[Subsampled Data]{
		\includegraphics[scale=0.23]{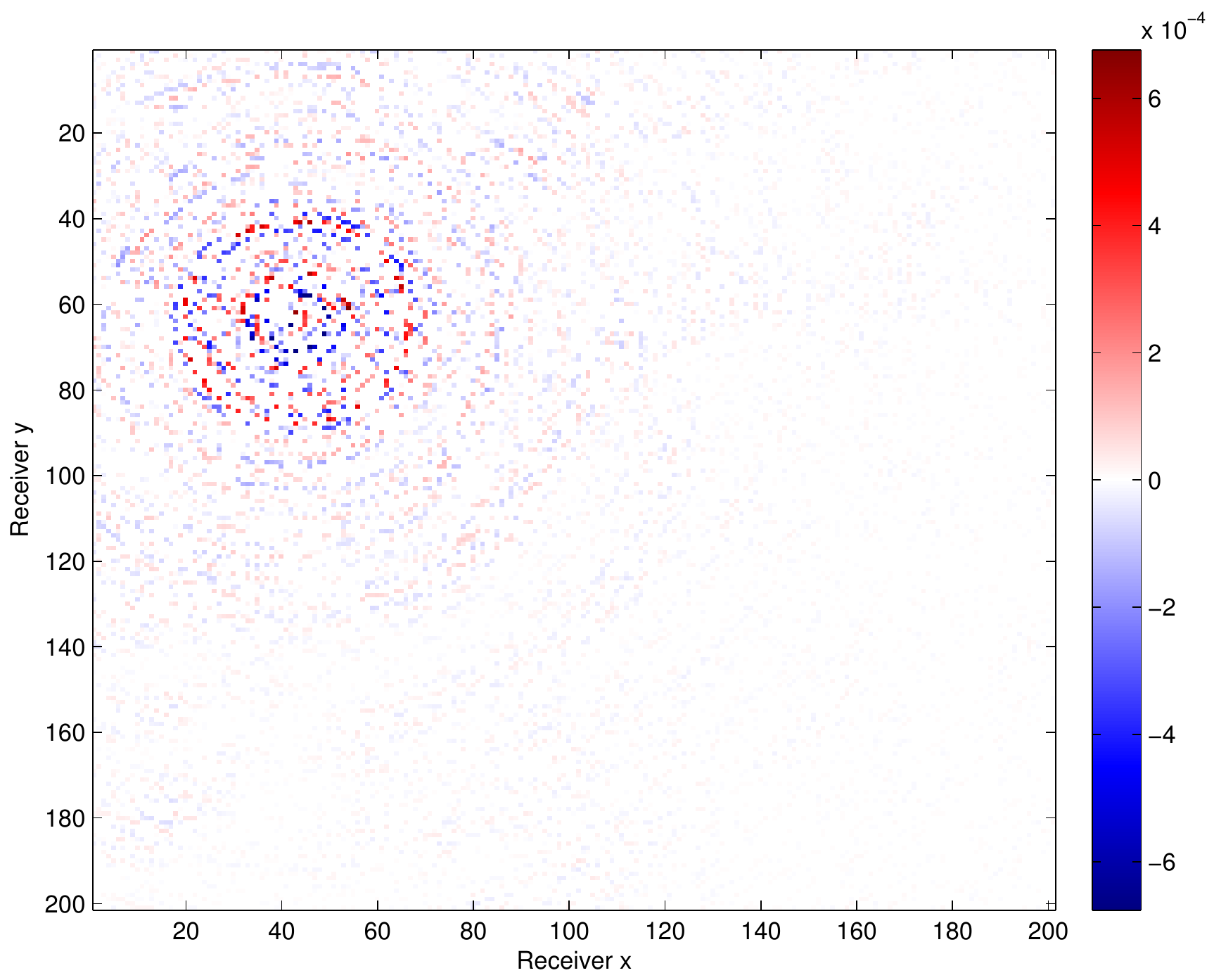}
	} 
	\subfloat[Interpolated Data - SNR 11.2 dB]{
		\includegraphics[scale=0.23]{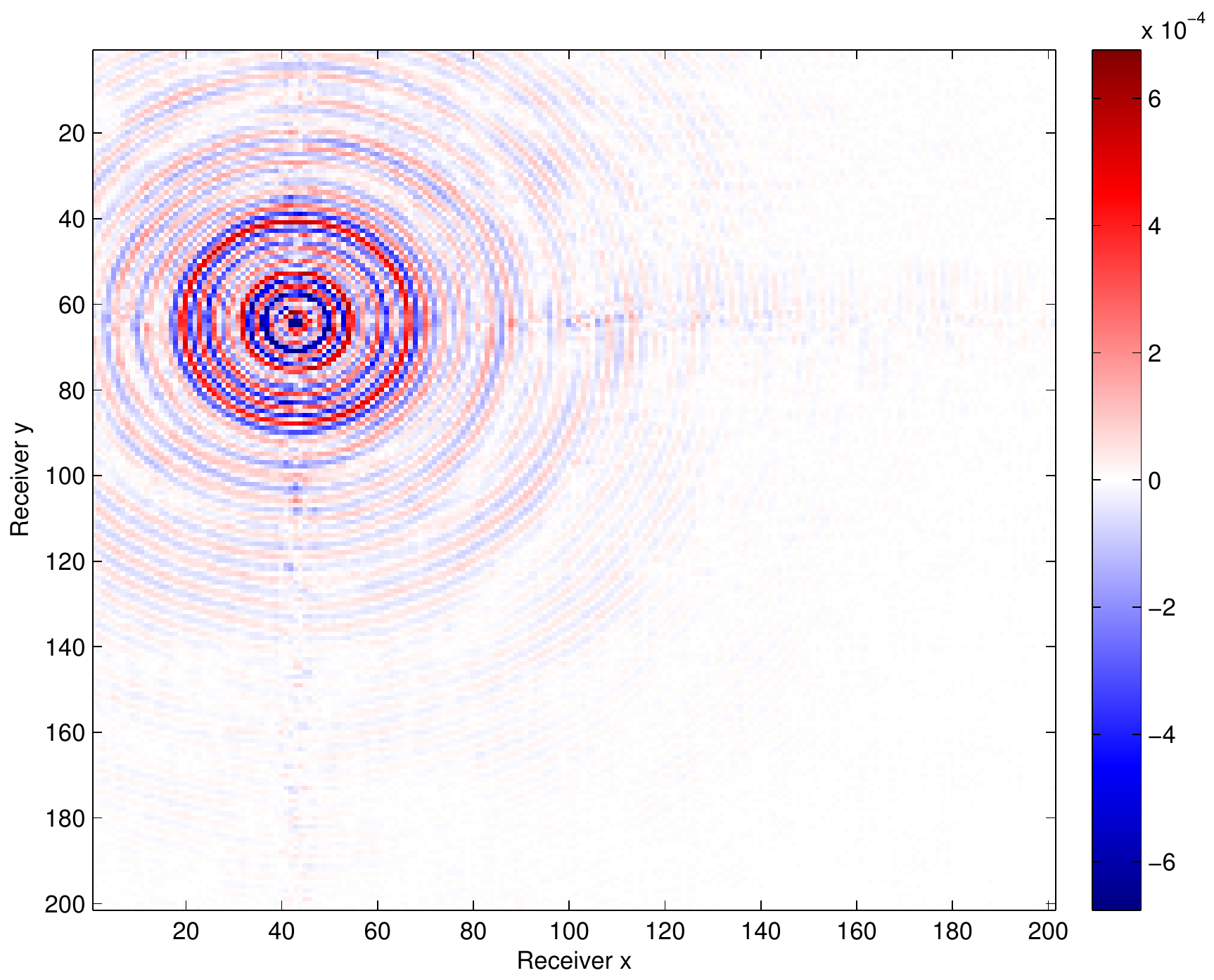}
	}
	\subfloat[Difference]{
		\includegraphics[scale=0.23]{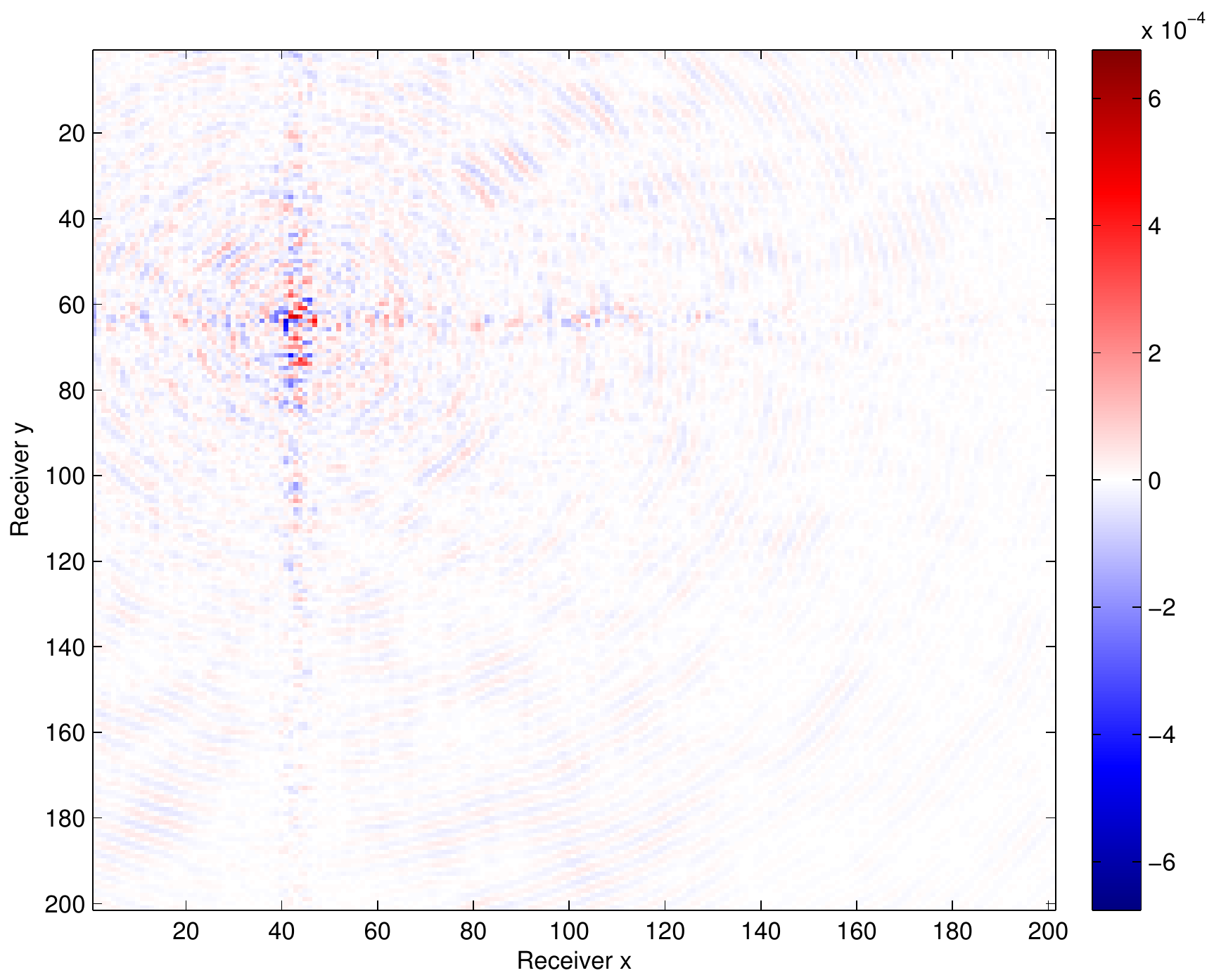}
	}
	\caption{75\% missing receivers, fixed source coordinates. \emph{Top:} 4.68 Hz, \emph{Middle:} 7.34 Hz, \emph{Bottom:} 12.3 Hz. }
	\label{fig:bgdata75missingrecs}
\end{figure}
\clearpage
\begin{figure}
\centering
\captionsetup[subfigure]{labelformat=empty,position=top,width=125pt}
\subfloat[True Data]{
		\includegraphics[scale=0.23]{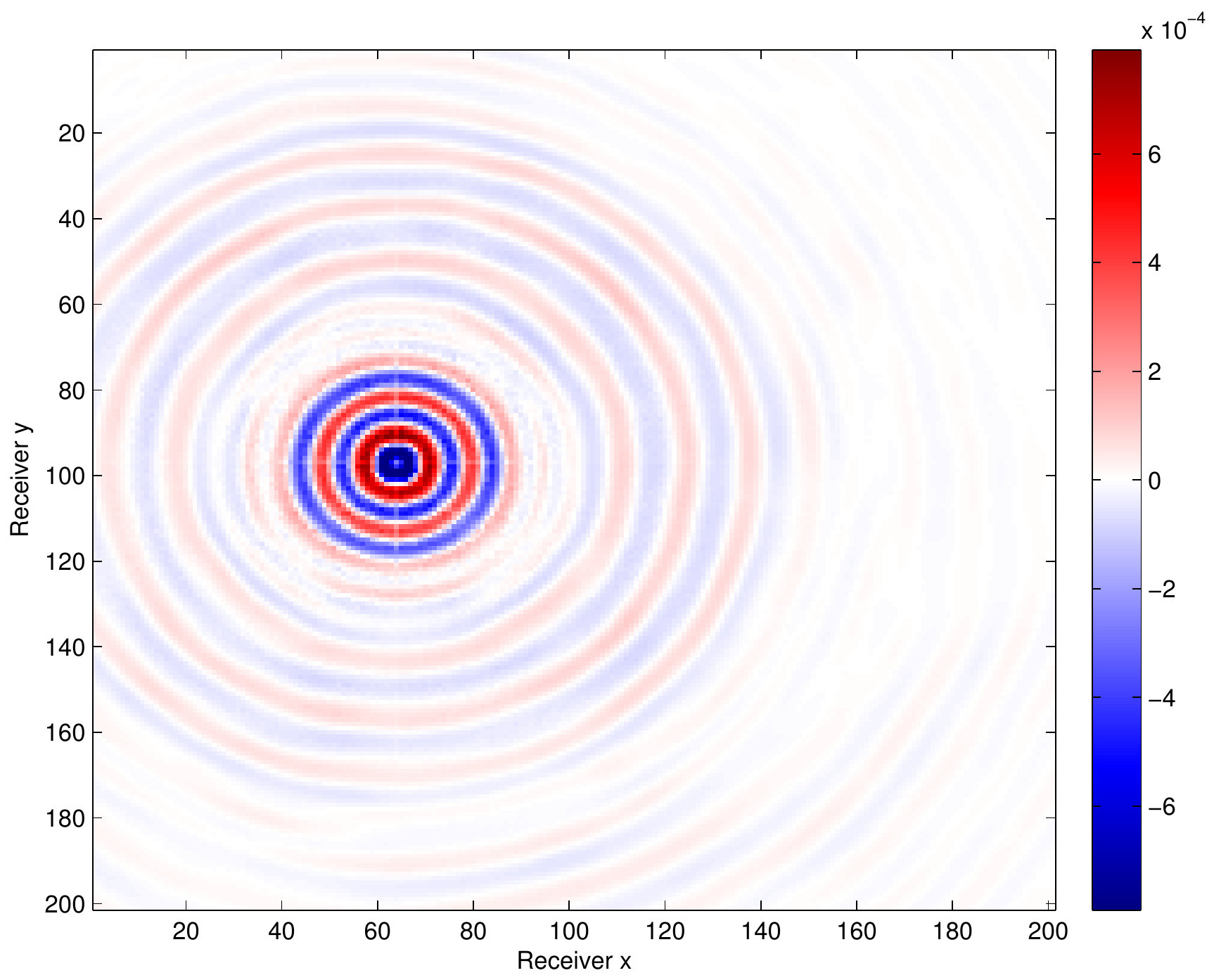}
		\label{fig:regrecovery-truedata}
	}
	\subfloat[Subsampled Data]{
		\includegraphics[scale=0.23]{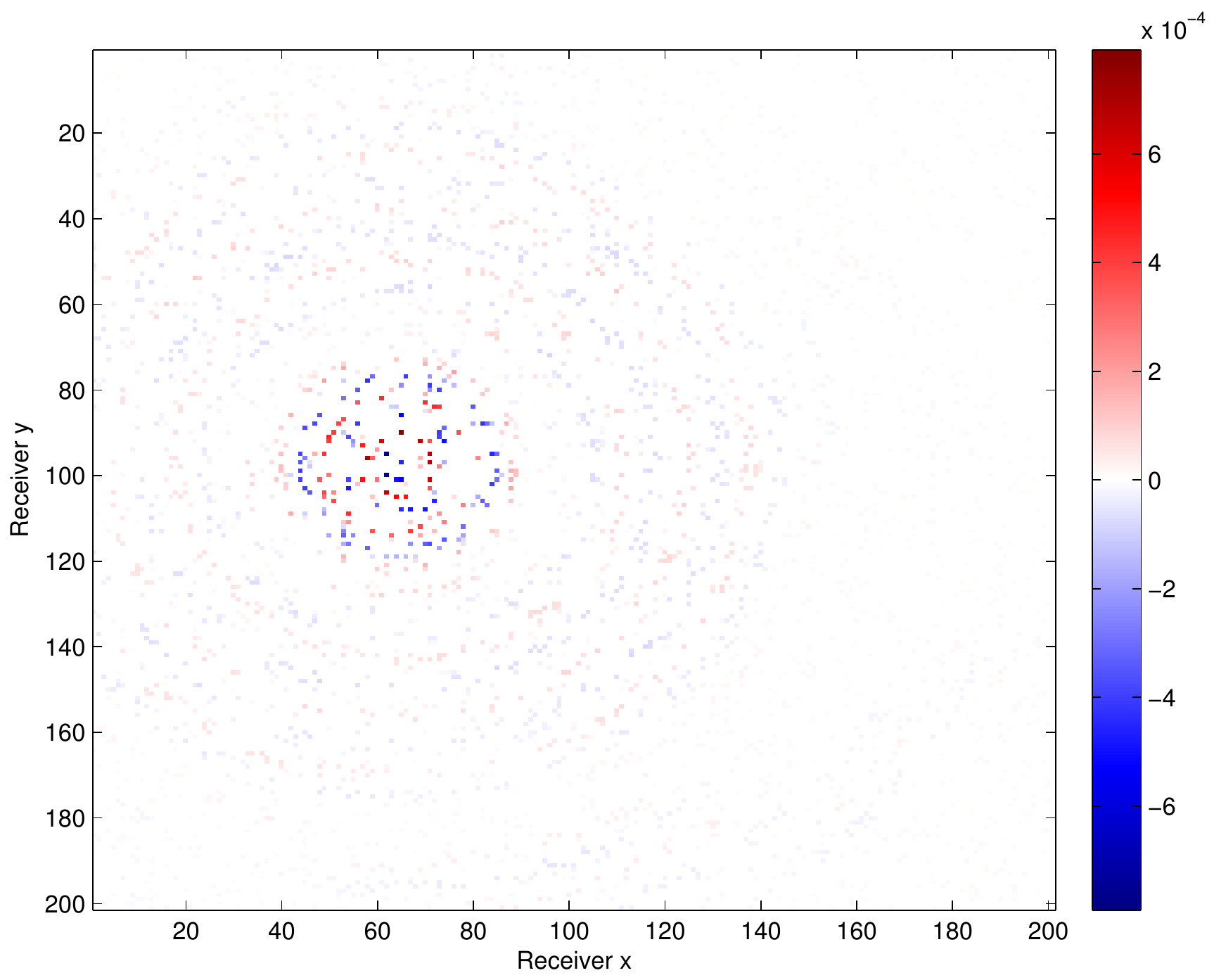}
		\label{fig:regrecovery-subdata}
	}
	\subfloat[No Regularization - SNR 10.4 dB ]{
		\includegraphics[scale=0.23]{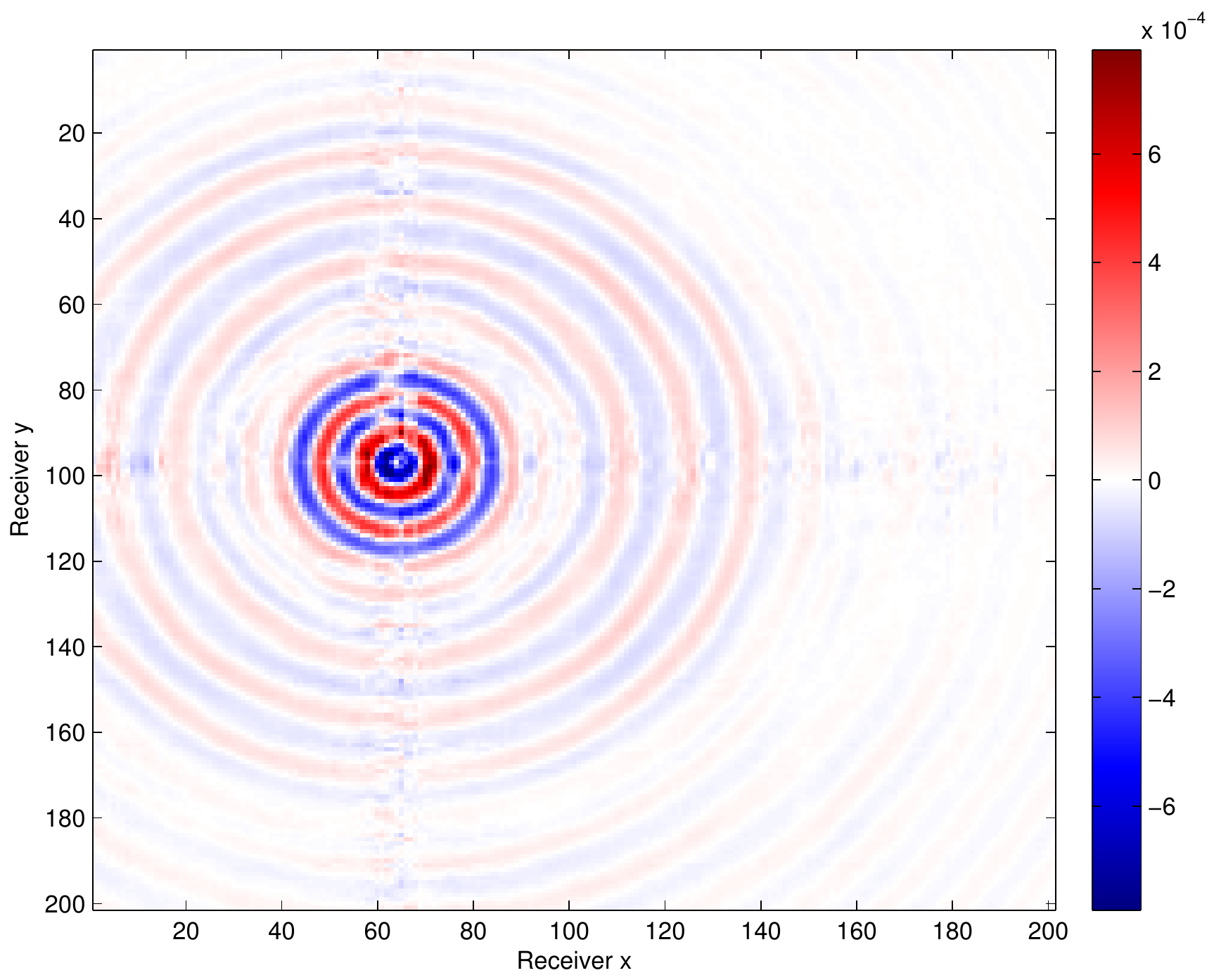}
		\label{fig:regrecovery-noreg}
	}
	\subfloat[Regularization - SNR 13.9 dB]{
		\includegraphics[scale=0.23]{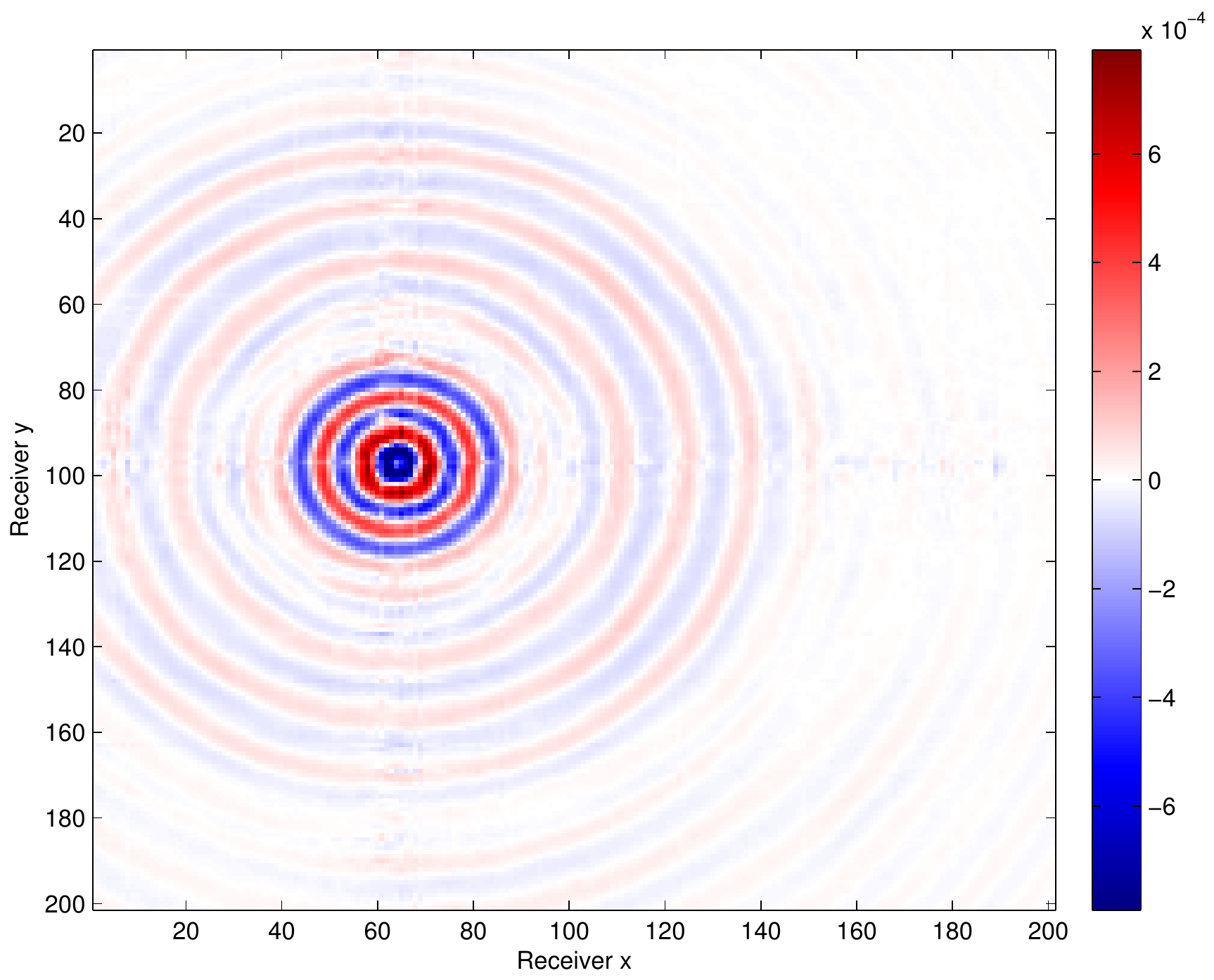}
		\label{fig:regrecovery-reg}
	}
\caption{Regularization reduces some of the spurious artifacts and
  reduces overfitting in the case where there is very little
  data. 4.86 Hz data, 90\% missing receivers. }
\label{fig:regularization}
\end{figure}

%
%
\appendix
\section{Adjoint multilinear operators}
\label{app:adjops}
To derive expressions for the adjoint derivatives, we first consider
the general multilinear product
\begin{align*}
  A_1 \times_1 A_2 \times_2 \dots A_d \times_d \tensor{B} \in \mathbb{R}^{m_1 \times ... \times m_d}
\end{align*}
with $A_i \in \mathbb{R}^{m_i \times n_i}$ and $\tensor{B} \in
\fullspace$. Now, let $P_i$ be the linear
operator that fixes each $A_j$ for $j\neq i$ in the above expression,
i.e.,
\begin{align*}
  P_{i} & : \mathbb{R}^{m_i \times n_i} \to \mathbb{R}^{m_1 \times ... \times m_d} \\
  P_{i}(C) & := A_1 \times_1 A_2 \times_2 \dots A_{i-1} \times_{i-1} C \times_i A_{i+1} \times_{i+1} \dots A_d \times_d \tensor{B }.
\end{align*}
In matricized form, this operator can be written as
\begin{align*}
(P_{i}(C))^{([d] \setminus i)} = C B^{(i)}(A^T_d \otimes A^T_{d-1} \otimes \dots \otimes A^T_{i+1} \otimes A^T_{i-1} \otimes \dots A_1^T) 
\end{align*}

Taking the inner product of the matrix $(P_{i}(C)^{([d] \setminus i)}$ and a tensor $\tensor{Y}$ matricized along the $i$th mode yields
\begin{align*}
\langle P_{i}(C)^{[d]\setminus i}, Y^{(i)} \rangle &= \text{tr}( C B^{(i)}(A^T_d \otimes A^T_{d-1} \otimes \dots \otimes A^T_{i+1} \otimes A^T_{i-1} \otimes \dots A_1^T) (Y^{(i)})^T ) \\
&= \langle C, Z \rangle
\end{align*}
where $Z = Y^{(i)}(A_d \otimes A_{d-1} \otimes \dots \otimes A_{i+1} \otimes A_{i-1} \otimes \dots A_1) (B^{(i)})^T$.

We note that $(A_d \otimes A_{d-1} \otimes \dots \otimes A_{i+1} \otimes A_{i-1} \otimes \dots A_1) (B^{(i)})^T$ is the matricized form of 
\begin{align*}
\tensor{W} = A_1^T \times_1 A_2^T \times_2 \dots A_{i-1}^T \times_{i-1}
  I_{m_i} \times_{i} A_{i+1}^T \times_{i+1} \dots A_d^T \times_d \tensor{B}
\end{align*}
along the modes $[d]\setminus i$, and that $Y^{(i)} W^{([d]\setminus i)}$ is the tensor contraction $\langle \tensor{Y}, \tensor{W} \rangle_{([d]\setminus i), ([d]\setminus i)}$. It follows that the adjoint of the operator $P_{i}(C)$ in the standard Euclidean inner product is given by
\begin{align*}
  P_{i}^*(\tensor{Y}) =\langle A_1^T \times_1 A_2^T \times_2 \dots A_{i-1}^T \times_{i-1}
  I_{m_i} \times_{i} A_{i+1}^T \times_{i+1} \dots A_d^T \times_d \tensor{Y} , \tensor{B}
  \rangle_{([d]\setminus i),([d]\setminus i)}.
\end{align*}
Likewise, for the linear operator
\begin{align*}
  P_{\tensor{B}}(\tensor{C}) := A_1 \times_1 \dots A_d \times_d \tensor{C} 
\end{align*}
we find that its adjoint is given by
\begin{align*} 
  P^*_{\tensor{B}}(\tensor{Y}) = A_1^T \times_1 \dots A_d^T \times_d 
  \tensor{Y}. 
\end{align*}

\section{Proof of \Cref{prop:retraction}}
\label{proof:qrretraction}
\begin{proof}
It is easy to see that the first point in Definition \ref{def:retraction} is satisfied, since for $X \in \text{St}(n,p)$, $\qf(X) = X$

Let $x = (U_t, \tensor{B}_t) \in \paramspaceclean$ and $\eta = (\delta U_t, \tensor{\delta B_t}) \in \mathcal{T}_x \paramspaceclean$. To avoid notational overload, we use the slight abuse of notation that $B_t := B_t^{(1,2)}$ for $t \neq \troot$.

Let $s \in [0,t) \mapsto x(s)$ be a curve in the parameter space $\paramspaceclean$ with $x(0) = x$ and $x'(0) = \eta$ and \\
$x(s) = (U_t(s), B_t(s))$ and $x'(s) = (\delta U_t(s), \delta B_t(s))$.

Then we have that, in Kronecker form,
\begin{align*}
DR_x(0_x)[\eta] = \begin{cases}
\frac{d}{ds}\qf(x(s)_{t}) \big|_{s=0} & \text{if } t \in L \\
\frac{d}{ds}\qf((R_{t_r}(s) \otimes R_{t_l}(s))(x(s)_t) \big|_{s=0} & \text{if } t \not\in \troot \cup L\\
\frac{d}{ds}(R_{t_r}(s) \otimes R_{t_l}(s))(x(s)_t) \big|_{s=0} & \text{if } t = \troot
\end{cases} 
\end{align*}

The fact that $DR_x(0_x)[\eta]_{t} = \delta U_t$ for $t \in L$ follows from Example 8.1.5 in \cite{optmatrixmanifold}.

To compute $DR_x(0_x)[\eta]_{t}$ for $t \not\in L \cup \troot$, we first note the formula from \cite{optmatrixmanifold}
\begin{align}
\label{eq:Dqf}
D\qf(Y)[U] = \qf(Y)\rho_{\text{skew}}(\qf(Y)^T U(\qf(Y)^TY)^{-1}) + (I - \qf(Y) \qf(Y)^T) U(\qf(Y)^TY)^{-1}
\end{align}

where $Y \in \mathbb{R}^{n\times k}_*$, $U \in T_Y \mathbb{R}_*^{n\times k} \simeq \mathbb{R}^{n \times k}$ and $\text{qf}(Y)$ is the Q-factor of the QR-decomposition of $Y$.

Therefore, if we set $Z(s) = (R_{t_r}(s) \otimes R_{t_l}(s))(x(s)_t)$, where $R_t(s)$ is the $R$-factor of the QR-decomposition of the matrix associated to node $t$, we have
\begin{align*}
Z'(0) = [(R'_{t_r}(0) \otimes I_{k_l}) + (I_{k_r} \otimes R'_{t_l}(0)]B_t + \delta B_t
\end{align*}

As a result of the discussion in Example 8.1.5 in \cite{optmatrixmanifold}, since $R_t(0) = I_{k_t}$ we have that
\begin{align*}
R'_t(0) = 
\begin{cases}
\rho_{UT}(U_t^T \delta U_t ) & \text{ for } t \in L \\
\rho_{UT}(B_t^T \delta B_t )  & \text{ for } t \not\in L \cup \troot
\end{cases}
\end{align*}
where $\rho_{UT}(A)$ is the projection onto the upper triangular term of the unique decomposition of a matrix into the sum of a skew-symmetric term and an upper triangular term.

Since $U_t \in \text{St}(n_t, k_t)$ and $B_t \in \text{St}(k_{t_l}k_{t_r},k_t)$, in light of the fact that for $X \in St(n,k)$,
\begin{align*}
T_X St(n,k) = \{ X \Omega + X^{\perp} K : \Omega = - \Omega^T \},
\end{align*}
then $X^T \delta X$ is skew symmetric, for any tangent vector $\delta X$, which implies that  $\rho_{UT}(X^T \delta X)$ is zero.

It follows that $R'_t(0) = 0$ for all $t \in T \setminus \troot$, and therefore
\begin{align*}
Z'(0) = \delta B_t
\end{align*}
from which we immediately obtain 
\begin{align*}
DR_x(0_x)[\eta]_t = \delta B_t \quad \text{for } t \not\in L \cup \troot
\end{align*}

A similar approach holds when $t = \troot$, and therefore, $R_x(\eta)$ is a retraction on $\paramspaceclean$.
\end{proof}

\section{Square-root based retraction}
\label{app:sqrt}
Another straightforward projection onto the orthonormal parameter
space is immediate from the remark that for a general full-rank $n
\times p$ matrix $X$, with $n > p$, the matrix $X (X^T X)^{-1/2}$ is
an orthonormal basis for the column space of $X$. In
\Cref{alg:sqrtretraction}, we only need to compute the eigenvalue
decomposition of a $k_t \times k_t$ matrix, which may be done more
efficiently than computing the QR-factorization of a $k_{t_l} k_{t_r}
\times k_t$ matrix in some instances. 

\begin{algorithm}[H]
\caption{Square-root-based orthogonalization}
\label{alg:sqrtretraction}
\begin{algorithmic}
\REQUIRE $x = (U_t, \tensor{B}_t)$ unorthogonalized
\FOR{ $t \in L$ }
\STATE $M_t = U_t^T U_t$
\STATE Compute the eigenvalue decomposition of $M_t$, $M_t = V_t D_t V_t^T$
\STATE  $U'_t \gets U_t M_t^{-1/2} $
\ENDFOR
\FOR{ $t \in T \setminus L$, visiting children before their parents }
\STATE $C_t  \gets (M_{t_l}^{1/2} \times_1 M_{t_r}^{1/2} \times_2 \tensor{B}_t)^{(1,2)} $
\IF{$t = \troot$}
\STATE $B'_t \gets (C_t)_{(1,2)}$
\ELSE
\STATE Compute the eigenvalue decomposition of $C_t^T C_t$, $C_t^T C_t = V_t D_t V_t^T$
\STATE $B'_t \gets (C_t (C_t^TC_t)^{-1/2})_{(1,2)}$
\ENDIF
\ENDFOR
\RETURN $x' = (U'_t,B'_t)$ in OHT
\end{algorithmic}
\end{algorithm}

\bibliographystyle{plainnat}
\bibliography{Paper}

\end{document}